\documentclass[12pt]{article}

\usepackage{amsmath,amssymb,amsfonts,amsthm}
\usepackage{ifpdf}
\usepackage{enumerate}
\usepackage{ulem}
\usepackage{leftidx}

\usepackage{breqn}

\usepackage{tikz}
\usetikzlibrary{calc}
\usepackage{tikz-qtree}
\usetikzlibrary{decorations.pathmorphing}

\ifpdf
\usepackage[pdftex,plainpages=false,hypertexnames=false,pdfpagelabels]{hyperref}
\else
\usepackage[dvips,plainpages=false,hypertexnames=false]{hyperref}
\fi

\newcommand{\arxiv}[1]{\href{http://arxiv.org/abs/#1}{\tt arXiv:\nolinkurl{#1}}}
\newcommand{\doi}[1]{\href{http://dx.doi.org/#1}{{\tt DOI:#1}}}

\newcommand{\googlebooks}[1]{(preview at \href{http://books.google.com/books?id=#1}{google books})}

\theoremstyle{plain}
\newtheorem{prop}{Proposition}[section]
\newtheorem{conj}[prop]{Conjecture}
\newtheorem{thm}[prop]{Theorem}

\newtheorem{lem}[prop]{Lemma}

\newtheorem{cor}[prop]{Corollary}
\newtheorem*{cor*}{Corollary}

\numberwithin{equation}{section}

\theoremstyle{remark}

\newtheorem{remark}[prop]{Remark}           
\newtheorem*{rem*}{Remark}               
\newtheorem*{ex*}{Example}                

\theoremstyle{definition}
\newtheorem{defn}[prop]{Definition}         
   
\newtheorem{nota}[prop]{Notation}   
\newtheorem*{defn*}{Definition}             

\theoremstyle{plain}

\newcounter{comment}
\newcommand{\noop}[1]{}

\def\clap#1{\hbox to 0pt{\hss#1\hss}}

\newcommand{\Natural}{\mathbb N}
\newcommand{\Integer}{\mathbb Z}

\newcommand{\Complex}{\mathbb C}

\def\semicolon{;}
\def\applytolist#1{
    \expandafter\def\csname multi#1\endcsname##1{
        \def\multiack{##1}\ifx\multiack\semicolon
            \def\next{\relax}
        \else
            \csname #1\endcsname{##1}
            \def\next{\csname multi#1\endcsname}
        \fi
        \next}
    \csname multi#1\endcsname}

\def\calc#1{\expandafter\def\csname c#1\endcsname{{\mathcal #1}}}
\applytolist{calc}QWERTYUIOPLKJHGFDSAZXCVBNM;
\def\bbc#1{\expandafter\def\csname bb#1\endcsname{{\mathbb #1}}}
\applytolist{bbc}QWERTYUIOPLKJHGFDSAZXCVBNM;
\def\bfc#1{\expandafter\def\csname bf#1\endcsname{{\mathbf #1}}}
\applytolist{bfc}QWERTYUIOPLKJHGFDSAZXCVBNM;

\newcommand{\id}{\boldsymbol{1}}
\renewcommand{\imath}{\mathfrak{i}}
\renewcommand{\jmath}{\mathfrak{j}}

\newcommand{\iso}{\cong}

\newcommand{\set}[1]{\left\{#1\right\}}

\makeatletter
\newcommand{\hashdef}[2]{\@namedef{#1}{#2}}
\newcommand{\hashlookup}[1]{\@nameuse{#1}}
\makeatother

\IfFileExists{../../graphs/lookup.tex}%
	{\newcommand{\pathtographs}{../../graphs/}}%
	{\newcommand{\pathtographs}{diagrams/graphs/}}

\hashdef{bwd1v1p1p1v0x1x0p0x1x0p0x1x1p0x0x1p0x0x1v1x1x1x0x0p1x1x1x0x0duals1v1x2x3v1x2}{5847414DC9499792}%
\hashdef{bwd1v1p1p1v1x0x1p0x1x1p0x0x1duals1v1x2x3}{9DBA93F4B54E2B1D}%
\hashdef{bwd1v1p1p1v1x1x0duals1v1x2x3}{53D091157B8A128E}%
\hashdef{bwd1v1p1p1v1x1x0duals1v2x1x3}{7831E0D3AB5CF21B}%
\hashdef{bwd1v1p1p1v1x1x0p0x1x0p0x1x1duals1v1x2x3}{C1985031C3F3AF97}%
\hashdef{bwd1v1p1v0x1p1x0p0x1p0x1v0x0x1x1p0x1x0x1p0x0x0x1v0x1x0duals1v1x2v2x1x3}{FBAB6DC5377798A5}%
\hashdef{bwd1v1p1v1x0p0x1v1x0p1x0p0x1v1x0x0p0x0x1v1x0p1x0p1x0p0x1p0x1v0x1x0x1x0v1duals1v1x2v1x3x2v1x2x4x3x5v1}{E062CBA9D4D0BB85}%
\hashdef{bwd1v1p1v1x0p0x1v1x0p1x0p1x0p0x1v1x0x0x1v1duals1v1x2v1x2x4x3v1}{DBDF79CAF0076A8E}%
\hashdef{bwd1v1p1v1x0p1x0p1x0p0x1v0x1x1x0p0x0x1x0p0x0x1x1v0x0x1duals1v1x2v3x2x1}{8C42AF4785087024}%
\hashdef{bwd1v1p1v1x0p1x0v1x0p1x0v1x1duals1v1x2v1x2}{8F007BF50C3E09AC}%
\hashdef{bwd1v1p1v1x0p1x0v1x0p1x0v1x1duals1v1x2v2x1}{45EDB457CC6AA7E8}%
\hashdef{bwd1v1p1v1x0p1x0v1x0v1p1v1x0p1x0v1x1duals1v1x2v1v1x2}{21FA8EFC6770814D}%
\hashdef{bwd1v1p1v1x0p1x0v1x0v1p1v1x0p1x0v1x1duals1v1x2v1v2x1}{3D83DA1B203A0BFF}%
\hashdef{bwd1v1v1p1v1x0p1x0p0x1p0x1v0x1x1x0duals1v1v1x3x2x4}{CC3B89144EA94D74}%
\hashdef{bwd1v1v1p1v1x1v1v1duals1v1v1v1}{A207DA908F977FE3}%
\hashdef{bwd1v1v1v1p1p1v1x0x0p0x1x0p0x0x1v1x0x0p0x1x0p0x0x1duals1v1v1x2x3v1x3x2}{B5425B05363851E6}%
\hashdef{bwd1v1v1v1p1p1v1x0x0p0x1x0p0x1x0v1x0x0duals1v1v1x2x3v1}{DB2C121504201C9D}%
\hashdef{fg1v0v1p1p1v1x2x3v1p2x1p1x1x0}{67D61D886E8C4EF3}%
\hashdef{fg1v0v1p1p1v1x2x3v1p2x3p1x2x1}{1D334DBBF31B9D77}%
\hashdef{fg1v0v1p1v1x2v1p1x1v1x0p1x0p0x1v1x3x2v1p1x0p0x0x1v1x0x0p1x0x0p1x0x0p0x0x1p0x0x1v1x2x4x3x5v0p1x1p1x1x0p0x1x0x1p0x0x0x1x0v0x1x0x1x0v1v0}{0CCEB5C76601FE00}%
\hashdef{fg1v0v1p1v1x2v1p1x1v1x0p1x0p1x0p0x1v1x2x4x3v1p1x0p1x1x0p1x0x0x1v1x0x0x1v1v0}{3A8F4C7507FAE40C}%
\hashdef{fg1v0v1p1v1x2v3p1x1v2x0p1x0p1x1v3x2x1v1p1x0p1x1x2}{16FDED80F04FA0F3}%

\newcommand{\bigraph}[1]{{\hspace{-3pt}\begin{array}{c}%
  \raisebox{-2.5pt}{\includegraphics[height=6mm]{\pathtographs \hashlookup{#1}}}%
\end{array}\hspace{-3pt}}}

\newcommand{\graph}[1]{{\hspace{-3pt}\begin{array}{c}%
  \raisebox{-2.5pt}{\includegraphics[height=10mm]{\pathtographs \hashlookup{#1}}}%
\end{array}\hspace{-3pt}}}

\newcommand{\tensor}{\otimes}

\newcommand{\Mat}[1]{\operatorname{\mathbf{Mat}}\left(#1\right)}

\newcommand{\End}[1]{\operatorname{End}\left(#1\right)}

\usepackage{xifthen}  
\usepackage{breqn}
\usepackage{yfonts}
\usepackage{ulem}

\tikzstyle{shaded}=[fill=rho!10!theta!20!gray!30!white]
\tikzstyle{unshaded}=[fill=white]
\tikzstyle{empty box}=[circle, draw, thick, fill=white, opaque, inner sep=2mm]
\tikzstyle{annular}=[scale=.7, inner sep=1mm, baseline]
\tikzstyle{rectangular}=[scale=.75, inner sep=1mm, baseline=-.1cm]

\usetikzlibrary{shapes}
\usetikzlibrary{backgrounds}
\usetikzlibrary{decorations,decorations.pathreplacing,decorations.markings}
\usetikzlibrary{fit,calc,through}
\usetikzlibrary{external}

\tikzstyle{mid>}=[decoration={markings, mark=at position 0.5 with {\arrow{>}}}, postaction={decorate}]
\tikzstyle{mid<}=[decoration={markings, mark=at position 0.5 with {\arrow{<}}}, postaction={decorate}]
\tikzstyle{upper>}=[decoration={markings, mark=at position 0.8 with {\arrow{>}}}, postaction={decorate}]
\tikzstyle{upper<}=[decoration={markings, mark=at position 0.8 with {\arrow{<}}}, postaction={decorate}]
\tikzstyle{lower>}=[decoration={markings, mark=at position 0.2 with {\arrow{>}}}, postaction={decorate}]
\tikzstyle{lower<}=[decoration={markings, mark=at position 0.2 with {\arrow{<}}}, postaction={decorate}]

\usepackage{xcolor}
\definecolor{dark-red}{rgb}{0.7,0.25,0.25}
\definecolor{dark-blue}{rgb}{0.15,0.15,0.55}
\definecolor{medium-blue}{rgb}{0,0,0.65}
\definecolor{A3}{RGB}{0,150,0}
\definecolor{A4}{rgb}{1,0.5,0}
\definecolor{zeta}{named}{violet}
\definecolor{alpha}{RGB}{0,150,0}
\definecolor{theta}{named}{blue}
\definecolor{rho}{named}{red}
\definecolor{mu}{named}{orange}
\hypersetup{
   colorlinks, linkcolor={purple},
   citecolor={medium-blue}, urlcolor={medium-blue}
}

\DeclareMathOperator{\spann}{span}
\DeclareMathOperator{\Out}{Out}

\let\Mat\undefined
\let\End\undefined
\DeclareMathOperator{\End}{End}
\DeclareMathOperator{\Mat}{Mat}

\DeclareMathOperator{\train}{train}
\DeclareMathOperator{\wheel}{wheel}
\DeclareMathOperator{\twocar}{twocar}
\DeclareMathOperator{\engine}{engine}
\DeclareMathOperator{\caboose}{caboose}
\newcommand{\jw}[1]{f^{(#1)}}

\renewcommand{\set}[2]{\left\{#1\middle|#2\right\}}

\usepackage{longtable}

\newcommand{\fish}{\cB\cH\cF}
\renewcommand{\AA}{\cA\cA}
\newcommand{\AT}{\cA\cT}
\newcommand{\TT}{\cT\cT}
\newcommand{\Afour}[1]{\textcolor{A4}{#1}}
\newcommand{\Athree}[1]{\textcolor{A3}{#1}}

\newcommand{\Ucircle}[2]{
	\draw[thick, unshaded] #2 circle (.4cm);
	\node at #2 {$#1$};
	\node at ($#2+(0,.55)$) {$\star$};
}
\newcommand{\largeUcircle}[2]{
	\draw[thick, unshaded] #2 circle (.4cm);
	\node at #2 {\scriptsize{{$#1$}}};
	\node at ($#2+(0,.55)$) {$\star$};
}
\newcommand{\UcircleStar}[2]{
	\draw[thick, unshaded] #2 circle (.4cm);
	\node at #2 {$#1$};
	\node at ($#2+(0,-.55)$) {$\star$};
}
\newcommand{\largeUcircleStar}[2]{
	\draw[thick, unshaded] #2 circle (.4cm);
	\node at #2 {\scriptsize{{$#1$}}};
	\node at ($#2+(0,-.55)$) {$\star$};
}
\newcommand{\nbox}[5]{
	\draw[thick, #1] ($#2+(-.4,-.4)+(-#3,0)$) -- ($#2+(-.4,.4)+(-#3,0)$) -- ($#2+(.4,.4)+(#4,0)$) -- ($#2+(.4,-.4)+(#4,0)$) -- ($#2+(-.4,-.4)+(-#3,0)$); 
	\coordinate (a) at ($#2+(-#3,0)$);
	\coordinate (b) at ($#2+(#4,0)$);
	\node at ($1/2*(a)+1/2*(b)$) {$#5$};
}
\newcommand{\zetaUpArrow}[1]{
	\draw[thick, zeta] ($#1+(-.1,-.05)$) -- ($#1+(0,.05)$) -- ($#1+(.1,-.05)$);
}
\newcommand{\zetaDownArrow}[1]{
	\draw[thick, zeta] ($#1+(-.1,.05)$) -- ($#1+(0,-.05)$) -- ($#1+(.1,.05)$);
}

\newcommand{\noshow}[1]{{}}
\newcommand{\numdam}[1]{{}}

\title{Fusion categories between $\cC \boxtimes \cD$ and $\cC * \cD$}
\author{Masaki Izumi, Scott Morrison, and David Penneys}
\begin{document}
\maketitle

\begin{abstract}
Given a pair of fusion categories $\cC$ and $\cD$, we may form the free product $\cC * \cD$ and the tensor product $\cC \boxtimes \cD$. It is natural to think of the tensor product as a quotient of the free product. What other quotients are possible?

When $\cC=\cD=A_2$, there is an infinite family of quotients interpolating between the free product and the tensor product (closely related to the $A_{2n-1}^{(1)}$ and $D_{n+2}^{(1)}$ subfactors at index 4).
Bisch and Haagerup discovered one example of such an intermediate quotient when $\cC=A_2$ and $\cD=T_2$, and suggested that there might be another family here. 
We show that such quotients are characterized by parameters $n \geq 1$ and $\omega$ with $\omega^{2n}=1$. 
For $n=1,2,3$, we show $\omega$ must be $1$, and construct the corresponding quotient ($n=1$ is the tensor product, $n=2$ is the example discovered by Bisch and Haagerup, and $n=3$ is new). 
We further show that there are no such quotients for $4 \leq n \leq 10$.
Our methods also apply to the case when $\cC=\cD=T_2$, and we prove similar results there.

During the preparation of this manuscript we learnt of an independent result of Liu's on subfactors. With the translation between the subfactor and fusion category settings provided here, it follows there are no such quotients for any $n \geq 4$.
\end{abstract}

\section{Introduction}

Given two pivotal categories $\cC$ and $\cD$, there are always two constructions resulting in pivotal categories which deserve to be called `composites' of $\cC$ with $\cD$, namely the tensor product $\cC\boxtimes \cD$ and the free product $\cC * \cD$.
Diagrammatically, these are very easy to describe. 
In the free product, we allow planar disjoint unions of diagrams from $\cC$ and $\cD$. In the tensor product, we superimpose a diagram from $\cC$ with a diagram from $\cD$. We can alternatively think about this as a diagram built from generators from both $\cC$ and $\cD$, allowing strings from $\cC$ and $\cD$ to cross each other, obeying the usual (symmetric, not braided) relations for crossings.

\begin{figure}[ht]
\centering
\begin{minipage}{0.47\textwidth}
\centering
\begin{tikzpicture}
\draw[gray] (0,0) circle (2cm);
\clip (0,0) circle (2cm);
\draw[thick,red] (-3,0) circle (3cm);
\draw[thick,red] (-2,0) -- (0,0);
\draw[thick,blue] (135:2) circle (0.5cm);
\draw[thick,blue] (0,-2) -- (1,2);
\draw[thick,red] (2,0) circle (0.8cm);
\end{tikzpicture} \\
(a) a typical free product morphism
\end{minipage}
\begin{minipage}{0.45\textwidth}
\centering
\begin{tikzpicture}
\draw[gray] (0,0) circle (2cm);
\clip (0,0) circle (2cm);
\draw[thick,red] (-3,0) circle (3cm);
\draw[thick,red] (-2,0) -- (0,0);
\draw[thick,blue] (135:2) circle (0.5cm);
\draw[thick,blue] (-2,-1) -- (2,0);
\draw[thick,red] (2,0) circle (0.8cm);
\end{tikzpicture} \\
(b) a typical $\tensor$-product morphism
\end{minipage}
\end{figure}

Now, one should think of the tensor product as a quotient of the free product: there is a faithful, dominant functor $\cC * \cD \to \cC \boxtimes \cD$ (because diagrams without crossings can be thought of as diagrams with zero crossings). Generalizing, we will call any such quotient of $\cC * \cD$ a composite of $\cC$ and $\cD$. 

We would like to study the question of which composites are possible, and begin in this paper by looking at the concrete examples where $\cC$ and $\cD$ are both either $A_2$ or $T_2$. These are the smallest non-trivial unitary fusion categories, and each is generated by a symmetrically self-dual object.

In any such composite, there is a particularly important invariant, the least length $\ell\geq 2$ such that the alternating words of length $\ell$ in $\alpha,\beta$ and in $\beta,\alpha$ are isomorphic, i.e.,
$$
\underbrace{(\alpha \beta \cdots)}_{\text{$\ell$ letters}} \cong \underbrace{(\beta\alpha \cdots)}_{\text{$\ell$ letters}},
$$ 
where $\alpha$ and $\beta$ are the symmetrically self-dual generators of $\cC$ and $\cD$.
In fact, we show that a composite, if it exists, is entirely determined by $\ell$ and another parameter $\omega$, which is an $\ell$-th root of unity. 
The argument relies on the jellyfish algorithm introduced in \cite{MR2979509}.
(It is more convenient to introduce an invariant $n$; when $\cC=\cD$, $n=\ell$, and when $\cC=A_2$ and $\cD=T_2$, $n=\ell/2$ since $\ell$ can never be odd because $\dim(\alpha)\neq \dim(\beta)$. Thus when $\cC=A_2$ and $\cD=T_2$, we have $(\alpha\beta)^n\cong (\beta\alpha)^n$.)

The composites of $A_2$ with itself offer an interesting example; there is an infinite family of composites interpolating between the tensor product and the free product. (These are closely related to the $A_{n}^{(1)}$ subfactors at index 4; see below for the connection with subfactor theory.)
For every value of $n \geq 2$, there exist exactly $n$ such composites, corresponding to each of the possible $n$-th roots of unity $\omega$.
We describe these in Section \ref{sec:A2A2}.

Bisch and Haagerup \cite{MR1386923} proposed in 1994 that they may also be an infinite family of composites of $A_2$ with $T_2$ (in the guise of a family of subfactors with index $3+\sqrt{5}$). We give constructions for some, namely the cases $n=1, 2,$ or $3$, each with $\omega=1$ \S \ref{sec:Existence}, and we show this is the only $\omega$ that is possible in these cases in Theorem \ref{thm:Unique}.
Our main result in this paper, however, is that not all of these exist: there are no such composites for $4 \leq n \leq 10$.

Our non-existence proof again uses the ideas from the jellyfish algorithm: we indicate two different ways to evaluate a particular diagram using a different sequence of applications of jellyfish relations, illustrated in Figure \ref{fig:illustration}. (In fact, this is a diagram with boundary, and we exhibit an explicit basis and two different ways to rewrite the diagram into that basis.) We find that for $4 \leq n \leq 10$ these give different results. Each individual $n$ requires a somewhat lengthy calculation, and at this point we do not have a uniform argument.

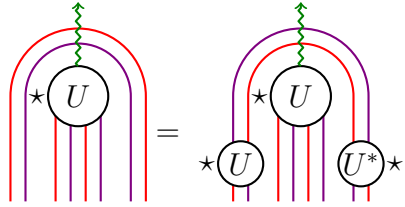
\begin{figure}[!ht]
$$\begin{tikzpicture}[baseline = -.6cm]
	\draw[thick, rho] (-.9,-1.4)--(-.9,0) arc (180:0:.9cm)--(.9,-1.4);
	\draw[thick, zeta] (-.7,-1.4)--(-.7,0) arc (180:0:.7cm)--(.7,-1.4);
	\draw[thick, zeta] (.3,0)--(.3,-1.4);
	\draw[thick, rho] (.1,0)--(.1,-1.4);
	\draw[thick, zeta] (-.1,0)--(-.1,-1.4);
	\draw[thick, rho] (-.3,0)--(-.3,-1.4);
	\draw[thick, unshaded] (0,0) circle (.4);
	\draw[->,line join=round,thick,green!50!black,decorate, decoration={    zigzag,    segment length=4,    amplitude=.9,post=lineto,    post length=2pt}] (0,0.4) -- (0,1.25);
	\node at (0,0) {$U$};
	\node at (180:.55cm) {$\star$};
\end{tikzpicture}
=
\begin{tikzpicture}[baseline = -.6cm]
	\coordinate (a) at (-.9,-.9);
	\coordinate (b) at (-.7,-.9);
	\coordinate (c) at (.7,-.9);
	\coordinate (d) at (.9,-.9);	 
	\draw[thick, rho] (a) -- (-.9,-1.4);
	\draw[thick, zeta] (b) -- (-.7,-1.4);
	\draw[thick, zeta] (c) -- (.7,-1.4);
	\draw[thick, rho] (d) -- (.9,-1.4);
	\draw[thick, zeta] (a)--(-.9,0) arc (180:0:.9cm)--(d);
	\draw[thick, rho] (b)--(-.7,0) arc (180:0:.7cm)--(c);
	\draw[thick, zeta] (.3,0)--(.3,-1.4);
	\draw[thick, rho] (.1,0)--(.1,-1.4);
	\draw[thick, zeta] (-.1,0)--(-.1,-1.4);
	\draw[thick, rho] (-.3,0)--(-.3,-1.4);
	\draw[thick, unshaded] (0,0) circle (.4);
	\node at (0,0) {$U$};
	\node at (180:.55cm) {$\star$};
	\draw[thick, unshaded] (-.8,-.9) circle (.3);
	\node at (-.8,-.9) {$U$};
	\node at (-1.25,-.9) {$\star$};
	\draw[thick, unshaded] (.8,-.9) circle (.3);
	\node at (.8,-.9) {$U^*$};
	\node at (1.25,-.9) {$\star$};
	\draw[->,line join=round,thick,green!50!black,decorate, decoration={    zigzag,    segment length=4,    amplitude=.9,post=lineto,    post length=2pt}] (0,0.4) -- (0,1.25);
\end{tikzpicture}
$$
\caption{The two different jellyfish evaluations. We either pull a certain unitary element $U$ through a pair of strings, or first create a canceling pair $UU^*$ on those strings, then pull the element $U$ between them. (In fact, the purple string here is a bundle of $2n-1$ strings.) Obtaining different answers, we see there can be no composites of $A_2$ with $T_2$ with parameter $4 \leq n \leq 10$.}
\label{fig:illustration}
\end{figure}

(Liu has independently proved, by a quite different method, that the corresponding subfactors do not exist for all $4 \leq n < \infty$ \cite{LiuFish}. Our Theorem \ref{thm:ATToFish} allows one to show the nonexistence of composites of $A_2$ with $T_2$ for all $n \geq 4$ as a corollary.)

When $\cC = \cD = T_2$, we have a similar result; unique composites exist for $n=2$ and $n=3$ (with $\omega=1$), and there are no composites for $4 \leq n \leq 10$. (Again, Liu's proof also applies to composites of two $A_4$ subfactors, and he obtains nonexistence for all $n \geq 4$. The same translation between fusion categories and subfactors then eliminates all composites of $T_2$ with $T_2$ for $n\geq 4$.)

\subsection{Connection to subfactor theory}

Given two hyperfinite subfactors $A \subset B$ and $C \subset D$, there is not a well-defined composition --- one first needs to pick an isomorphism between $B$ and $C$, and the resulting composite $A \subset D$ can depend sensitively on this choice.
If the isomorphism is generic, we get, by results of Popa and Vaes \cite{MR1372533,MR2471930}, the free composite subfactor of Bisch and Jones (see \cite[Section 8]{MR1950890}).

Jones' index rigidity theorem \cite{MR0696688} shows that the first indices of composite subfactors are $4=2\cdot 2$, $3+\sqrt{5}=2\left( \frac{3+\sqrt{5}}{2}\right)$, and $6=2\cdot 3$.

Subfactors of index 4 were completely classified in \cite{MR1278111}. Those subfactors with intermediates are easily seen to come from groups. In \cite{MR1386923}, Bisch and Haagerup start with two finite groups $H,K$ with outer actions on the hyperfinite $\textrm{II}_1$-factor $R$, and they study the composite subfactor $R^H\subset R\rtimes K$. As one might expect, this subfactor depends on the group $G$ generated by $H,K$ in $\Out(R)$ (and in particular not just on the isomorphism class of the subfactors $R^H \subset R$ and $R \subset R \rtimes K$). For $H=K=\Integer/2$, $G$ must be some quotient of the infinite dihedral group $D_\infty=\Integer/2 * \Integer/2$, and this construction exhausts all subfactors with intermediates at index 4.

At index 6, we can produce composite subfactors for any $G$ which is a quotient of the modular group $PSL(2,\Integer)=\Integer/2*\Integer/3$. 
The situation here is much more complicated, and we do not expect that the Bisch-Haagerup subfactors exhaust all possibilities. 
(In particular, one should also expect composites of $A_3$ and $A_5$ subfactors.)

A composite of $A_3$ and $A_4$ subfactors has index $3+\sqrt{5}$, and in fact any composite subfactor at this index is a composite of these subfactors. One of the main goals of this paper is to understand such composites. None come from the Bisch-Haagerup construction, merely because the index is not a composite integer.
(However, one can see that each composite subfactor at this index is of the form $N^\alpha \subset N\subset M$ where $N\subset M$ is the hyperfinite $A_4$ subfactor and $\alpha$ is an outer automorphism of period 2 \cite{MR0107827}.)

Bisch and Jones constructed the free composite of Temperley-Lieb subfactors in \cite{MR1437496}, and in 1994 Bisch and Haagerup found a sequence of possible principal graphs at index $3+\sqrt{5}$. Our work on this project was motivated by understanding these principal graphs.

In fact, this question is very closely related to the question of composites of fusion categories raised above---the even part of a composite subfactor at index $3+\sqrt{5}$ is itself a composite of the fusion categories $\frac{1}{2} A_3 \iso A_2$ and $\frac{1}{2} A_4 \iso T_2$. See Section \ref{sec:Subfactors} for details. Thus whenever we can rule out the existence of some class of composites of the fusion categories, we rule out possible composite subfactors.

Moreover, the connection is even tighter. The free product $A_2 * T_2$ tensor category contains an algebra object, which allows us to reconstruct a subfactor. This gives the free composite constructed by Bisch and Jones. This algebra object passes down to any quotient of the free product, and so from any composite of $A_2$ with $T_2$, we obtain a subfactor composed of $A_3$ and $A_4$ subfactors. Thus there is actually a one-to-one correspondence between the two classes of objects under study.

Since the Fuss-Catalan subfactor planar algebra at $3+\sqrt{5}$ is not amenable, it remains a completely open question whether there are infinitely many non-isomorphic hyperfinite subfactors with this standard invariant. 
Fuss-Catalan arises as the standard invariant of a hyperfinite subfactor by \cite{MR1372533,MR2471930}. As alluded to above, if standard invariants $P_\bullet$ and $Q_\bullet$ arise from hyperfinite $\textrm{II}_1$-subfactor, then $P_\bullet*Q_\bullet$ does too.

\subsection{Definitions}\label{sec:Definitions}

The free product of two shaded planar algebras, due to Bisch and Jones, was defined in \cite[Section 8]{MR1950890} (see also \cite{MR1437496}). 
We give the straightforward translation of this notion to the general case of the free product of two strict pivotal categories.

\begin{defn}
Given two strict pivotal tensor categories $\cC$ and $\cD$, we define their free product $\cC*\cD$ as follows. For simplicity, let's assume the tensor identities of $\cC$ and $\cD$ are simple.
The objects of $\cC*\cD$ are words in the objects of $\cC$ and $\cD$. 

To specify the morphisms, we describe the invariant vectors for each word $w$ (that is, the morphisms from the tensor identity to $w$). In fact, the morphisms from $w_1$ to $w_2$ are exactly the same as the invariant vectors for $w_2 \bar{w_1}$.

Consider a disc with $w$ written around the boundary. A planar partition of this disc is a partition of the disc into two submanifolds, each disjoint unions of discs, partially glued along their boundaries, up to isotopies fixing the boundary of the disc. We call these two submanifolds the $\cC$ region and the $\cD$ region. 
We require that the letters of $w$ from $\cC$ appear only in the $\cC$ regions, and similarly for $\cD$. Each disc of each region thus inherits a word of objects from the appropriate category around its boundary (possibly empty).

An invariant vector for $w$ then consists of a planar partition along with an invariant vector for the corresponding word for each disc of each of the two regions, up to a certain equivalence relation,  given by the following diagram
$$
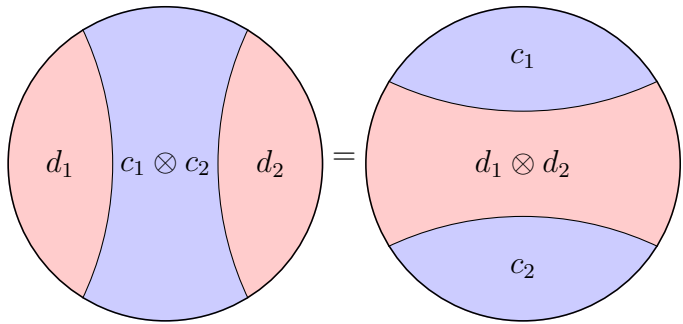
\begin{tikzpicture}[x=0.7cm,y=0.7cm,baseline]
\path[fill=blue!20] (0,0) circle (3);
\clip (0,0) circle (3);
\draw[fill=red!20] (7,0) circle (6);
\draw[fill=red!20] (-7,0) circle (6);
\draw[very thick] (0,0) circle (3);
\node at (0,0) {$c_1 \otimes c_2$};
\node at (-2,0) {$d_1$};
\node at (2,0) {$d_2$};
\end{tikzpicture}
=
\begin{tikzpicture}[x=0.7cm,y=0.7cm,baseline]
\path[fill=red!20] (0,0) circle (3);
\clip (0,0) circle (3);
\draw[fill=blue!20] (0,7) circle (6);
\draw[fill=blue!20] (0,-7) circle (6);
\draw[very thick] (0,0) circle (3);
\node at (0,0) {$d_1 \otimes d_2$};
\node at (0,-2) {$c_2$};
\node at (0,2) {$c_1$};
\end{tikzpicture}
$$
where the $c_i$ are morphisms from $\cC$ and the $d_i$ are morphisms from $\cD$.

Composition and tensor product are defined in the obvious way, gluing such discs together along their boundaries.

We see that $\cC*\cD$ is again a strict pivotal category. 
The dual of $w=a_1\cdots a_n$ where $a_i\in\cC$ or $\cD$ is $\overline{a_n}\cdots \overline{a_1}$.
The evaluation and coevaluation maps for words are given by the obvious planar partitions labelled by evaluation or coevaluation maps for the letters. 
\end{defn}

\begin{remark}
If the tensor identities are not simple (i.e. there are closed diagrams which are not multiples of the empty diagram), we would need to allow more general planar partitions, where the two submanifolds are not necessarily disjoint unions of discs. We leave the details to an interested reader.
\end{remark}

\begin{remark}
The even more general cases of pivotal 2-categories, or free products amalgamated over a common subcategory, are interesting subjects for future research!
\end{remark}

We expect that the free product of two unitary categories is again unitary. If both categories are Temperley-Lieb categories or the even halves of Temperley-Lieb categories, the free product is unitary by results of Bisch-Jones \cite{MR1437496}. The general case would take us too far afield for now.

\begin{defn}
A quotient of a pivotal category $\cC$ is a pivotal category $\cD$ together with a faithful, dominant tensor functor $F\colon \cC\to \cD$ which preserves the pivotal structure.
\end{defn}

One should think of a quotient of $\cC$ as  some pivotal category generated by  the same objects as $\cC$ but more morphisms. Because of these additional morphisms, objects that were simple in $\cC$ may break up into smaller objects in $\cD$, or objects that were not isomorphic in $\cC$ may become isomorphic in $\cD$. Conversely every simple object in $\cD$ arises as a summand of some object in $\cC$.

\subsection{Acknowledgements}

We would like to thanks Vaughan Jones, Dietmar Bisch, and Uffe Haagerup for many helpful conversations about the `fish' subfactors. We also thank Zhengwei Liu for explaining his perspective on the problem, and sharing his excellent results with us prior to publication.

We would like to thank the Mathematical Sciences Institute at the Australian National University for hosting all the authors while undertaking this research.

Masaki Izumi was supported by JSPS, the Grant-in-Aid for Scientific Research (B) 22340032.
Scott Morrison was supported by an Australian Research Council `Discovery Early Career Researcher Award', DE120100232.
David Penneys was supported in part by the Natural Sciences and Engineering Research Council of Canada.
Scott Morrison and David Penneys were both supported by DOD-DARPA grant HR0011-12-1-0009.

\section{$A_2$ with $A_2$}

We recall that the pivotal category $A_2$ has two simple objects $1$ and $\theta$, which are both symmetrically self-dual. The category $A_2$ has no generators as a pivotal category (meaning all morphisms are tensor generated by identities and (co)evaluations) and the following relations:
 \begin{align}
\begin{tikzpicture}[baseline = -.1cm]
	\draw[thick, theta] (0,0) circle (.3cm);
	\draw[thick] (-.4,-.4)--(-.4,.4)--(.4,.4)--(.4,-.4)--(-.4,-.4);
\end{tikzpicture}\,
&=
1
\tag{AA1}\label{rel:AA1}
\\
\begin{tikzpicture}[baseline = -.1cm]
	\draw[thick, theta] (-.2,-.4) -- ( -.2, .4);	
	\draw[thick, theta] (.2,-.4) -- (.2, .4);
	\draw[thick] (-.4,-.4)--(-.4,.4)--(.4,.4)--(.4,-.4)--(-.4,-.4);
\end{tikzpicture}\,
&=
\begin{tikzpicture}[baseline = -.1cm]
	\draw[thick, theta] (-.2,-.4)arc (180:0:.2cm);	
	\draw[thick, theta] (-.2,.4) arc (-180:0:.2cm);
	\draw[thick] (-.4,-.4)--(-.4,.4)--(.4,.4)--(.4,-.4)--(-.4,-.4);
\end{tikzpicture}\,.
\tag{AA2}\label{rel:AA2}
\end{align}
The first relation simply says $\dim{\theta}=1$. In a unitary category, the second relation follows from the first by calculating the norm of the difference of the two terms.

We begin by considering the free product $A_2* A_2$ where the two copies of $A_2$ are generated by objects $\alpha$ and $\theta$ satisfying $\alpha\otimes \alpha\cong 1$ and $\theta\otimes \theta\cong 1$. 
 This setting has been studied previously, e.g. in \cite{MR1065437,MR1213139},  and is a warmup case to our real goal, where we replace one or both copies of $A_2$ with $T_2$.

We represent $\alpha,\theta$ by unoriented green and blue strings respectively:
$$
\alpha=
\begin{tikzpicture}[baseline = -.1cm]	
	\draw[thick, alpha] (0,-.4) -- (0, .4);
	\draw[thick] (-.4,-.4)--(-.4,.4)--(.4,.4)--(.4,-.4)--(-.4,-.4);
\end{tikzpicture}
\text{ and }
\theta=
\begin{tikzpicture}[baseline = -.1cm]	
	\draw[thick, theta] (0,-.4) -- (0, .4);
	\draw[thick] (-.4,-.4)--(-.4,.4)--(.4,.4)--(.4,-.4)--(-.4,-.4);
\end{tikzpicture}\,.
$$
By Frobenius reciprocity, the distinct simple objects of $A_2*A_2$ are all alternating words in $\alpha,\theta$. Hence representatives of the isomorphism classes of simples are given by
$$
\underset{1}{
\begin{tikzpicture}[baseline = -.1cm]
	\draw[thick] (-.4,-.4)--(-.4,.4)--(.4,.4)--(.4,-.4)--(-.4,-.4);
\end{tikzpicture}
}
\,,\,
\underset{\alpha}{
\begin{tikzpicture}[baseline = -.1cm]
	\draw[thick, alpha] (0,-.4) -- ( 0, .4);	
	\draw[thick] (-.4,-.4)--(-.4,.4)--(.4,.4)--(.4,-.4)--(-.4,-.4);
\end{tikzpicture}
}
\,,\,
\underset{\theta}{
\begin{tikzpicture}[baseline = -.1cm]
	\draw[thick, theta] (0,-.4) -- ( 0, .4);	
	\draw[thick] (-.4,-.4)--(-.4,.4)--(.4,.4)--(.4,-.4)--(-.4,-.4);
\end{tikzpicture}
}
\,,\,
\underset{\alpha\theta}{
\begin{tikzpicture}[baseline = -.1cm]
	\draw[thick, alpha] (-.1,-.4) -- ( -.1, .4);	
	\draw[thick, theta] (.1,-.4) -- ( .1, .4);	
	\draw[thick] (-.4,-.4)--(-.4,.4)--(.4,.4)--(.4,-.4)--(-.4,-.4);
\end{tikzpicture}
}
\,,\,
\underset{\theta\alpha}{
\begin{tikzpicture}[baseline = -.1cm]
	\draw[thick, theta] (-.1,-.4) -- ( -.1, .4);	
	\draw[thick, alpha] (.1,-.4) -- ( .1, .4);	
	\draw[thick] (-.4,-.4)--(-.4,.4)--(.4,.4)--(.4,-.4)--(-.4,-.4);
\end{tikzpicture}
}
\,,\,
\underset{\alpha\theta\alpha}{
\begin{tikzpicture}[baseline = -.1cm]
	\draw[thick, alpha] (-.2,-.4) -- (-.2, .4);	
	\draw[thick, theta] (0,-.4) -- ( 0, .4);	
	\draw[thick, alpha] (.2,-.4) -- (.2, .4);
	\draw[thick] (-.4,-.4)--(-.4,.4)--(.4,.4)--(.4,-.4)--(-.4,-.4);
\end{tikzpicture}
}
\,,\,
\underset{\theta\alpha\theta}{
\begin{tikzpicture}[baseline = -.1cm]
	\draw[thick, theta] (-.2,-.4) -- (-.2, .4);	
	\draw[thick, alpha] (0,-.4) -- ( 0, .4);	
	\draw[thick, theta] (.2,-.4) -- (.2, .4);
	\draw[thick] (-.4,-.4)--(-.4,.4)--(.4,.4)--(.4,-.4)--(-.4,-.4);
\end{tikzpicture}
}
\,,\dots
$$

\subsection{Quotients of $A_2*A_2$}\label{sec:A2A2}

We will now show that all unitary quotients of $A_2*A_2$ are parametrized by an $n\geq 2$ and an $n$-th root of unity $\omega$.

Note that all alternating words in $\alpha,\theta$ are always simple in any unitary quotient of $A_2*A_2$ by Frobenius reciprocity.

The proof of the following proposition is a simple induction argument and is left to the reader. (It is also similar to Proposition \ref{prop:DistinctIrreducibleBimodules}, and the reader may look there for the general technique.)

\begin{prop}\label{prop:DistinctIrreducibleBimodulesA2A2}
Suppose that for some $n\geq 2$, the alternating words in $\alpha,\theta$ and $\theta,\alpha$ of length $n$ are not isomorphic, i.e.,
 $$
\underbrace{(\alpha \theta \cdots)}_{\text{length $n$}} \ncong \underbrace{(\theta\alpha \cdots)}_{\text{length $n$}} .
 $$ 
 Then all alternating words in $\alpha,\theta$ with length at most $n+1$ give distinct simple objects, except that the alternating words in $\alpha,\theta$ and $\theta,\alpha$ of length $n+1$ might not be distinct.
 \end{prop}

\begin{cor}\label{cor:NontrivialQuotientA2A2}
Either there is an $n\geq 2$ such that 
$$
\underbrace{(\alpha \theta \cdots)}_{\text{length $n$}} \cong \underbrace{(\theta\alpha \cdots)}_{\text{length $n$}},
$$ 
or there is no such $n$.
In either case, we know all the distinct simples generated by $\alpha,\theta$.
\end{cor}

The free product tensor category $A_2 * A_2$ has no extra relations. In the tensor product $A_2\boxtimes A_2$, $\alpha\theta\cong \theta \alpha$, which means we have an isomorphism from $\alpha\theta\to\theta\alpha$ which we denote by a crossing:
$$
\begin{tikzpicture}[baseline = -.1cm, scale=1.4]
	\draw[thick, alpha] (-.2,-.4) -- (.2, .4);	
	\draw[thick, theta] (.2,-.4) -- (-.2, .4);
\end{tikzpicture}
:
\alpha\, \theta\,
\overset{\cong}{\longrightarrow}
\,\theta\,\alpha.
$$

If there is an $n\in\Natural$ as in Corollary \ref{cor:NontrivialQuotientA2A2}, then we have a unitary isomorphism $U: (\alpha\theta)^{2n}\to 1$. For example, when $n=3$, we have
$$
\begin{tikzpicture}[baseline = -.1cm]
	\draw[thick, alpha] (-.2,-.8) -- (-.2,0);
	\draw[thick, theta] (0,-.8) -- (0,0);	
	\draw[thick, alpha] (.2,-.8) -- (.2,0);	
	\draw[thick, theta] (-.2,.8) -- (-.2,0);
	\draw[thick, alpha] (0,.8) -- (0,0);	
	\draw[thick, theta] (.2,.8) -- (.2,0);	
	\draw[thick, unshaded] (0,0) circle (.4cm);
	\node at (0,0) {$U$};
	\node at (-.55,0) {$\star$};
\end{tikzpicture}
:
\alpha\theta\alpha
\overset{\cong}{\longrightarrow}
\theta\alpha\theta,
$$
which means we have the following relations: 
\begin{equation}
UU^* = 
\begin{tikzpicture}[baseline = .5cm]
	\draw[thick, alpha] (-.2,2)--(-.2,1.2);
	\draw[thick, theta] (0,2)--(0,1.2);
	\draw[thick, alpha] (.2,2)--(.2,1.2);
	\draw[thick, theta] (-.2,1.2)--(-.2,0);
	\draw[thick, alpha] (0,1.2)--(0,0);
	\draw[thick, theta] (.2,1.2)--(.2,0);
	\draw[thick, alpha] (-.2,0)--(-.2,-.8);
	\draw[thick, theta] (0,0)--(0,-.8);
	\draw[thick, alpha] (.2,0)--(.2,-.8);
	\draw[unshaded, thick] (0,0) circle (.4cm);
	\draw[unshaded, thick] (0,1.2) circle (.4cm);
	\draw[unshaded, thick] (0,0) circle (.4cm);
	\draw[unshaded, thick] (0,1.2) circle (.4cm);
	\node at (0,0) {$U$};
	\node at (0,1.2) {$U^*$};
	\node at (-.55,0) {$\star$};
	\node at (-.55,1.2) {$\star$};
\end{tikzpicture}
=
\begin{tikzpicture}[baseline = -.1cm]
	\draw[thick, alpha] (-.2,-.8)--(-.2,.8);
	\draw[thick, theta] (0,-.8)--(0,.8);
	\draw[thick, alpha] (.2,-.8)--(.2,.8);
\end{tikzpicture}
=
\id_{\alpha\theta\alpha}
\text{ and }
U^*U= 
\begin{tikzpicture}[baseline = .5cm]
	\draw[thick, theta] (-.2,2)--(-.2,1.2);
	\draw[thick, alpha] (0,2)--(0,1.2);
	\draw[thick, theta] (.2,2)--(.2,1.2);
	\draw[thick, alpha] (-.2,1.2)--(-.2,0);
	\draw[thick, theta] (0,1.2)--(0,0);
	\draw[thick, alpha] (.2,1.2)--(.2,0);
	\draw[thick, theta] (-.2,0)--(-.2,-.8);
	\draw[thick, alpha] (0,0)--(0,-.8);
	\draw[thick, theta] (.2,0)--(.2,-.8);
	\draw[unshaded, thick] (0,0) circle (.4cm);
	\draw[unshaded, thick] (0,1.2) circle (.4cm);
	\node at (0,0) {$U^*$};
	\node at (0,1.2) {$U$};
	\node at (-.55,0) {$\star$};
	\node at (-.55,1.2) {$\star$};
\end{tikzpicture}
=
\begin{tikzpicture}[baseline = -.1cm]
	\draw[thick, theta] (-.2,-.8)--(-.2,.8);
	\draw[thick, alpha] (0,-.8)--(0,.8);
	\draw[thick, theta] (.2,-.8)--(.2,.8);
\end{tikzpicture}
=
\id_{\theta\alpha\theta}.
\tag{AA3}\label{rel:AA3}
\end{equation}
Moreover, we may normalize $U$ by a phase so that 
\begin{equation}
\cF^{-1}(U)=
\begin{tikzpicture}[baseline = -.1cm, yscale=-1]
	\draw[thick, theta] (-.2,-.3) arc (0:-180:.25cm) --(-.7,.8);
	\draw[thick, alpha] (-.2,.8)--(-.2,0);
	\draw[thick, theta] (0,.8)--(0,0);
	\draw[thick, alpha] (0,0)--(0,-.8);
	\draw[thick, theta] (.2,0)--(.2,-.8);
	\draw[thick, alpha] (.2,.3) arc (180:0:.25cm) --(.7,-.8);
	\draw[unshaded, thick] (0,0) circle (.4cm);
	\node at (0,0) {$U$};
	\node at (-.55,0) {$\star$};
\end{tikzpicture}
=
\begin{tikzpicture}[baseline = -.1cm]
	\draw[thick, alpha] (-.2,.8)--(-.2,0);
	\draw[thick, theta] (0,.8)--(0,0);
	\draw[thick, alpha] (.2,.8)--(.2,0);
	\draw[thick, theta] (-.2,0)--(-.2,-.8);
	\draw[thick, alpha] (0,0)--(0,-.8);
	\draw[thick, theta] (.2,0)--(.2,-.8);
	\draw[unshaded, thick] (0,0) circle (.4cm);
	\node at (0,0) {$U^*$};
	\node at (-.55,0) {$\star$};
\end{tikzpicture}
=
\omega_U^{-1}
\begin{tikzpicture}[baseline = -.1cm]
	\draw[thick, alpha] (-.2,-.3) arc (0:-180:.25cm) --(-.7,.8);
	\draw[thick, theta] (-.2,.8)--(-.2,0);
	\draw[thick, alpha] (0,.8)--(0,0);
	\draw[thick, theta] (0,0)--(0,-.8);
	\draw[thick, alpha] (.2,0)--(.2,-.8);
	\draw[thick, theta] (.2,.3) arc (180:0:.25cm) --(.7,-.8);
	\draw[unshaded, thick] (0,0) circle (.4cm);
	\node at (0,0) {$U$};
	\node at (-.55,0) {$\star$};
\end{tikzpicture}
=
\omega_U^{-1}\cF(U)
\tag{AA4}\label{rel:AA4}
\end{equation}
for some $n$-th root of unity $\omega_U$.

\begin{remark}\label{rem:A2A2StarStructure}
If we were looking for all quotients, not just unitary ones, we would find at this point that there is a second possibility. Once we have an isomorphism $U: (\alpha \theta)^{2n} \to 1$, there are exactly two ways to endow the quotient with a $*$-structure. The first case is the unitary case described above, and in the other case, $U$ satisfies $UU^* = -1$. We won't pursue this here.
\end{remark}

\begin{prop}\label{prop:A2A2Jellyfish}
$U$ satisfies the following jellyfish relations:
$$
\begin{tikzpicture}[baseline = -.3cm]
	\draw[thick, theta] (-.7,-.8)-- (-.7,0) arc(180:0:.7cm) -- (.7,-.8);
	\draw[] (0,0) -- (0,-.8);
	\node at (-.2,-.6) {\scriptsize{$2n$}};
	\draw[thick, unshaded]  (0,0) circle (.4cm);
	\node at (0,0) {$U$};
	\node at (-.55,0) {$\star$};
\end{tikzpicture}
=
\begin{tikzpicture}[baseline = -.3cm]
	\draw[] (0,0) -- (0,-.8);
	\node at (-.2,-.6) {\scriptsize{$2n$}};
	\draw[thick, theta] (.5,-.8) arc (180:0:.2cm);
	\draw[thick, unshaded]  (0,0) circle (.4cm);
	\node at (0,0) {$U^*$};
	\node at (-.55,0) {$\star$};
\end{tikzpicture}
\text{ and }
\begin{tikzpicture}[baseline = -.3cm]
	\draw[thick, alpha] (-.7,-.8)-- (-.7,0) arc(180:0:.7cm) -- (.7,-.8);
	\draw[] (0,0) -- (0,-.8);
	\node at (-.2,-.6) {\scriptsize{$2n$}};
	\draw[thick, unshaded]  (0,0) circle (.4cm);
	\node at (0,0) {$U$};
	\node at (-.55,0) {$\star$};
\end{tikzpicture}
=
\omega_U^{}
\begin{tikzpicture}[baseline = -.3cm]
	\draw[] (0,0) -- (0,-.8);
	\node at (-.2,-.6) {\scriptsize{$2n$}};
	\draw[thick, alpha] (-.9,-.8) arc (180:0:.2cm);
	\draw[thick, unshaded]  (0,0) circle (.4cm);
	\node at (0,0) {$U^*$};
	\node at (-.55,0) {$\star$};
\end{tikzpicture}
$$
\end{prop}
\begin{proof}
This follows immediately by applying Relations \eqref{rel:AA2} and \eqref{rel:AA4}.
\end{proof}

The relations in Proposition \ref{prop:A2A2Jellyfish} immediately allow us to evaluate all closed diagrams in $U,U^*$ using the jellyfish algorithm of \cite{MR2979509}. 
Thus there is at most one quotient of $A_2 * A_2$ for each pair $(n, \omega_U)$. 

\begin{defn}
For $2\leq n<\infty$, let $\AA_{n,\omega_U}$ be the unitary quotient of $A_2*A_2$ generated by $U$ satisfying Relations \eqref{rel:AA1}-\eqref{rel:AA4}, provided that it exists. Note that $\AA_{1,1}$ is $A_2\boxtimes A_2$.
\end{defn}

The following theorem would take us too far off course for now. We postpone such an exploration to \cite{MPAffineAandD}, where we also explain the connection between $\AA_{n,\omega_U}$ and the $A_{2n-1}^{(1)}$ and $D_{n+2}^{(1)}$ subfactors.

\begin{thm}\label{thm:AAQuotients}
The category $\AA_{n,\omega_U}$ exists for each $n\geq 2$ and $\omega_U^n=1$ and is realized by the pointed unitary category ${\sf Vec}_{D_{2n}}^{\lambda}$ for some $\lambda\in H^3(D_{2n},\Complex^\times)$ determined by $\omega_U$.
\end{thm}

\section{$A_2$ with $T_2$}

We now discuss quotients of the free product of $A_2*T_2$, where $T_2$ is the even half of $A_4$ ($A_4$ is Temperley-Lieb with $\delta=\tau=\frac{1+\sqrt{5}}{2}$). 

Recall that $T_2$ has two simple objects $1,\rho$ where $\rho\otimes \rho\cong 1\oplus \rho$. 
We denote $\rho$ by a red strand, and we write a trivalent vertex for the intertwiner $\rho\otimes \rho \to \rho$ given by
$$
\begin{tikzpicture}[baseline = -.1cm]
	\filldraw[rho] (0,0) circle (.05cm);
	\draw[thick, rho] (0,0) -- (-.2,-.4);
	\draw[thick, rho] (0,0) -- (0,.4);
	\draw[thick, rho] (0,0) -- (.2,-.4);
	\nbox{}{(0,0)}{0}{0}{}
\end{tikzpicture}
=
\left(
\frac{[2]}{[3]-1}
\right)^{1/2}
\begin{tikzpicture}[baseline = -.6cm, scale=.6]
	\clip (-1.2,-2.4) -- (-1.2,.8) -- (1.2,.8) -- (1.2,-2.4);
	\draw[shaded] (-.4,-2.5) -- (-.4,-1.2) arc (180:0:.4cm) -- (.4,-2.5) -- (.8,-2.5) -- (.8,-1.2) -- (.2,-.4) -- (.2,.9) -- (-.2,.9) -- (-.2,-.4) -- (-.8,-1.2) -- (-.8,-2.5);
	\nbox{unshaded}{(0,0)}{0}{0}{2}
	\nbox{unshaded}{(-.6,-1.6)}{0}{0}{2}
	\nbox{unshaded}{(.6,-1.6)}{0}{0}{2}
\end{tikzpicture}
$$
where we just write $2$ for $\jw{2}$, and $[2]=[3]=\tau$.

\begin{prop}\label{prop:RhoSkeinRelations}
We have the following skein relations in $T_2$:
\begin{align}
\begin{tikzpicture}[baseline = -.1cm]
	\filldraw[rho] (0,.2) circle (.05cm);
	\filldraw[rho] (0,-.2) circle (.05cm);
	\draw[thick, rho] (0,0) circle (.2cm);	
	\draw[thick, rho] (0,-.2) -- (0, -.4);
	\draw[thick, rho] (0,.2) -- (0, .4);
	\draw[thick] (-.4,-.4)--(-.4,.4)--(.4,.4)--(.4,-.4)--(-.4,-.4);
\end{tikzpicture}
&=
\begin{tikzpicture}[baseline = -.1cm]	
	\draw[thick, rho] (0,-.4) -- (0, .4);
	\draw[thick] (-.4,-.4)--(-.4,.4)--(.4,.4)--(.4,-.4)--(-.4,-.4);
\end{tikzpicture}
\notag
\\
\begin{tikzpicture}[baseline = -.1cm]
	\draw[thick, rho] (0,0) circle (.3cm);
	\draw[thick] (-.4,-.4)--(-.4,.4)--(.4,.4)--(.4,-.4)--(-.4,-.4);
\end{tikzpicture}
&=
\begin{tikzpicture}[baseline = -.1cm]
	\filldraw[rho] (-.3,0) circle (.05cm);
	\filldraw[rho] (.3,0) circle (.05cm);
	\draw[thick, rho] (0,0) circle (.3cm);
	\draw[thick, rho] (-.3,0)--(.3,0);
	\draw[thick] (-.4,-.4)--(-.4,.4)--(.4,.4)--(.4,-.4)--(-.4,-.4);
\end{tikzpicture}
=
\tau
\notag
\\
\begin{tikzpicture}[baseline = -.1cm]
	\filldraw[rho] (0,-.2) circle (.05cm);
	\draw[thick, rho] (0,0) circle (.2cm);	
	\draw[thick, rho] (0,-.2) -- (0, -.4);
	\draw[thick] (-.4,-.4)--(-.4,.4)--(.4,.4)--(.4,-.4)--(-.4,-.4);
\end{tikzpicture}
&=
0
\notag
\\
\begin{tikzpicture}[baseline = -.1cm]
	\filldraw[rho] (0,0) circle (.05cm);
	\draw[thick, rho] (0,0) arc (0:-180:.15cm) -- (-.3,.4);
	\draw[thick, rho] (0,0) -- (0,.4);
	\draw[thick, rho] (0,0) -- (.2,-.4);
	\draw[thick] (-.5,-.4)--(-.5,.4)--(.4,.4)--(.4,-.4)--(-.5,-.4);
\end{tikzpicture}
&=
\begin{tikzpicture}[baseline = -.1cm]
	\filldraw[rho] (0,0) circle (.05cm);
	\draw[thick, rho] (0,0) -- (-.2,-.4);
	\draw[thick, rho] (0,0) -- (0,.4);
	\draw[thick, rho] (0,0) -- (.2,-.4);
	\draw[thick] (-.4,-.4)--(-.4,.4)--(.4,.4)--(.4,-.4)--(-.4,-.4);
\end{tikzpicture}^{\,*}
=
\begin{tikzpicture}[baseline = -.1cm]
	\filldraw[rho] (0,0) circle (.05cm);
	\draw[thick, rho] (0,0) -- (-.2,.4);
	\draw[thick, rho] (0,0) -- (0,-.4);
	\draw[thick, rho] (0,0) -- (.2,.4);
	\draw[thick] (-.4,-.4)--(-.4,.4)--(.4,.4)--(.4,-.4)--(-.4,-.4);
\end{tikzpicture}
\notag
\\
\begin{tikzpicture}[baseline = -.1cm]
	\draw[thick, rho] (-.2,-.4) -- ( -.2, .4);	
	\draw[thick, rho] (.2,-.4) -- (.2, .4);
	\draw[thick] (-.4,-.4)--(-.4,.4)--(.4,.4)--(.4,-.4)--(-.4,-.4);
\end{tikzpicture}
&=
\frac{1}{\tau}\,
\begin{tikzpicture}[baseline = -.1cm]
	\draw[thick, rho] (-.2,-.4)arc (180:0:.2cm);	
	\draw[thick, rho] (-.2,.4) arc (-180:0:.2cm);
	\draw[thick] (-.4,-.4)--(-.4,.4)--(.4,.4)--(.4,-.4)--(-.4,-.4);
\end{tikzpicture}\,
+
\begin{tikzpicture}[baseline = -.1cm]
	\filldraw[rho] (0,.2) circle (.05cm);
	\filldraw[rho] (0,-.2) circle (.05cm);
	\draw[thick, rho] (-.2,.4) -- (0,.2) --  ( .2, .4);	
	\draw[thick, rho] (-.2,-.4) -- (0,-.2) --  ( .2, -.4);	
	\draw[thick, rho] (0,-.2) -- (0, .2);
	\draw[thick] (-.4,-.4)--(-.4,.4)--(.4,.4)--(.4,-.4)--(-.4,-.4);
\end{tikzpicture}\,
\label{rel:AT1}
\tag{AT1}
\end{align}
\end{prop}
\begin{proof}
The first equation follows from
$$
\begin{tikzpicture}[baseline = -.1cm, scale=.6]
	\clip (-1.2,-2.4) -- (-1.2,2.4) -- (1.2,2.4) -- (1.2,-2.4);
	\filldraw[shaded] (-.2,2.5) -- (-.2,1.2) -- (-.8,.4) -- (-.8,-.4) -- (-.2,-1.2) -- (-.2,-2.5) -- (.2,-2.5) -- (.2,-1.2) -- (.8,-.4) -- (.8,.4) -- (.2,1.2) -- (.2,2.5);
	\filldraw[unshaded] (-.4,.4) arc (180:0:.4cm) -- (.4,-.4) arc (0:-180:.4cm);
	\nbox{unshaded}{(0,1.6)}{0}{0}{2}
	\nbox{unshaded}{(-.6,0)}{0}{0}{2}
	\nbox{unshaded}{(.6,0)}{0}{0}{2}
	\nbox{unshaded}{(0,-1.6)}{0}{0}{2}
\end{tikzpicture}
=
\begin{tikzpicture}[baseline = -.1cm, scale=.6]
	\clip (-1.2,-2.4) -- (-1.2,2.4) -- (1.2,2.4) -- (1.2,-2.4);
	\filldraw[shaded] (-.2,2.5) -- (-.2,1.2) -- (-.8,.4) -- (-.8,-.4) -- (-.2,-1.2) -- (-.2,-2.5) -- (.2,-2.5) -- (.2,-1.2) -- (.8,-.4) -- (.8,.4) -- (.2,1.2) -- (.2,2.5);
	\filldraw[unshaded] (-.4,.4) arc (180:0:.4cm) -- (.4,-.4) arc (0:-180:.4cm);
	\nbox{unshaded}{(0,1.6)}{0}{0}{2}
	\nbox{unshaded}{(-.6,0)}{0}{0}{2}
	\nbox{}{(.6,0)}{0}{0}{}
	\nbox{unshaded}{(0,-1.6)}{0}{0}{2}
\end{tikzpicture}
-
\frac{1}{[2]}
\begin{tikzpicture}[baseline = -.1cm, scale=.6]
	\clip (-1.2,-2.4) -- (-1.2,2.4) -- (1.2,2.4) -- (1.2,-2.4);
	\filldraw[shaded] (-.2,2.5) -- (-.2,1.2) -- (-.8,.4) -- (-.8,-.4) -- (-.2,-1.2) -- (-.2,-2.5) -- (.2,-2.5) -- (.2,-1.2) -- (.8,-.4) -- (.8,.4) -- (.2,1.2) -- (.2,2.5);
	\filldraw[unshaded] (-.4,.4) arc (180:0:.4cm) -- (.4,-.4) arc (0:-180:.4cm);
	\nbox{unshaded}{(0,1.6)}{0}{0}{2}
	\nbox{unshaded}{(-.6,0)}{0}{0}{2}
	\fill[unshaded] (.3,.4) -- (.3,-.4) -- (.9,-.4) -- (.9,.4);
	\draw[shaded] (.4,-.4) arc (180:0:.2cm);
	\draw[shaded] (.4,.4) arc (-180:0:.2cm);
	\nbox{}{(.6,0)}{0}{0}{}
	\nbox{unshaded}{(0,-1.6)}{0}{0}{2}
\end{tikzpicture}
=
\left(
\frac{[3]}{[2]}-\frac{1}{[2]}
\right)\,
\begin{tikzpicture}[baseline = -.1cm]	
	\draw[thick, rho] (0,-.4) -- (0, .4);
	\nbox{}{(0,0)}{0}{0}{}
\end{tikzpicture}\,.
$$
The second, third, and fourth equations are straightforward. To prove Equation \eqref{rel:AT1}, we note that by the first four equations, the right hand side is the sum of two orthogonal projections which are isomorphic to $\jw{0}$ and $\jw{2}$, and both are dominated by the left hand side. Since in $A_4$, $\jw{2}\otimes \jw{2}\cong \jw{0}\oplus \jw{2}$, we are finished.
\end{proof}

The distinct simple objects of $A_2*T_2$ are all alternating words in $\rho,\theta$. Hence representatives of the isomorphism classes of simples are given by
$$
\underset{1}{
\begin{tikzpicture}[baseline = -.1cm]
	\draw[thick] (-.4,-.4)--(-.4,.4)--(.4,.4)--(.4,-.4)--(-.4,-.4);
\end{tikzpicture}
}
\,,\,
\underset{\rho}{
\begin{tikzpicture}[baseline = -.1cm]
	\draw[thick, rho] (0,-.4) -- ( 0, .4);	
	\draw[thick] (-.4,-.4)--(-.4,.4)--(.4,.4)--(.4,-.4)--(-.4,-.4);
\end{tikzpicture}
}
\,,\,
\underset{\theta}{
\begin{tikzpicture}[baseline = -.1cm]
	\draw[thick, theta] (0,-.4) -- ( 0, .4);	
	\draw[thick] (-.4,-.4)--(-.4,.4)--(.4,.4)--(.4,-.4)--(-.4,-.4);
\end{tikzpicture}
}
\,,\,
\underset{\rho\theta}{
\begin{tikzpicture}[baseline = -.1cm]
	\draw[thick, rho] (-.1,-.4) -- ( -.1, .4);	
	\draw[thick, theta] (.1,-.4) -- ( .1, .4);	
	\draw[thick] (-.4,-.4)--(-.4,.4)--(.4,.4)--(.4,-.4)--(-.4,-.4);
\end{tikzpicture}
}
\,,\,
\underset{\theta\rho}{
\begin{tikzpicture}[baseline = -.1cm]
	\draw[thick, theta] (-.1,-.4) -- ( -.1, .4);	
	\draw[thick, rho] (.1,-.4) -- ( .1, .4);	
	\draw[thick] (-.4,-.4)--(-.4,.4)--(.4,.4)--(.4,-.4)--(-.4,-.4);
\end{tikzpicture}
}
\,,\,
\underset{\rho\theta\rho}{
\begin{tikzpicture}[baseline = -.1cm]
	\draw[thick, rho] (-.2,-.4) -- (-.2, .4);	
	\draw[thick, theta] (0,-.4) -- ( 0, .4);	
	\draw[thick, rho] (.2,-.4) -- (.2, .4);
	\draw[thick] (-.4,-.4)--(-.4,.4)--(.4,.4)--(.4,-.4)--(-.4,-.4);
\end{tikzpicture}
}
\,,\,
\underset{\theta\rho\theta}{
\begin{tikzpicture}[baseline = -.1cm]
	\draw[thick, theta] (-.2,-.4) -- (-.2, .4);	
	\draw[thick, rho] (0,-.4) -- ( 0, .4);	
	\draw[thick, theta] (.2,-.4) -- (.2, .4);
	\draw[thick] (-.4,-.4)--(-.4,.4)--(.4,.4)--(.4,-.4)--(-.4,-.4);
\end{tikzpicture}
}
\,,\dots
$$

\subsection{Quotients of $A_2*T_2$}\label{sec:ATQuotient}

We now show all unitary quotients of $A_2*T_2$ are parametrized by an $n\in \Natural$ and an $n$-th root of unity $\omega$. 
Suppose we are working in some unitary quotient of $A_2*T_2$.

\begin{prop}\label{prop:DistinctIrreducibleBimodules}
Suppose that $(\rho\theta)^k\ncong (\theta\rho)^k$ for some $k\geq 1$. Then all alternating words in $\rho,\theta$ with length less than or equal to $2k+2$ give distinct simple objects, except that  $(\rho\theta)^{k+1}$ may not be distinct from $(\theta\rho)^{k+1}$.
\end{prop}
\begin{proof}
We induct on $k$. If $k=1$, then it is a straightforward calculation using Frobenius reciprocity (which holds in the unitary quotient) to show that 
$1,\rho,\theta,\rho\theta,\theta\rho,\rho\theta\rho,\theta\rho\theta,\rho\theta\rho\theta$ 
are distinct and simple. 
For example once one shows $\rho\theta$ and $\theta\rho$ are irreducible, we have 
$$
\langle \rho\theta\rho\theta,\rho\theta\rho\theta\rangle
=
\langle \rho\theta\rho,\rho\theta\rho\rangle
=
\langle \rho\theta,\rho\theta\rangle
+
\langle \rho\theta\rho,\rho\theta\rangle
=
1
+
\langle \theta\rho,\rho\theta\rangle
+
\langle \theta\rho,\theta\rangle
=
1.
$$
A similar calculation shows $\theta\rho\theta\rho$ is simple, but note that we cannot yet compute $\langle \rho\theta\rho\theta,\theta\rho\theta\rho\rangle$.

Suppose the result holds true for $k>1$, and suppose we also know that $(\rho\theta)^{k+1}\neq (\theta\rho)^{k+1}$. Then we calculate:
\begin{align*}
\langle (\rho\theta)^k\rho,(\rho\theta)^k\rangle 
& = \langle (\theta\rho)^{k},(1\oplus \rho)(\theta\rho)^{k-1}\theta\rangle \\
& = \langle (\theta\rho)^{k},(\theta\rho)^{k-1}\theta\rangle + \langle (\theta\rho)^{k},(\rho\theta)^{k}\rangle \\
& = 0 \displaybreak[1]
\\
\langle (\rho\theta)^k\rho,(\rho\theta)^k\rho\rangle 
& = \langle (\rho\theta)^k(1\oplus \rho),(\rho\theta)^k\rangle 
= \langle (\rho\theta)^k,(\rho\theta)^k\rangle  + \langle (\rho\theta)^k \rho,(\rho\theta)^k\rangle 
= 1
\\
\langle (\theta\rho)^k\theta,(\theta\rho)^k\rangle
& = \langle (\rho\theta)^k,(\rho\theta)^{k-1}\theta\rangle
= 0
\\
\langle (\theta\rho)^k\theta,(\theta\rho)^k\theta\rangle
& = \langle (\theta\rho)^k,(\theta\rho)^k\rangle
= 1
\\
\langle (\rho\theta)^{k}\rho,(\theta\rho)^{k}\theta\rangle
& = 0 
\tag{simples with different dimensions}
\\
\langle (\rho\theta)^{k+1},(\rho\theta)^k\rho\rangle
& = \langle (\theta\rho)^{k}\theta,(1\oplus \rho)(\theta\rho)^k\rangle\\
& = \langle (\theta\rho)^{k}\theta,(\theta\rho)^k\rangle + \langle (\theta\rho)^{k}\theta,(\rho\theta)^{k}\rho\rangle\\
& = 0
\end{align*}
and so forth. We see that $(\rho\theta)^k\rho,(\theta\rho)^k\theta,(\rho\theta)^{k+1},(\theta\rho)^{k+1}$ are all distinct and simple, except that possibly $(\rho\theta)^{k+1}\cong(\theta\rho)^{k+1}$.
\end{proof}

\begin{cor}\label{cor:UExists}
Either there is an $n\in\Natural$ such that $(\rho\theta)^n\cong(\theta\rho)^n$, or there is no such $n$.
In either case, we know all the distinct simples generated by $\theta,\rho$.
\end{cor}


If there is an $n\in\Natural$ as in Corollary \ref{cor:UExists}, there is a unitary isomorphism $U\colon (\rho\theta)^n\to (\theta\rho)^n$. Let $\zeta = (\theta\rho)^{n-1}\theta$. 
We denote $\zeta$ by a \textcolor{zeta}{purple} strand, and we denote the isomorphism $U$ as follows:
$$
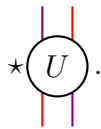
\begin{tikzpicture}[baseline = -.1cm]
	\draw[thick, rho] (-.2,-.8) --(-.2,0)--(.2,0)-- (.2, .8);	
	\draw[thick, zeta] (.2,-.8) --(.2,0)--(-.2,0)-- (-.2, .8);
	\draw[thick, unshaded] (0,0) circle (.4cm);
	\node at (0,0) {$U$};
	\node at (-.55,0) {$\star$};
\end{tikzpicture}\,.
$$
Since $U^*U= \id_{\zeta\rho}$ and $UU^*= \id_{\rho\zeta}$, we immediately obtain:
\begin{equation}
\begin{tikzpicture}[baseline = -.1cm]
	\draw[thick, zeta] (-.4,.2) -- ( .4, .2);	
	\draw[thick, rho] (-.4,-.2) -- (.4, -.2);
	\draw[thick] (-.4,-.4)--(-.4,.4)--(.4,.4)--(.4,-.4)--(-.4,-.4);
\end{tikzpicture}
=
\begin{tikzpicture}[baseline = -.1cm]
	\draw[thick, rho] (-1.4,-.4)-- (-.6,0) .. controls ++(45:.6cm) and ++(135:.6cm) .. (.6,0)--(1.4,-.4);	
	\draw[thick, zeta] (-1.4,0) -- (1.4, 0);
	\draw[thick, unshaded] (.6,0) circle (.4cm);
	\node at (.6,0) {$U$};
	\node at (.6,.55) {$\star$};
	\draw[thick, unshaded] (-.6,0) circle (.4cm);
	\node at (-.6,0) {$U^*$};
	\node at (-.6,.55) {$\star$};
\end{tikzpicture}
\text{ and }
\begin{tikzpicture}[baseline = -.1cm]
	\draw[thick, rho] (-.4,.2) -- ( .4, .2);	
	\draw[thick, zeta] (-.4,-.2) -- (.4, -.2);
	\draw[thick] (-.4,-.4)--(-.4,.4)--(.4,.4)--(.4,-.4)--(-.4,-.4);
\end{tikzpicture}
=
\begin{tikzpicture}[baseline = -.1cm]
	\draw[thick, zeta] (-1.4,-.4)-- (-.6,0) .. controls ++(45:.6cm) and ++(135:.6cm) .. (.6,0)--(1.4,-.4);	
	\draw[thick, rho] (-1.4,0) -- (1.4, 0);
	\draw[thick, unshaded] (.6,0) circle (.4cm);
	\node at (.6,0) {$U^*$};
	\node at (.6,.55) {$\star$};
	\draw[thick, unshaded] (-.6,0) circle (.4cm);
	\node at (-.6,0) {$U$};
	\node at (-.6,.55) {$\star$};
\end{tikzpicture}\,.
\label{rel:AT2}
\tag{AT2}
\end{equation}
We may normalize by a phase so that
\begin{equation}
\cF^{-1}(U)=
\begin{tikzpicture}[baseline = -.1 cm]
	\draw[thick, rho] (-.2,-.8) -- (-.2,.8);
	\draw[thick, rho] (.2,-.8) -- (.2,.8);	
	\node at (-.1,.6) {\scriptsize{{$\cdot$}}};
	\node at (0,.6) {\scriptsize{{$\cdot$}}};
	\node at (.1,.6) {\scriptsize{{$\cdot$}}};
	\node at (-.1,-.6) {\scriptsize{{$\cdot$}}};
	\node at (0,-.6) {\scriptsize{{$\cdot$}}};
	\node at (.1,-.6) {\scriptsize{{$\cdot$}}};
	\node [rotate around = {90:(0,0)}] at (0,.95) {$\}$};
	\node [rotate around = {90:(0,0)}] at (0,-.95) {$\{$};
	\node at (.25,1.2) {\scriptsize{{$\rho(\theta\rho)^{n-1}$}}};
	\node at (.25,-1.2) {\scriptsize{{$\rho(\theta\rho)^{n-1}$}}};
	\draw[thick, theta] (-.3,.27) arc(0:180:.2cm) --(-.7,-.8);
	\draw[thick, theta] (.3,-.27) arc(-180:0:.2cm) --(.7,.8);	
	\draw[thick, unshaded] (0,0) circle (.4cm);
	\node at (0,0) {$U$};
	\node at (-.55,0) {$\star$};
\end{tikzpicture}
=
\begin{tikzpicture}[baseline = -.1cm]
	\draw[thick, zeta] (-.2,-.8) --(-.2,0)--(.2,0)-- (.2, .8);	
	\draw[thick, rho] (.2,-.8) --(.2,0)--(-.2,0)-- (-.2, .8);
	\draw[thick, unshaded] (0,0) circle (.4cm);
	\node at (0,0) {$U^*$};
	\node at (-.55,0) {$\star$};
\end{tikzpicture}
=
\omega_U^{-1}
\begin{tikzpicture}[baseline = -.1cm]
	\draw[thick, rho] (-.2,-.35) arc(0:-180:.2cm) --(-.6,.8);
	\draw[thick, rho] (.2,.35) arc(180:0:.2cm) --(.6,-.8);	
	\draw[thick, zeta] (.2,-.8) --(.2,0)--(-.2,0)-- (-.2, .8);
	\draw[thick, unshaded] (0,0) circle (.4cm);
	\node at (0,0) {$U$};
	\node at (135:.55cm) {$\star$};
\end{tikzpicture}
=\omega_U^{-1}\cF(U)
\label{rel:AT3}
\tag{AT3}
\end{equation}
for some $2n$-th root of unity $\omega_U$.

\begin{remark}\label{rem:A2T2StarStructure}
As in Remark \ref{rem:A2A2StarStructure}, not assuming unitarity, there are exactly 2 ways to put a $*$-structure on the quotient. However, we only consider the unitary case.
\end{remark}

\subsection{Jellyfish relations}\label{sec:Jellyfish}

We now derive jellyfish relations for our generator $U$.

\begin{lem}\label{lem:RhoSpaghettiJellyfish}
$
\begin{tikzpicture}[baseline = -.1cm]
	\draw[thick, zeta] (-.4,.2) -- (.4,.2);
	\filldraw[rho] (0,0) circle (.05cm);	
	\draw[thick, rho] (0,0) -- (-.4,-.4);
	\draw[thick, rho] (0,0) -- (0,-.4);
	\draw[thick, rho] (0,0) -- (.4,-.4);
	\draw[thick] (-.4,-.4)--(-.4,.4)--(.4,.4)--(.4,-.4)--(-.4,-.4);
\end{tikzpicture}
=
\sigma_U
\,
\begin{tikzpicture}[baseline = -.1cm]
	\draw[thick, zeta] (-1.8,0) -- (1.8,0);	
	\coordinate (a) at (0,.9);
	\filldraw[rho] (a) circle (.05cm);
	\draw[thick, rho] (a) -- (-1.8,-.6);
	\draw[thick, rho] (a) -- (0,-.6);
	\draw[thick, rho] (a) -- (1.8,-.6);
	\draw[thick, unshaded] (-1,0) circle (.4cm);
	\node at (-1,0) {$U^*$};
	\node at (-1,.55) {$\star$};
	\draw[thick, unshaded]  (0,0) circle (.4cm);
	\node at (0,0) {$U^*$};
	\node at (115:.55cm) {$\star$};
	\draw[thick, unshaded]  (1,0) circle (.4cm);
	\node at (1,0) {$U$};
	\node at (1,.55) {$\star$};
\end{tikzpicture}
$
where $\sigma_U$ is a choice of square root of $\omega_U$.
\end{lem}


\begin{remark}\label{rem:SwitchSign}
Note that switching the sign of $U$ switches the sign of $\sigma_U$.
\end{remark}


\begin{proof}
Note that $\langle \zeta \rho , \zeta \rho^2\rangle = 1$, so since both diagrams have the same nonzero norm $\tau\dim(\zeta)$ (using Equation \eqref{rel:AT2}), there is a unimodular scalar $\lambda\in \mathbb{T}$ such that the diagram on the left is equal to $\lambda$ times the diagram on the right. It remains to determine the scalar $\lambda$.

We can write 
$
\begin{tikzpicture}[baseline = -.1cm]
	\draw[thick, rho] (-.4,.2) -- (.4, .2);
	\draw[thick, zeta] (-.4,0) -- ( .4, 0);	
	\draw[thick, rho] (-.4,-.2) -- (.4, -.2);
	\draw[thick] (-.4,-.4)--(-.4,.4)--(.4,.4)--(.4,-.4)--(-.4,-.4);
\end{tikzpicture}
$
in two different ways:
\begin{align*}
\begin{tikzpicture}[baseline = -.1cm]
	\draw[thick, rho] (-.4,.2) -- (.4, .2);
	\draw[thick, zeta] (-.4,0) -- ( .4, 0);	
	\draw[thick, rho] (-.4,-.2) -- (.4, -.2);
	\draw[thick] (-.4,-.4)--(-.4,.4)--(.4,.4)--(.4,-.4)--(-.4,-.4);
\end{tikzpicture}
&=
\begin{tikzpicture}[baseline = .1cm]
	\draw[thick, rho] (-1.4,.7) -- (1.4, .7);
	\draw[thick, rho] (-1.4,-.4)-- (-.6,0) .. controls ++(45:.6cm) and ++(135:.6cm) .. (.6,0)--(1.4,-.4);	
	\draw[thick, zeta] (-1.4,0) -- (1.4, 0);
	\draw[thick, unshaded] (.6,0) circle (.4cm);
	\node at (.6,0) {$U$};
	\node at (.6,.55) {$\star$};
	\draw[thick, unshaded] (-.6,0) circle (.4cm);
	\node at (-.6,0) {$U^*$};
	\node at (-.6,.55) {$\star$};
\end{tikzpicture}
=
\frac{1}{\tau}\,
\begin{tikzpicture}[baseline = .1cm]
	\draw[thick, rho] (-1.4,-.4)-- (-.6,0) .. controls ++(45:.6cm) and ++(0:.6cm) .. (-.6,.7)--(-1.4,.7);
	\draw[thick, rho] (1.4,-.4)-- (.6,0) .. controls ++(135:.6cm) and ++(180:.6cm) .. (.6,.7)--(1.4,.7);	
	\draw[thick, zeta] (-1.4,0) -- (1.4, 0);
	\draw[thick, unshaded] (.6,0) circle (.4cm);
	\node at (.6,0) {$U$};
	\node at (.6,.55) {$\star$};
	\draw[thick, unshaded] (-.6,0) circle (.4cm);
	\node at (-.6,0) {$U^*$};
	\node at (-.6,.55) {$\star$};
\end{tikzpicture}
+
\begin{tikzpicture}[baseline = .1cm]
	\draw[thick, rho] (-1.4,.7) -- (1.4, .7);
	\filldraw[rho] (-.2,.7) circle (.05cm);
	\filldraw[rho] (.2,.7) circle (.05cm);
	\draw[thick, rho] (-1.4,-.4)-- (-.6,0) .. controls ++(45:.3cm) and ++(270:.3cm) .. (-.2,.7);
	\draw[thick, rho] (1.4,-.4)-- (.6,0) .. controls ++(135:.3cm) and ++(270:.3cm) .. (.2,.7);
	\draw[thick, zeta] (-1.4,0) -- (1.4, 0);
	\draw[thick, unshaded] (.6,0) circle (.4cm);
	\node at (.6,0) {$U$};
	\node at (.6,.55) {$\star$};
	\draw[thick, unshaded] (-.6,0) circle (.4cm);
	\node at (-.6,0) {$U^*$};
	\node at (-.6,.55) {$\star$};
\end{tikzpicture}
\text{ and}\displaybreak[1]\\
\begin{tikzpicture}[baseline = -.1cm]
	\draw[thick, rho] (-.4,.2) -- (.4, .2);
	\draw[thick, zeta] (-.4,0) -- ( .4, 0);	
	\draw[thick, rho] (-.4,-.2) -- (.4, -.2);
	\draw[thick] (-.4,-.4)--(-.4,.4)--(.4,.4)--(.4,-.4)--(-.4,-.4);
\end{tikzpicture}
&=
\begin{tikzpicture}[baseline = -.1cm,yscale=-1]
	\draw[thick, rho] (-1.4,.7) -- (1.4, .7);
	\draw[thick, rho] (-1.4,-.4)-- (-.6,0) .. controls ++(45:.6cm) and ++(135:.6cm) .. (.6,0)--(1.4,-.4);	
	\draw[thick, zeta] (-1.4,0) -- (1.4, 0);
	\draw[thick, unshaded] (.6,0) circle (.4cm);
	\node at (.6,0) {$U^*$};
	\node at (.6,.55) {$\star$};
	\draw[thick, unshaded] (-.6,0) circle (.4cm);
	\node at (-.6,0) {$U$};
	\node at (-.6,.55) {$\star$};
\end{tikzpicture}
=
\begin{tikzpicture}[baseline = -.1cm,yscale=-1]
	\draw[thick, rho] (-1.4,.6) -- (1.4, .6);
	\draw[thick, rho] (-1.4,-.4)-- (-.6,0) .. controls ++(45:.6cm) and ++(135:.6cm) .. (.6,0)--(1.4,-.4);	
	\draw[thick, zeta] (-1.4,0) -- (1.4, 0);
	\draw[thick, unshaded] (.6,0) circle (.4cm);
	\node at (.6,0) {$U^*$};
	\node at (.6,-.55) {$\star$};
	\draw[thick, unshaded] (-.6,0) circle (.4cm);
	\node at (-.6,0) {$U$};
	\node at (-.6,-.55) {$\star$};
\end{tikzpicture}
=
\frac{1}{\tau}\,
\begin{tikzpicture}[baseline = -.1cm, yscale=-1]
	\draw[thick, rho] (-1.4,-.4)-- (-.6,0) .. controls ++(45:.6cm) and ++(0:.6cm) .. (-.6,.6)--(-1.4,.6);
	\draw[thick, rho] (1.4,-.4)-- (.6,0) .. controls ++(135:.6cm) and ++(180:.6cm) .. (.6,.6)--(1.4,.6);	
	\draw[thick, zeta] (-1.4,0) -- (1.4, 0);
	\draw[thick, unshaded] (.6,0) circle (.4cm);
	\node at (.6,0) {$U^*$};
	\node at (.6,-.55) {$\star$};
	\draw[thick, unshaded] (-.6,0) circle (.4cm);
	\node at (-.6,0) {$U$};
	\node at (-.6,-.55) {$\star$};
\end{tikzpicture}
+
\begin{tikzpicture}[baseline = -.1cm,yscale=-1]
	\draw[thick, rho] (-1.4,.6) -- (1.4, .6);
	\filldraw[rho] (-.2,.6) circle (.05cm);
	\filldraw[rho] (.2,.6) circle (.05cm);
	\draw[thick, rho] (-1.4,-.4)-- (-.6,0) .. controls ++(45:.3cm) and ++(270:.3cm) .. (-.2,.6);
	\draw[thick, rho] (1.4,-.4)-- (.6,0) .. controls ++(135:.3cm) and ++(270:.3cm) .. (.2,.6);
	\draw[thick, zeta] (-1.4,0) -- (1.4, 0);
	\draw[thick, unshaded] (.6,0) circle (.4cm);
	\node at (.6,0) {$U^*$};
	\node at (.6,-.55) {$\star$};
	\draw[thick, unshaded] (-.6,0) circle (.4cm);
	\node at (-.6,0) {$U$};
	\node at (-.6,-.55) {$\star$};
\end{tikzpicture}
\,.
\end{align*}
Note that the first diagram in the sum on the right hand side for the first expression is equal to the first diagram in the sum of the right hand side for the second expression. Hence the second diagram in the sum on the right hand side for the first expression must also equal the second diagram in the sum of the right hand side for the second expression. We have
\begin{align*}
\begin{tikzpicture}[baseline = -.1cm,yscale=-1]
	\draw[thick, rho] (-1.4,.6) -- (1.4, .6);
	\filldraw[rho] (-.2,.6) circle (.05cm);
	\filldraw[rho] (.2,.6) circle (.05cm);
	\draw[thick, rho] (-1.4,-.4)-- (-.6,0) .. controls ++(45:.3cm) and ++(270:.3cm) .. (-.2,.6);
	\draw[thick, rho] (1.4,-.4)-- (.6,0) .. controls ++(135:.3cm) and ++(270:.3cm) .. (.2,.6);
	\draw[thick, zeta] (-1.4,0) -- (1.4, 0);
	\draw[thick, unshaded] (.6,0) circle (.4cm);
	\node at (.6,0) {$U^*$};
	\node at (.6,-.55) {$\star$};
	\draw[thick, unshaded] (-.6,0) circle (.4cm);
	\node at (-.6,0) {$U$};
	\node at (-.6,-.55) {$\star$};
\end{tikzpicture}
&=
\lambda^2
\begin{tikzpicture}[baseline = -.1cm]
	\draw[thick, zeta] (-3.8,0) -- (4.8,0);
	\coordinate (a) at (-1,1) ;
	\coordinate (b) at (2,1) ;	
	\filldraw[rho] (a) circle (.05cm);
	\filldraw[rho] (b) circle (.05cm);
	\draw[thick, rho]  (-3.8,.4) -- (-3,0) .. controls ++(-45:.6cm) and ++(225:.6cm) .. (-2,0) -- (a);
	\draw[thick, rho]  (b) -- (3,0) .. controls ++(-45:.6cm) and ++(225:.6cm) .. (4,0) -- (4.8,.4);
	\draw[thick, rho]  (a) -- (0,0) .. controls ++(-45:.6cm) and ++(225:.6cm) .. (1,0) -- (b);
	\draw[thick, rho] (a) -- (-1,-.4) arc (0:-90:.2cm)--(-3.8,-.6);
	\draw[thick, rho] (b) -- (2,-.4) arc (180:270:.2cm)--(4.8,-.6);
	\Ucircle{U}{(-3,0)}
	\Ucircle{U^*}{(-2,0)}
	\draw[thick, unshaded] (-1,0) circle (.4cm);
	\node at (-1,0) {$U^*$};
	\node at (-1.2,.55) {$\star$};
	\Ucircle{U}{(0,0)}
	\Ucircle{U^*}{(1,0)}
	\draw[thick, unshaded] (2,0) circle (.4cm);
	\node at (2,0) {$U^*$};
	\node at (1.8,.55) {$\star$};
	\Ucircle{U}{(3,0)}
	\Ucircle{U^*}{(4,0)}
\end{tikzpicture}
\displaybreak[1]\\
&=
\lambda^2\,
\begin{tikzpicture}[baseline = -.1 cm]
	\draw[thick, rho] (-1.4,.7) -- (1.4, .7);
	\filldraw[rho] (-.2,.7) circle (.05cm);
	\filldraw[rho] (.2,.7) circle (.05cm);
	\draw[thick, rho] (-1.4,-.4)-- (-.6,0) .. controls ++(45:.3cm) and ++(270:.3cm) .. (-.2,.7);
	\draw[thick, rho] (1.4,-.4)-- (.6,0) .. controls ++(135:.3cm) and ++(270:.3cm) .. (.2,.7);
	\draw[thick, zeta] (-1.4,0) -- (1.4, 0);
	\draw[thick, unshaded] (.6,0) circle (.4cm);
	\node at (.6,0) {$U^*$};
	\node at (.1,.25) {$\star$};
	\draw[thick, unshaded] (-.6,0) circle (.4cm);
	\node at (-.6,0) {$U^*$};
	\node at (-.6,.55) {$\star$};
\end{tikzpicture}
\displaybreak[1]\\
&=
\lambda^2\omega_U^{-1}\,
\begin{tikzpicture}[baseline = -.1 cm]
	\draw[thick, rho] (-1.4,.7) -- (1.4, .7);
	\filldraw[rho] (-.2,.7) circle (.05cm);
	\filldraw[rho] (.2,.7) circle (.05cm);
	\draw[thick, rho] (-1.4,-.4)-- (-.6,0) .. controls ++(45:.3cm) and ++(270:.3cm) .. (-.2,.7);
	\draw[thick, rho] (1.4,-.4)-- (.6,0) .. controls ++(135:.3cm) and ++(270:.3cm) .. (.2,.7);
	\draw[thick, zeta] (-1.4,0) -- (1.4, 0);
	\draw[thick, unshaded] (.6,0) circle (.4cm);
	\node at (.6,0) {$U$};
	\node at (.6,.55) {$\star$};
	\draw[thick, unshaded] (-.6,0) circle (.4cm);
	\node at (-.6,0) {$U^*$};
	\node at (-.6,.55) {$\star$};
\end{tikzpicture}
\,.
\end{align*}
Hence $\lambda^2=\omega_U$.
\end{proof}


\begin{thm}\label{thm:JellyfishRelationsU}
The following jellyfish relations hold for $U$:
\begin{enumerate}[(1)]
\item
$
\begin{tikzpicture}[baseline = -.3cm]
	\draw[thick, theta] (-.7,-.8)-- (-.7,0) arc(180:0:.7cm) -- (.7,-.8);
	\draw[thick, rho] (-.3,0) -- (-.3,-.8);
	\draw[thick, zeta] (-.1,0) -- (-.1,-.8);
	\draw[thick, rho] (.1,0) -- (.1,-.8);
	\draw[thick, zeta] (.3,0) -- (.3,-.8);
	\draw[thick, unshaded]  (0,0) circle (.4cm);
	\node at (0,0) {$U$};
	\node at (-.55,0) {$\star$};
\end{tikzpicture}
=
\begin{tikzpicture}[baseline = -.3cm]
	\draw[thick, zeta] (-.3,0) -- (-.3,-.8);
	\draw[thick, rho] (-.1,0) -- (-.1,-.8);
	\draw[thick, zeta] (.1,0) -- (.1,-.8);
	\draw[thick, rho] (.3,0) -- (.3,-.8);
	\draw[thick, theta] (.5,-.8) arc (180:0:.2cm);
	\draw[thick, unshaded]  (0,0) circle (.4cm);
	\node at (0,0) {$U^*$};
	\node at (-.55,0) {$\star$};
\end{tikzpicture}
$
\item
$
\begin{tikzpicture}[baseline = -.3cm]
	\draw[thick, rho] (-.7,-.8)-- (-.7,0) arc(180:0:.7cm) -- (.7,-.8);
	\coordinate (a) at (0,.7);
	\filldraw[rho] (a) circle (.05cm);
	\draw[thick, rho] (a)--(0,0);
	\draw[thick, zeta] (-.2,0) -- (-.2,-.8);
	\draw[thick, rho] (0,0) -- (0,-.8);
	\draw[thick, zeta] (.2,0) -- (.2,-.8);
	\draw[thick, unshaded]  (0,0) circle (.4cm);
	\node at (0,0) {$U$};
	\node at (45:.55cm) {$\star$};
\end{tikzpicture}
=
\omega_U\,
\begin{tikzpicture}[baseline = -.3cm]
	\draw[thick, rho] (-.7,-.8)-- (-.7,0) arc(180:0:.7cm) -- (.7,-.8);
	\coordinate (a) at (0,.7);
	\filldraw[rho] (a) circle (.05cm);
	\draw[thick, rho] (a)--(0,0);
	\draw[thick, zeta] (-.2,0) -- (-.2,-.8);
	\draw[thick, rho] (0,0) -- (0,-.8);
	\draw[thick, zeta] (.2,0) -- (.2,-.8);
	\draw[thick, unshaded]  (0,0) circle (.4cm);
	\node at (0,0) {$U^*$};
	\node at (135:.55cm) {$\star$};
\end{tikzpicture}
=
\sigma_U^{-1}\,
\begin{tikzpicture}[baseline = -.3cm]
	\draw[thick, zeta] (-.6,0) -- (.6,0);
	\coordinate (a) at (0,-.5);	
	\filldraw[rho] (a) circle (.05cm);
	\draw[thick, rho] (-.6,0)--(a)--(.6,0);
	\draw[thick, rho] (a)--(0,-.8);
	\draw[thick, rho] (-.7,0) -- (-.7,-.8);
	\draw[thick, zeta] (-.5,0) -- (-.5,-.8);
	\draw[thick, zeta] (.5,0) -- (.5,-.8);
	\draw[thick, rho] (.7,0) -- (.7,-.8);
	\draw[thick, unshaded]  (-.6,0) circle (.4cm);
	\node at (-.6,0) {$U$};
	\node at (-.6,.55) {$\star$};
	\draw[thick, unshaded]  (.6,0) circle (.4cm);
	\node at (.6,0) {$U^*$};
	\node at (.6,.55) {$\star$};
\end{tikzpicture}
$
\item
$\displaystyle
\begin{tikzpicture}[baseline = -.3cm]
	\draw[thick, rho] (-.7,-.8)-- (-.7,0) arc(180:0:.7cm) -- (.7,-.8);
	\draw[thick, rho] (-.3,0) -- (-.3,-.8);
	\draw[thick, zeta] (-.1,0) -- (-.1,-.8);
	\draw[thick, rho] (.1,0) -- (.1,-.8);
	\draw[thick, zeta] (.3,0) -- (.3,-.8);
	\draw[thick, unshaded]  (0,0) circle (.4cm);
	\node at (0,0) {$U$};
	\node at (-.55,0) {$\star$};
\end{tikzpicture}
=
\frac{\omega_U}{\tau}\,
\begin{tikzpicture}[baseline = -.3cm]
	\draw[thick, zeta] (-.3,0) -- (-.3,-.8);
	\draw[thick, rho] (-.1,0) -- (-.1,-.8);
	\draw[thick, zeta] (.1,0) -- (.1,-.8);
	\draw[thick, rho] (.3,0) -- (.3,-.8);
	\draw[thick, rho] (-.9,-.8) arc (180:0:.2cm);
	\draw[thick, unshaded]  (0,0) circle (.4cm);
	\node at (0,0) {$U^*$};
	\node at (-.55,0) {$\star$};
\end{tikzpicture}
+
\sigma_U^{-1}\,
\begin{tikzpicture}[baseline = -.3cm]
	\draw[thick, zeta] (-.6,0) -- (.6,0);
	\coordinate (a) at (0,-.5);	
	\coordinate (b) at (-1,-.5);	
	\filldraw[rho] (a) circle (.05cm);
	\filldraw[rho] (b) circle (.05cm);
	\draw[thick, rho] (-.6,0)--(a)--(.6,0);
	\draw[thick, rho] (a)--(0,-.8);
	\draw[thick, rho] (-1.2,-.8)--(b)--(-.8,-.8);
	\draw[thick, rho] (b)--(-.6,0);
	\draw[thick, zeta] (-.6,0) -- (-.6,-.8);
	\draw[thick, zeta] (.5,0) -- (.5,-.8);
	\draw[thick, rho] (.7,0) -- (.7,-.8);
	\draw[thick, unshaded]  (-.6,0) circle (.4cm);
	\node at (-.6,0) {$U$};
	\node at (-.6,.55) {$\star$};
	\draw[thick, unshaded]  (.6,0) circle (.4cm);
	\node at (.6,0) {$U^*$};
	\node at (.6,.55) {$\star$};
\end{tikzpicture}
$
\end{enumerate}
\end{thm}
\begin{proof}
(1) follows from Relation \eqref{rel:AA1}. (2) follows Lemma \ref{lem:RhoSpaghettiJellyfish} by multiplying on the left by $U$ and on the right by $U^*$. (3) follows from (2) and the relations in Proposition \ref{prop:RhoSkeinRelations}.
\end{proof}


\begin{cor}\label{cor:JellyfishRelationsUstar}
By taking adjoints in Theorem \ref{thm:JellyfishRelationsU}, we get the following jellyfish relations for $U^*$: 
\begin{enumerate}[(1)]
\item
$
\begin{tikzpicture}[baseline = -.3cm,xscale=-1]
	\draw[thick, theta] (-.7,-.8)-- (-.7,0) arc(180:0:.7cm) -- (.7,-.8);
	\draw[thick, rho] (-.3,0) -- (-.3,-.8);
	\draw[thick, zeta] (-.1,0) -- (-.1,-.8);
	\draw[thick, rho] (.1,0) -- (.1,-.8);
	\draw[thick, zeta] (.3,0) -- (.3,-.8);
	\draw[thick, unshaded]  (0,0) circle (.4cm);
	\node at (0,0) {$U^*$};
	\node at (.55,0) {$\star$};
\end{tikzpicture}
=
\begin{tikzpicture}[baseline = -.3cm,xscale=-1]
	\draw[thick, zeta] (-.3,0) -- (-.3,-.8);
	\draw[thick, rho] (-.1,0) -- (-.1,-.8);
	\draw[thick, zeta] (.1,0) -- (.1,-.8);
	\draw[thick, rho] (.3,0) -- (.3,-.8);
	\draw[thick, theta] (.5,-.8) arc (180:0:.2cm);
	\draw[thick, unshaded]  (0,0) circle (.4cm);
	\node at (0,0) {$U$};
	\node at (.55,0) {$\star$};
\end{tikzpicture}
$
\item
$
\begin{tikzpicture}[baseline = -.3cm,xscale=-1]
	\draw[thick, rho] (-.7,-.8)-- (-.7,0) arc(180:0:.7cm) -- (.7,-.8);
	\coordinate (a) at (0,.7);
	\filldraw[rho] (a) circle (.05cm);
	\draw[thick, rho] (a)--(0,0);
	\draw[thick, zeta] (-.2,0) -- (-.2,-.8);
	\draw[thick, rho] (0,0) -- (0,-.8);
	\draw[thick, zeta] (.2,0) -- (.2,-.8);
	\draw[thick, unshaded]  (0,0) circle (.4cm);
	\node at (0,0) {$U^*$};
	\node at (45:.55cm) {$\star$};
\end{tikzpicture}
=
\sigma_U\,
\begin{tikzpicture}[baseline = -.3cm,xscale=-1]
	\draw[thick, zeta] (-.6,0) -- (.6,0);
	\coordinate (a) at (0,-.5);	
	\filldraw[rho] (a) circle (.05cm);
	\draw[thick, rho] (-.6,0)--(a)--(.6,0);
	\draw[thick, rho] (a)--(0,-.8);
	\draw[thick, rho] (-.7,0) -- (-.7,-.8);
	\draw[thick, zeta] (-.5,0) -- (-.5,-.8);
	\draw[thick, zeta] (.5,0) -- (.5,-.8);
	\draw[thick, rho] (.7,0) -- (.7,-.8);
	\draw[thick, unshaded]  (-.6,0) circle (.4cm);
	\node at (-.6,0) {$U^*$};
	\node at (-.6,.55) {$\star$};
	\draw[thick, unshaded]  (.6,0) circle (.4cm);
	\node at (.6,0) {$U$};
	\node at (.6,.55) {$\star$};
\end{tikzpicture}
$
\item
$\displaystyle
\begin{tikzpicture}[baseline = -.3cm,xscale=-1]
	\draw[thick, rho] (-.7,-.8)-- (-.7,0) arc(180:0:.7cm) -- (.7,-.8);
	\draw[thick, rho] (-.3,0) -- (-.3,-.8);
	\draw[thick, zeta] (-.1,0) -- (-.1,-.8);
	\draw[thick, rho] (.1,0) -- (.1,-.8);
	\draw[thick, zeta] (.3,0) -- (.3,-.8);
	\draw[thick, unshaded]  (0,0) circle (.4cm);
	\node at (0,0) {$U^*$};
	\node at (.55,0) {$\star$};
\end{tikzpicture}
=
\frac{\omega_U^{-1}}{\tau}\,
\begin{tikzpicture}[baseline = -.3cm,xscale=-1]
	\draw[thick, zeta] (-.3,0) -- (-.3,-.8);
	\draw[thick, rho] (-.1,0) -- (-.1,-.8);
	\draw[thick, zeta] (.1,0) -- (.1,-.8);
	\draw[thick, rho] (.3,0) -- (.3,-.8);
	\draw[thick, rho] (-.9,-.8) arc (180:0:.2cm);
	\draw[thick, unshaded]  (0,0) circle (.4cm);
	\node at (0,0) {$U$};
	\node at (.55,0) {$\star$};
\end{tikzpicture}
+
\sigma_U\,
\begin{tikzpicture}[baseline = -.3cm,xscale=-1]
	\draw[thick, zeta] (-.6,0) -- (.6,0);
	\coordinate (a) at (0,-.5);	
	\coordinate (b) at (-1,-.5);	
	\filldraw[rho] (a) circle (.05cm);
	\filldraw[rho] (b) circle (.05cm);
	\draw[thick, rho] (-.6,0)--(a)--(.6,0);
	\draw[thick, rho] (a)--(0,-.8);
	\draw[thick, rho] (-1.2,-.8)--(b)--(-.8,-.8);
	\draw[thick, rho] (b)--(-.6,0);
	\draw[thick, zeta] (-.6,0) -- (-.6,-.8);
	\draw[thick, zeta] (.5,0) -- (.5,-.8);
	\draw[thick, rho] (.7,0) -- (.7,-.8);
	\draw[thick, unshaded]  (-.6,0) circle (.4cm);
	\node at (-.6,0) {$U^*$};
	\node at (-.6,.55) {$\star$};
	\draw[thick, unshaded]  (.6,0) circle (.4cm);
	\node at (.6,0) {$U$};
	\node at (.6,.55) {$\star$};
\end{tikzpicture}
$\end{enumerate}
\end{cor}

\begin{thm}\label{thm:JellyfishEvaluation}
Relations \eqref{rel:AA1}-\eqref{rel:AA2}, the relations in Proposition \ref{prop:RhoSkeinRelations}, Relations \eqref{rel:AT2}-\eqref{rel:AT3}, and the relations in Theorem \ref{thm:JellyfishRelationsU} are sufficient to evaluate all closed diagrams.
\end{thm}
\begin{proof}
It suffices to show we can evaluate any closed diagram in jellyfish form by Theorem \ref{thm:JellyfishRelationsU} and Corollary \ref{cor:JellyfishRelationsUstar}.

Suppose we have such a diagram $D$. We show that we can evaluate $D$ by induction on the number of generators $U,U^*$ in the diagram.
\begin{enumerate}
\item[\underline{$k=0$:}]
If there are no $U$'s or $U^*$'s, we can evaluate $D$ using Relation \eqref{rel:AA1} and the relations in Proposition \ref{prop:RhoSkeinRelations}.
\item[\underline{$k=1$:}]
If there is only one $U$ or $U^*$ in $D$, all $\theta$ strings connect back to the generator, so somewhere there is an innermost $\theta$ cap. Inside this cap is a red diagram with only one boundary point, which must be zero by Proposition \ref{prop:RhoSkeinRelations}.
\item[\underline{$k\geq 2$:}]
Suppose there are $k\geq 2$ generators $U$ or $U^*$, and suppose that we may evaluate all closed diagrams $D$ with $k-1$ generators.

If there are no trivalent vertices, the usual argument in the jellyfish algorithm applies \cite{MR2979509}, so there must be two generators connected by at least $n$ strings along the boundary. 
We can then use Relation \eqref{rel:AT2} to obtain a diagram with $k-2$ generators. We are finished by the induction hypothesis.

Suppose now that there are trivalent vertices. 
Using Relation \eqref{rel:AT1} in Proposition \ref{prop:RhoSkeinRelations}, which does not increase the number of generators, we may assume that no two trivalent vertices of $D$ are connected. 
Hence each string connected to a trivalent vertex connects to a generator.
If there is a vertex connected by two $\rho$ strings to a generator $U'$, then between those $\rho$ strings there is an innermost $\theta$ cap. The argument from the $k=1$ case shows $D=0$.

Now we may assume each vertex attaches to 3 distinct generators. 
Isotope $D$ so that all trivalent vertices have strings emanating from the top, and these strings travel upward and attach to $U$'s or $U^*$'s with no critical points (this amounts to picking a linear ordering of the generators rather than the cyclic ordering afforded by jellyfish form)
$$
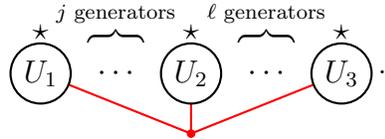
\begin{tikzpicture}[baseline = 0cm]
	\coordinate (a) at (2,-.8) ;
	\filldraw[rho] (a) circle (.05cm);
	\draw[thick, rho] (a) -- (0,0);
	\draw[thick, rho] (a) -- (2,0);
	\draw[thick, rho] (a) -- (4,0);	
	\Ucircle{U_1}{(0,0)}
	\node at (1,0) {$\cdots$};
	\Ucircle{U_2}{(2,0)}
	\node at (3,0) {$\cdots$};
	\Ucircle{U_3}{(4,0)}
	\node [rotate around = {90:(0,0)}] at (1,.5) {$\Big\}$};
	\node at (1,.8) {\scriptsize{$j$ generators}};
	\node [rotate around = {90:(0,0)}] at (3,.5) {$\Big\}$};
	\node at (3,.8) {\scriptsize{$\ell$ generators}};
\end{tikzpicture}\,.
$$
(Above, $U_1,U_2,U_3$ are either $U$ or $U^*$.)
Hence each trivalent vertex bounds two inner regions in the diagram.
Pick an innermost trivalent vertex $v$, i.e., a trivalent vertex for which these two inner regions contain no other trivalent vertices. 
Let $U_1,U_2,U_3$ be the distinct generators attached to $v$ as in the diagram above.
Let $j$ be the number of distinct generators between $U_1$ and $U_2$, and let $\ell$ be the number of distinct generators between $U_2$ and $U_3$.

If $j$ and $\ell$ are both zero, then $U_2$ is connected to either $U_1$ or $U_3$ by at least $n$ strings, and we may use Relation \eqref{rel:AT2} to reduce the number of generators by 2.
If $j>0$, looking at the region above our innermost trivalent vertex $v$ and between $U_1$ and $U_2$, we see have a polygonal region whose vertices are the copies of $U$ or $U^*$, with some number of diagonals. There are two strings connecting $v$ to $U_1$ and $U_2$, we consider these as a single distinguished edge of the polygon. Now, the usual jellyfish argument proceeds by showing that every polygon with diagonals has a vertex with no incident diagonals. 
If we were assured only one such vertex, it may be the case that this vertex were $U_1$ or $U_2$, and we would get stuck at this point.
However, the stronger result is that there is a pair of nonadjacent vertices with no incident diagonals (if $j>1$), and hence at least one of the $j$ generators strictly between $U_1$ and $U_2$ is connected to one its neighbors by at least $n$ strands. Hence we may reduce the diagram using Relation \eqref{rel:AT2}, leaving it in jellyfish form, and we are finished by the induction hypothesis.
\end{enumerate}
\end{proof}

\begin{defn}
For $1\leq n<\infty$, let $\AT_{n,\omega_U}$ be the unitary quotient of $A_2*T_2$ generated by $U$ satisfying 
Relation \eqref{rel:AA1}, the relations of Proposition \ref{prop:RhoSkeinRelations}, and Relations \eqref{rel:AT2}-\eqref{rel:AT3},  provided that it exists. Note that $\AT_{1,1}$ is $A_2\boxtimes T_2$.
\end{defn}

We show in Theorem \ref{thm:Unique} that for $n=1,2,3$, we must have $\omega_U=1$.
We show that $\AT_{n,1}$ exists for $n=1,2,3$ in Subsection \ref{sec:Existence}.
We show in Theorem \ref{thm:Nonexistence} that for $4\leq n\leq 10$, $\AT_{n,\omega_U}$ does not exist.

\subsection{A basis for jellyfish calculations}

\begin{prop}\label{prop:basis}
Consider all diagrams of the form
$$
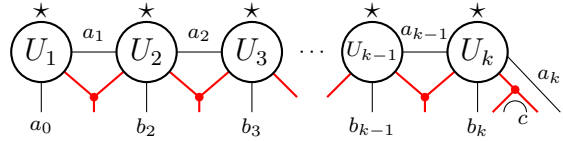
\begin{tikzpicture}[baseline = -.3cm]
	\coordinate (a) at (-3.6,0);
	\coordinate (b) at (-2.2,0);
	\coordinate (c) at (-.8,0);
	\coordinate (d) at (.8,0);
	\coordinate (e) at (2.2,0);
	\coordinate (ab) at ($(a)+(.7,-.6)$);
	\coordinate (bc) at ($(b)+(.7,-.6)$);
	\coordinate (de) at ($(d)+(.7,-.6)$);
	\coordinate (ue) at ($(e)+(.5,-.5)$);
	\draw (a) -- node [above] {{\scriptsize{$a_1$}}} (b);
	\draw (b) -- node [above] {{\scriptsize{$a_2$}}} (c);
	\draw (d) -- node [above] {{\scriptsize{$a_{k-1}$}}} (e);
	\draw[thick, rho] (a)--(ab) -- (b);
	\draw[thick, rho] (ab) -- ($(ab)+(0,-.2)$);
	\draw[thick, rho] (b)--(bc) -- (c);
	\draw[thick, rho] (bc) -- ($(bc)+(0,-.2)$);
	\draw[thick, rho] (c)--($(c)+(.6,-.6)$);
	\draw[thick, rho] (d)--($(d)+(-.6,-.6)$);
	\draw[thick, rho] (d)--(de) -- (e);
	\draw[thick, rho] (de) -- ($(de)+(0,-.2)$);
	\draw[thick, rho] (e)--(ue);
	\draw[thick, rho] (ue)--($(ue)+(-.3,-.3)$);
	\draw[thick, rho] (ue)--($(ue)+(.3,-.3)$);
	\draw[] ($(ue)+(-.15,-.3)$) arc (180:0:.15cm);
	\node at ($(ue)+(.1,-.4)$) {{\scriptsize{$c$}}};	
	\draw[] ($(e)+(.2,.15)$)-- node [right] {{\scriptsize{$a_{k}$}}} ($(ue)+(.6,-.3)$);
	\draw[] ($(a)+(0,0)$)-- ($(a)+(0,-.8)$);
	\node at ($(a)+(-0,-1)$) {{\scriptsize{$a_0$}}};
	\draw[] ($(b)+(0,0)$)-- ($(b)+(0,-.8)$);
	\node at ($(b)+(0,-1)$) {{\scriptsize{$b_2$}}};
	\draw[] ($(c)+(0,0)$)-- ($(c)+(0,-.8)$);
	\node at ($(c)+(0,-1)$) {{\scriptsize{$b_3$}}};
	\draw[] ($(d)+(0,0)$)-- ($(d)+(0,-.8)$);
	\node at ($(d)+(0,-1)$) {{\scriptsize{$b_{k-1}$}}};
	\draw[] ($(e)+(0,0)$)-- ($(e)+(0,-.8)$);
	\node at ($(e)+(0,-1)$) {{\scriptsize{$b_{k}$}}};
	\Ucircle{U_1}{(a)}
	\Ucircle{U_2}{(b)}
	\Ucircle{U_3}{(c)}
	\node at ($(c)+(.8,0)$) {{\scriptsize{$\cdots$}}};
	\largeUcircle{U_{k-1}}{(d)}
	\Ucircle{U_k}{(e)}
	\filldraw[rho] (ab) circle (.05cm);
	\filldraw[rho] (bc) circle (.05cm);
	\filldraw[rho] (de) circle (.05cm);
	\filldraw[rho] (ue) circle (.05cm);
\end{tikzpicture}
$$
where the labels on the strings indicate the total number of $\rho$ and $\theta$ strings in the bundle,
satisfying the following criteria:
\begin{itemize}
\item $k$ is even,
\item the $U_i$'s alternate between $U$ and $U^*$, but $U_1$ is not necessarily $U$,
\item $a_0\geq 2n$ and $0\leq a_{k}\leq 2n-2$,
\item $1\leq c\leq 2n-1$ is odd, and the bundle is of the form $\theta(\rho\theta)^{j}$ for some $j$,
\item for $2\leq i\leq k$, $1\leq b_i\leq 2n-1$ and $b_i$ is odd, and
\item for $i=1,\dots, k-1$, $1\leq a_i\leq 2n-1$ and $a_i-a_{i\pm 1}$ is odd (the parity of the $a_i$'s alternates).
\end{itemize}
Then each such diagram has nonzero norm squared, and distinct diagrams with the same number of strings attached to the external boundary are orthogonal.
\end{prop}
\begin{proof}
It is straightforward to calculate that the norm squared of such a diagram is a power of $\tau$, which is nonzero.

Suppose now that we have two  distinct diagrams with the same number of strings attached to the external boundary. 
Let the first diagram have constants $(a_0,\dots a_{k},b_2,\dots, b_k,c)$, and let the second diagram have constants $(a_0',\dots a_{\ell}',b_2',\dots, b_\ell',c')$.
Without loss of generality, assume $k\leq \ell$.
We cannot have $a_i=a_i'$ for all $i=0,\dots, k$, since the $b_i$'s are determined by the $a_i$'s, and then $c,c'$ are determined by the $a_i$'s and the number of external boundary points.

Let $j$ be minimal such that $a_j\neq a_j'$, which implies that $b_{j}\neq b_{j}'$.
If $a_j>a_j'$, then $b_{j}< b_{j}'$. Taking the inner product, we get the following sub-diagram:
$$
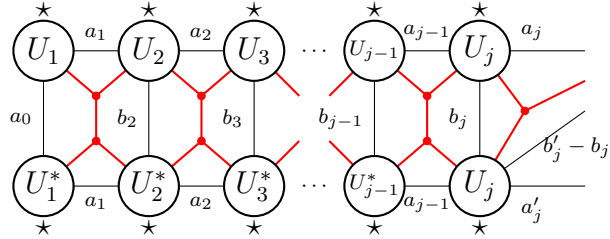
\begin{tikzpicture}[baseline = -.3cm]
	\coordinate (a) at (-3.6,0);
	\coordinate (b) at (-2.2,0);
	\coordinate (c) at (-.8,0);
	\coordinate (d) at (.8,0);
	\coordinate (e) at (2.2,0);
	\coordinate (f) at (3.6,0);
	\coordinate (ap) at (-3.6,-1.8);
	\coordinate (bp) at (-2.2,-1.8);
	\coordinate (cp) at (-.8,-1.8);
	\coordinate (dp) at (.8,-1.8);
	\coordinate (ep) at (2.2,-1.8);
	\coordinate (fp) at (3.6,-1.8);
	\coordinate (ab) at ($(a)+(.7,-.6)$);
	\coordinate (bc) at ($(b)+(.7,-.6)$);
	\coordinate (de) at ($(d)+(.7,-.6)$);
	\coordinate (ef) at ($(e)+(.6,-.8)$);
	\coordinate (apbp) at ($(ap)+(.7,.6)$);
	\coordinate (bpcp) at ($(bp)+(.7,.6)$);
	\coordinate (dpep) at ($(dp)+(.7,.6)$);
	\draw (a) -- node [above] {{\scriptsize{$a_1$}}} (b);
	\draw (b) -- node [above] {{\scriptsize{$a_2$}}} (c);
	\draw (d) -- node [above] {{\scriptsize{$a_{j-1}$}}} (e);
	\draw (e) -- node [above] {{\scriptsize{$a_{j}$}}} (f);
	\draw (ap) -- node [below] {{\scriptsize{$a_1$}}} (bp);
	\draw (bp) -- node [below] {{\scriptsize{$a_2$}}} (cp);
	\draw (dp) -- node [below] {{\scriptsize{$a_{j-1}$}}} (ep);
	\draw (ep) -- node [below] {{\scriptsize{$a_{j}'$}}} (fp);
	\draw[thick, rho] (a)--(ab) -- (b);
	\draw[thick, rho] (ab) -- (apbp);
	\draw[thick, rho] (ap)--(apbp) -- (bp);
	\draw[thick, rho] (b)--(bc) -- (c);
	\draw[thick, rho] (bc) -- (bpcp);
	\draw[thick, rho] (bp)--(bpcp) -- (cp);
	\draw[thick, rho] (c)--($(c)+(.6,-.6)$);
	\draw[thick, rho] (cp)--($(cp)+(.6,.6)$);
	\draw[thick, rho] (d)--($(d)+(-.6,-.6)$);
	\draw[thick, rho] (dp)--($(dp)+(-.6,.6)$);
	\draw[thick, rho] (d)--(de) -- (e);
	\draw[thick, rho] (de) -- (dpep);
	\draw[thick, rho] (dp)--(dpep) -- (ep);	
	\draw[thick, rho] (e)--(ef) -- ($(f)+(0,-.4)$);	
	\draw[thick, rho] (ef) -- (ep);
	\draw[] (a)-- node [left] {{\scriptsize{$a_0$}}} (ap);
	\draw[] (b)-- node [left] {{\scriptsize{$b_2$}}} (bp);
	\draw[] (c)-- node [left] {{\scriptsize{$b_3$}}} (cp);
	\draw[] (d)-- node [left] {{\scriptsize{$b_{j-1}$}}} (dp);
	\draw[] (e)-- node [left] {{\scriptsize{$b_{j}$}}} (ep);
	\draw[]  ($(f)+(0,-.8)$)-- node [right] {{\scriptsize{$b_{j}'-b_{j}$}}} (ep);
	\Ucircle{U_1}{(a)}
	\Ucircle{U_2}{(b)}
	\Ucircle{U_3}{(c)}
	\node at ($(c)+(.8,0)$) {{\scriptsize{$\cdots$}}};
	\largeUcircle{U_{j-1}}{(d)}
	\Ucircle{U_{j}}{(e)}
	\UcircleStar{U_1^*}{(ap)}
	\UcircleStar{U_2^*}{(bp)}
	\UcircleStar{U_3^*}{(cp)}
	\node at ($(cp)+(.8,0)$) {{\scriptsize{$\cdots$}}};
	\largeUcircleStar{U_{j-1}^*}{(dp)}
	\UcircleStar{U_{j}}{(ep)}
	\filldraw[rho] (ab) circle (.05cm);
	\filldraw[rho] (bc) circle (.05cm);
	\filldraw[rho] (de) circle (.05cm);
	\filldraw[rho] (ef) circle (.05cm);
	\filldraw[rho] (apbp) circle (.05cm);
	\filldraw[rho] (bpcp) circle (.05cm);
	\filldraw[rho] (dpep) circle (.05cm);
\end{tikzpicture}
$$
We may iteratively cancel the first $j-1$ pairs of generators $U_i,U_i^*$ counting from the left, since $a_0\geq 2n$, and $a_{i-1}+b_i\geq 2n$ for all $i=2,\dots, j-1$. We then get some power of $\tau$ times
$$
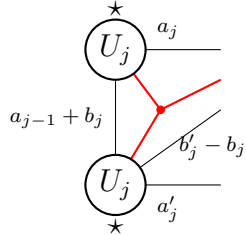
\begin{tikzpicture}[baseline = -.3cm]
	\coordinate (e) at (2.2,0);
	\coordinate (f) at (3.6,0);
	\coordinate (ep) at (2.2,-1.8);
	\coordinate (fp) at (3.6,-1.8);
	\coordinate (ef) at ($(e)+(.6,-.8)$);
	\draw (e) -- node [above] {{\scriptsize{$a_{j}$}}} (f);
	\draw (ep) -- node [below] {{\scriptsize{$a_{j}'$}}} (fp);
	\draw[thick, rho] (e)--(ef) -- ($(f)+(0,-.4)$);	
	\draw[thick, rho] (ef) -- (ep);
	\draw[] (e)-- node [left] {{\scriptsize{$a_{j-1}+b_{j}$}}} (ep);
	\draw[]  ($(f)+(0,-.8)$)-- node [right] {{\scriptsize{$b_{j}'-b_{j}$}}} (ep);
	\Ucircle{U_{j}}{(e)}
	\UcircleStar{U_{j}}{(ep)}
	\filldraw[rho] (ef) circle (.05cm);
\end{tikzpicture}
$$
which is zero since $a_{j-1}+b_{j}\geq 2n$. Similarly, we get zero if $a_{j}'>a_j$.
\end{proof}


\begin{cor}\label{cor:basis}
A similar statement to Proposition \ref{prop:basis} holds for diagrams of the form
$$
\begin{tikzpicture}[baseline = -.3cm, xscale=-1]
	\coordinate (a) at (-3.6,0);
	\coordinate (b) at (-2.2,0);
	\coordinate (c) at (-.8,0);
	\coordinate (d) at (.8,0);
	\coordinate (e) at (2.2,0);
	\coordinate (ab) at ($(a)+(.7,-.6)$);
	\coordinate (bc) at ($(b)+(.7,-.6)$);
	\coordinate (de) at ($(d)+(.7,-.6)$);
	\coordinate (ue) at ($(e)+(.5,-.5)$);
	\draw (a) -- node [above] {{\scriptsize{$a_k$}}} (b);
	\draw (b) -- node [above] {{\scriptsize{$a_{k-1}$}}} (c);
	\draw (d) -- node [above] {{\scriptsize{$a_{2}$}}} (e);
	\draw[thick, rho] (a)--(ab) -- (b);
	\draw[thick, rho] (ab) -- ($(ab)+(0,-.2)$);
	\draw[thick, rho] (b)--(bc) -- (c);
	\draw[thick, rho] (bc) -- ($(bc)+(0,-.2)$);
	\draw[thick, rho] (c)--($(c)+(.6,-.6)$);
	\draw[thick, rho] (d)--($(d)+(-.6,-.6)$);
	\draw[thick, rho] (d)--(de) -- (e);
	\draw[thick, rho] (de) -- ($(de)+(0,-.2)$);
	\draw[thick, rho] (e)--(ue);
	\draw[thick, rho] (ue)--($(ue)+(-.3,-.3)$);
	\draw[thick, rho] (ue)--($(ue)+(.3,-.3)$);
	\draw[] ($(ue)+(-.15,-.3)$) arc (180:0:.15cm);
	\node at ($(ue)+(.1,-.4)$) {{\scriptsize{$d$}}};	
	\draw[] ($(e)+(.2,.15)$)-- node [left] {{\scriptsize{$a_{1}$}}} ($(ue)+(.6,-.3)$);
	\draw[] ($(a)+(0,0)$)-- ($(a)+(0,-.8)$);
	\node at ($(a)+(-0,-1)$) {{\scriptsize{$a_{k+1}$}}};
	\draw[] ($(b)+(0,0)$)-- ($(b)+(0,-.8)$);
	\node at ($(b)+(0,-1)$) {{\scriptsize{$b_k$}}};
	\draw[] ($(c)+(0,0)$)-- ($(c)+(0,-.8)$);
	\node at ($(c)+(0,-1)$) {{\scriptsize{$b_{k-1}$}}};
	\draw[] ($(d)+(0,0)$)-- ($(d)+(0,-.8)$);
	\node at ($(d)+(0,-1)$) {{\scriptsize{$b_{3}$}}};
	\draw[] ($(e)+(0,0)$)-- ($(e)+(0,-.8)$);
	\node at ($(e)+(0,-1)$) {{\scriptsize{$b_{2}$}}};
	\largeUcircle{U_{k+1}}{(a)}
	\Ucircle{U_k}{(b)}
	\largeUcircle{U_{k-1}}{(c)}
	\node at ($(c)+(.8,0)$) {{\scriptsize{$\cdots$}}};
	\Ucircle{U_{3}}{(d)}
	\Ucircle{U_2}{(e)}
	\filldraw[rho] (ab) circle (.05cm);
	\filldraw[rho] (bc) circle (.05cm);
	\filldraw[rho] (de) circle (.05cm);
	\filldraw[rho] (ue) circle (.05cm);
\end{tikzpicture}
$$
where the $a_i$'s and $b_i$'s satisfy the same criteria, but instead of $a_0$ and $c$, we have $a_{k+1}\geq 2n$ and $0\leq d \leq 2n-2$, where the $d$ bundle is of the form $(\rho\theta)^{j}$.
\end{cor}


\begin{nota}
We use the following notation for one car trains:
\begin{align*}
\train[\{\},\wheel[c]]
&=
\begin{tikzpicture}[baseline = -.3cm]
	\draw[thick, zeta] (-.3,0) -- (-.3,-.8);
	\draw[thick, rho] (-.1,0) -- (-.1,-.8);
	\draw[thick, zeta] (.1,0) -- (.1,-.8);
	\draw[thick, rho] (.3,0) -- (.3,-.8);
	\node at (.7,-.4) {{\scriptsize{$c$}}};
	\draw (.5,-.8) arc (180:0:.2cm);
	\draw[thick, unshaded]  (0,0) circle (.4cm);
	\node at (0,0) {$U$};
	\node at (.55,0) {$\star$};
\end{tikzpicture}
\\
\train[\wheel[d],\{\}]
&=
\begin{tikzpicture}[baseline = -.3cm,xscale=-1]
	\draw[thick, zeta] (-.3,0) -- (-.3,-.8);
	\draw[thick, rho] (-.1,0) -- (-.1,-.8);
	\draw[thick, zeta] (.1,0) -- (.1,-.8);
	\draw[thick, rho] (.3,0) -- (.3,-.8);
	\draw (.5,-.8) arc (180:0:.2cm);
	\node at (.7,-.4) {{\scriptsize{$d$}}};
	\draw[thick, unshaded]  (0,0) circle (.4cm);
	\node at (0,0) {$U$};
	\node at (.55,0) {$\star$};
\end{tikzpicture}
\,.
\end{align*}
We use the following notation for diagrams in the form of Proposition \ref{prop:basis}:
\begin{align*}
\train[a_0, a_1,\dots, a_{k-1}, a_k,\wheel[c]]
&=
\begin{tikzpicture}[baseline = -.3cm]
	\coordinate (a) at (-3.6,0);
	\coordinate (b) at (-2.2,0);
	\coordinate (c) at (-.8,0);
	\coordinate (d) at (.8,0);
	\coordinate (e) at (2.2,0);
	\coordinate (ab) at ($(a)+(.7,-.6)$);
	\coordinate (bc) at ($(b)+(.7,-.6)$);
	\coordinate (de) at ($(d)+(.7,-.6)$);
	\coordinate (ue) at ($(e)+(.5,-.5)$);
	\draw (a) -- node [above] {{\scriptsize{$a_1$}}} (b);
	\draw (b) -- node [above] {{\scriptsize{$a_2$}}} (c);
	\draw (d) -- node [above] {{\scriptsize{$a_{k-1}$}}} (e);
	\draw[thick, rho] (a)--(ab) -- (b);
	\draw[thick, rho] (ab) -- ($(ab)+(0,-.2)$);
	\draw[thick, rho] (b)--(bc) -- (c);
	\draw[thick, rho] (bc) -- ($(bc)+(0,-.2)$);
	\draw[thick, rho] (c)--($(c)+(.6,-.6)$);
	\draw[thick, rho] (d)--($(d)+(-.6,-.6)$);
	\draw[thick, rho] (d)--(de) -- (e);
	\draw[thick, rho] (de) -- ($(de)+(0,-.2)$);
	\draw[thick, rho] (e)--(ue);
	\draw[thick, rho] (ue)--($(ue)+(-.3,-.3)$);
	\draw[thick, rho] (ue)--($(ue)+(.3,-.3)$);
	\draw[] ($(ue)+(-.15,-.3)$) arc (180:0:.15cm);
	\node at ($(ue)+(.1,-.4)$) {{\scriptsize{$c$}}};	
	\draw[] ($(e)+(.2,.15)$)-- node [right] {{\scriptsize{$a_{k}$}}} ($(ue)+(.6,-.3)$);
	\draw[] ($(a)+(0,0)$)-- ($(a)+(0,-.8)$);
	\node at ($(a)+(-0,-1)$) {{\scriptsize{$a_0$}}};
	\draw[] ($(b)+(0,0)$)-- ($(b)+(0,-.8)$);
	\node at ($(b)+(0,-1)$) {{\scriptsize{$b_2$}}};
	\draw[] ($(c)+(0,0)$)-- ($(c)+(0,-.8)$);
	\node at ($(c)+(0,-1)$) {{\scriptsize{$b_3$}}};
	\draw[] ($(d)+(0,0)$)-- ($(d)+(0,-.8)$);
	\node at ($(d)+(0,-1)$) {{\scriptsize{$b_{k-1}$}}};
	\draw[] ($(e)+(0,0)$)-- ($(e)+(0,-.8)$);
	\node at ($(e)+(0,-1)$) {{\scriptsize{$b_{k}$}}};
	\Ucircle{U_1}{(a)}
	\Ucircle{U_2}{(b)}
	\Ucircle{U_3}{(c)}
	\node at ($(c)+(.8,0)$) {{\scriptsize{$\cdots$}}};
	\largeUcircle{U_{k-1}}{(d)}
	\Ucircle{U_k}{(e)}
	\filldraw[rho] (ab) circle (.05cm);
	\filldraw[rho] (bc) circle (.05cm);
	\filldraw[rho] (de) circle (.05cm);
	\filldraw[rho] (ue) circle (.05cm);
\end{tikzpicture}
\\
\train[\wheel[d], a_1,a_2\dots, a_{k}, a_{k+1}]
&=
\begin{tikzpicture}[baseline = -.3cm, xscale=-1]
	\coordinate (a) at (-3.6,0);
	\coordinate (b) at (-2.2,0);
	\coordinate (c) at (-.8,0);
	\coordinate (d) at (.8,0);
	\coordinate (e) at (2.2,0);
	\coordinate (ab) at ($(a)+(.7,-.6)$);
	\coordinate (bc) at ($(b)+(.7,-.6)$);
	\coordinate (de) at ($(d)+(.7,-.6)$);
	\coordinate (ue) at ($(e)+(.5,-.5)$);
	\draw (a) -- node [above] {{\scriptsize{$a_k$}}} (b);
	\draw (b) -- node [above] {{\scriptsize{$a_{k-1}$}}} (c);
	\draw (d) -- node [above] {{\scriptsize{$a_{2}$}}} (e);
	\draw[thick, rho] (a)--(ab) -- (b);
	\draw[thick, rho] (ab) -- ($(ab)+(0,-.2)$);
	\draw[thick, rho] (b)--(bc) -- (c);
	\draw[thick, rho] (bc) -- ($(bc)+(0,-.2)$);
	\draw[thick, rho] (c)--($(c)+(.6,-.6)$);
	\draw[thick, rho] (d)--($(d)+(-.6,-.6)$);
	\draw[thick, rho] (d)--(de) -- (e);
	\draw[thick, rho] (de) -- ($(de)+(0,-.2)$);
	\draw[thick, rho] (e)--(ue);
	\draw[thick, rho] (ue)--($(ue)+(-.3,-.3)$);
	\draw[thick, rho] (ue)--($(ue)+(.3,-.3)$);
	\draw[] ($(ue)+(-.15,-.3)$) arc (180:0:.15cm);
	\node at ($(ue)+(.1,-.4)$) {{\scriptsize{$d$}}};	
	\draw[] ($(e)+(.2,.15)$)-- node [left] {{\scriptsize{$a_{1}$}}} ($(ue)+(.6,-.3)$);
	\draw[] ($(a)+(0,0)$)-- ($(a)+(0,-.8)$);
	\node at ($(a)+(-0,-1)$) {{\scriptsize{$a_{k+1}$}}};
	\draw[] ($(b)+(0,0)$)-- ($(b)+(0,-.8)$);
	\node at ($(b)+(0,-1)$) {{\scriptsize{$b_k$}}};
	\draw[] ($(c)+(0,0)$)-- ($(c)+(0,-.8)$);
	\node at ($(c)+(0,-1)$) {{\scriptsize{$b_{k-1}$}}};
	\draw[] ($(d)+(0,0)$)-- ($(d)+(0,-.8)$);
	\node at ($(d)+(0,-1)$) {{\scriptsize{$b_{3}$}}};
	\draw[] ($(e)+(0,0)$)-- ($(e)+(0,-.8)$);
	\node at ($(e)+(0,-1)$) {{\scriptsize{$b_{2}$}}};
	\Ucircle{U}{(a)}
	\Ucircle{U^*}{(b)}
	\Ucircle{U}{(c)}
	\node at ($(c)+(.8,0)$) {{\scriptsize{$\cdots$}}};
	\Ucircle{U}{(d)}
	\Ucircle{U^*}{(e)}
	\filldraw[rho] (ab) circle (.05cm);
	\filldraw[rho] (bc) circle (.05cm);
	\filldraw[rho] (de) circle (.05cm);
	\filldraw[rho] (ue) circle (.05cm);
\end{tikzpicture}\,.
\end{align*}
We now omit the $b_i$'s since they can be recovered from the $a_i$'s.
These diagrams can be thought of as products of two car trains
\begin{equation}
\twocar[a_{j-1},a_j,a_{j+1}]
=
\begin{tikzpicture}[baseline = -.3cm]
	\coordinate (b) at (-2.2,0);
	\coordinate (c) at (-.8,0);
	\coordinate (bc) at ($(b)+(.7,-.6)$);
	\coordinate (cd) at ($(c)+(.7,-.6)$);
	\draw ($(b)+(-1,0)$) -- (b);
	\node at ($(b)+(-.8,.2)$) {{\scriptsize{$a_{j-1}$}}};
	\draw (b) -- node [above] {{\scriptsize{$a_j$}}} (c);
	\draw (c) -- ($(c)+(1.5,0)$);
	\node at ($(c)+(.8,.2)$) {{\scriptsize{$a_{j+1}$}}};
	\draw[thick, rho] ($(b)+(0,-.2)$)--($(b)+(-1,-.2)$);
	\draw[thick, rho] (b)--(bc) -- (c);
	\draw[thick, rho] (bc) -- ($(bc)+(0,-.2)$);
	\draw[thick, rho] (c)--(cd) .. controls ++(45:.2cm) and ++(180:.6cm) .. ($(cd)+(.8,.4)$);
	\draw[thick, rho] (cd)--($(cd)+(0,-.2)$);
	\draw[] ($(b)+(0,0)$)-- ($(b)+(0,-.8)$);
	\draw[] ($(c)+(0,0)$)-- ($(c)+(0,-.8)$);
	\Ucircle{U_1}{(b)}
	\Ucircle{U_2}{(c)}
	\filldraw[rho] (bc) circle (.05cm);
	\filldraw[rho] (cd) circle (.05cm);
\end{tikzpicture}
\label{eqn:TwoCarTrain}
\end{equation}
with an engine and a caboose:
\begin{align*}
\engine[a_k,c]
&=
\begin{tikzpicture}[baseline = -.3cm]
	\draw (0,-.8)--(0,-.2) arc (0:90:.2cm) --  (-.7,0);
	\draw[thick, rho] (-.2,-.8)--(-.2,-.4) arc (0:90:.2cm) -- (-.7,-.2);
	\draw (-.7,-.8) arc (180:0:.15cm);
	\node at (-.65,-.5) {{\scriptsize{$c$}}};
	\node at (.2,-.4) {{\scriptsize{$a_k$}}};
\end{tikzpicture}
\\
\caboose[a_0]
&=
\begin{tikzpicture}[baseline = -.3cm]
	\draw (0,-.8)--(0,-.2) arc (180:90:.2cm) -- (.4,0);
	\draw[thick, rho] (.2,-.8)--(.2,-.4) arc (180:90:.2cm);
	\node at (-.2,-.4) {{\scriptsize{$a_0$}}};
\end{tikzpicture}\,,
\end{align*}
with the convention that the last $\twocar$ in the train must have all its strings connected to the caboose (and no strings going downward).
We multiply the train parts by concatenating horizontally:
$$
\twocar[a_1,a_2,a_3]\twocar[a_3,a_4,a_5]
=
\begin{tikzpicture}[baseline = -.3cm]
	\coordinate (b) at (-2.1,0);
	\coordinate (c) at (-.7,0);
	\coordinate (d) at (.7,0);
	\coordinate (e) at (2.1,0);
	\coordinate (bc) at ($(b)+(.7,-.6)$);
	\coordinate (cd) at ($(c)+(.7,-.6)$);
	\coordinate (de) at ($(d)+(.7,-.6)$);
	\coordinate (ef) at ($(e)+(.7,-.6)$);
	\draw ($(b)+(-1,0)$) -- (b);
	\node at ($(b)+(-.7,.2)$) {{\scriptsize{$a_1$}}};
	\draw (b) -- node [above] {{\scriptsize{$a_2$}}} (c);
	\draw (c) -- node [above] {{\scriptsize{$a_3$}}} (d);
	\draw (d) -- node [above] {{\scriptsize{$a_4$}}} (e);
	\draw (e)--($(e)+(1.5,0)$);
	\node at ($(e)+(.7,.2)$) {{\scriptsize{$a_5$}}};
	\draw[thick, rho] ($(b)+(0,-.2)$)--($(b)+(-1,-.2)$);
	\draw[thick, rho] (b)--(bc) -- (c);
	\draw[thick, rho] (bc) -- ($(bc)+(0,-.2)$);
	\draw[thick, rho] (c)--(cd) -- (d);
	\draw[thick, rho] (cd) -- ($(cd)+(0,-.2)$);
	\draw[thick, rho] (d)--(de) -- (e);
	\draw[thick, rho] (de) -- ($(de)+(0,-.2)$);
	\draw[thick, rho] (e)--(ef) .. controls ++(45:.2cm) and ++(180:.6cm) .. ($(ef)+(.8,.4)$);
	\draw[thick, rho] (ef) -- ($(ef)+(0,-.2)$);
	\draw[] ($(b)+(0,0)$)-- ($(b)+(0,-.8)$);
	\draw[] ($(c)+(0,0)$)-- ($(c)+(0,-.8)$);
	\draw[] ($(d)+(0,0)$)-- ($(d)+(0,-.8)$);
	\draw[] ($(e)+(0,0)$)-- ($(e)+(0,-.8)$);
	\Ucircle{U_1}{(b)}
	\Ucircle{U_2}{(c)}
	\Ucircle{U_3}{(d)}
	\Ucircle{U_4}{(e)}
	\filldraw[rho] (bc) circle (.05cm);
	\filldraw[rho] (cd) circle (.05cm);
	\filldraw[rho] (de) circle (.05cm);
	\filldraw[rho] (ef) circle (.05cm);
\end{tikzpicture}.
$$

To simplify future calculation, we will also allow the external $a_i$'s in our two car trains to surpass $2n$, i.e., for the two car train in Equation \eqref{eqn:TwoCarTrain}, if $a_j<2n-2$, then $a_{j-1}$ or $a_{j+1}$ may be more than $2n-1$. 
If $a_j=2n-1$, we have the following two car trains for which one of $a_{j\pm 1}$ is $2n$:
\begin{align*}
\twocar[2n,2n-1,a_j]
&=
\begin{tikzpicture}[baseline = -.3cm]
	\coordinate (a) at (-3.2,0);
	\coordinate (b) at (-2.2,0);
	\coordinate (c) at (-.6,0);
	\coordinate (d) at (.4,0);
	\coordinate (ab) at ($(b)+(-1,-.2)$);
	\coordinate (bc) at ($(b)+(.8,-.6)$);
	\coordinate (cd) at ($(c)+(.6,-.6)$);
	\draw (a) -- (b);
	\node at ($(b)+(-.7,.2)$) {{\scriptsize{$2n$}}};
	\draw (b) -- node [above] {{\scriptsize{$2n-1$}}} (c);
	\draw (c) -- (d);
	\node at ($(c)+(.7,.2)$) {{\scriptsize{$a_{j}$}}};
	\draw[thick, rho] (b)--(bc) -- (c);
	\draw[thick, rho] (bc) .. controls ++(270:.6cm) and ++(0:.6cm) .. (ab);
	\draw[thick, rho] ($(c)+(0,-.2)$)--($(c)+(1,-.2)$);
	\draw[] ($(c)+(0,0)$)-- ($(c)+(0,-.8)$);
	\Ucircle{U}{(b)}
	\Ucircle{U^*}{(c)}
	\filldraw[rho] (bc) circle (.05cm);
\end{tikzpicture}
\text{ where $a_j$ is even}
\\
\twocar[a_j,2n-1,2n]
&=
\begin{tikzpicture}[baseline = -.3cm, xscale=-1]
	\coordinate (a) at (-3.2,0);
	\coordinate (b) at (-2.2,0);
	\coordinate (c) at (-.6,0);
	\coordinate (d) at (.4,0);
	\coordinate (ab) at ($(b)+(-1,-.2)$);
	\coordinate (bc) at ($(b)+(.8,-.6)$);
	\coordinate (cd) at ($(c)+(.6,-.6)$);
	\draw (a) -- (b);
	\node at ($(b)+(-.7,.2)$) {{\scriptsize{$2n$}}};
	\draw (b) -- node [above] {{\scriptsize{$2n-1$}}} (c);
	\draw (c) -- (d);
	\node at ($(c)+(.7,.2)$) {{\scriptsize{$a_{j}$}}};
	\draw[thick, rho] (b)--(bc) -- (c);
	\draw[thick, rho] (bc) .. controls ++(270:.6cm) and ++(0:.6cm) .. (ab);
	\draw[thick, rho] ($(c)+(0,-.2)$)--($(c)+(1,-.2)$);
	\draw[] ($(c)+(0,0)$)-- ($(c)+(0,-.8)$);
	\Ucircle{U^*}{(b)}
	\Ucircle{U}{(c)}
	\filldraw[rho] (bc) circle (.05cm);
\end{tikzpicture}
\text{ where $a_j$ is even}
\end{align*}
When we use the two car trains with $2n$ horizontal strands, we have the following multiplication rules for contracting the $2n$ strands:
\begin{align*}
\twocar[x_1&,x_2,2n] \twocar[2n,x_3,x_4] \\
&=\begin{cases}
0 & \text{ if $x_2,x_3< 2n-1$}\\
\twocar[x_1,x_3,x_4] &\text{ if $x_2=2n-1$ and $x_3< 2n-1$}\\
\twocar[x_1,x_2,x_4] &\text{ if $x_2< 2n-1$ and $x_3=2n-1$}\\
-\tau^{-1}\twocar[x_1,2n-1,x_4] &\text{ if $x_2,x_3=2n-1$.}
\end{cases}
\end{align*}
The last identity follows from the fact that
$$
\begin{tikzpicture}[baseline = -.1cm]
	\draw[thick, unshaded] (-.6,-.4)--(-.6,.4)--(.6,.4)--(.6,-.4)--(-.6,-.4);
	\coordinate (a) at (0,.1) ;
	\coordinate (b) at (-.2,-.2) ;
	\coordinate (c) at (.2,-.2) ;
	\filldraw[rho] (a) circle (.05cm);
	\filldraw[rho] (b) circle (.05cm);
	\filldraw[rho] (c) circle (.05cm);
	\draw[thick, rho] (-.4,-.4) -- (b) -- (c) -- (.4,-.4);
	\draw[thick, rho] (b) -- (a) -- (c);	
	\draw[thick, rho] (a) -- (0, .4);
\end{tikzpicture}
=
-\tau^{-1}
\begin{tikzpicture}[baseline = -.1cm]
	\draw[thick, unshaded] (-.4,-.4)--(-.4,.4)--(.4,.4)--(.4,-.4)--(-.4,-.4);
	\coordinate (a) at (0,0) ;
	\filldraw[rho] (a) circle (.05cm);
	\draw[thick, rho] (-.2,-.4) -- (a) --  ( .2, -.4);	
	\draw[thick, rho] (a) -- (0, .4);
\end{tikzpicture}\,.
$$
\end{nota}

We now describe the nice jellyfish relations for these two car trains. The proofs of the following three lemmas are straightforward applications of Theorem \ref{thm:JellyfishRelationsU} and Corollary \ref{cor:JellyfishRelationsUstar}.

\begin{lem}\label{lem:ThetaJF}
The diagrams in the form of Proposition \ref{prop:basis} satisfy the following jellyfish relation for $\theta$ strings, where we assume the $a_{2j}$'s are even and the $a_{2j+1}$'s are odd:
\begin{align*}
\theta (\train[a_0, a_1,\dots&, a_{k-1}, a_k,\wheel[c]])
\\&=

\\
&=
\train[a_0+1, a_1-1,\dots, a_{k-1}-1, a_k+1,\wheel[c]].
\end{align*}
\end{lem}


\begin{lem}\label{lem:DoubleRhoJF}
When $d_1,d_2$ 
are odd and $e\leq 2n-2$ is even, then
$$
%
+X
$$
where $X$ is the diagram on the right in the table below according to the values of $d_1,d_2$. 
$$
\begin{array}{|c|c|}
\hline
d_1,d_2 & X\\\hline
d_1,d_2< 2n-1 
&
\begin{tikzpicture}[baseline = -.2cm]
	\coordinate (b) at (-2.2,0);
	\coordinate (c) at (-.6,0);
	\coordinate (d) at (.6,0);
	\coordinate (e) at (2.2,0);
	\coordinate (bc) at ($(b)+(.8,-.6)$);
	\coordinate (cd) at ($(c)+(.6,-.6)$);
	\coordinate (de) at ($(d)+(.8,-.6)$);
	\coordinate (ef) at ($(e)+(.7,-.6)$);
	\draw ($(b)+(-1,0)$) -- (b);
	\node at ($(b)+(-.9,.2)$) {{\scriptsize{$d_1+1$}}};
	\draw (b) -- node [above] {{\scriptsize{$2n-1$}}} (c);
	\draw (c) -- node [above] {{\scriptsize{$e$}}} (d);
	\draw (d) -- node [above] {{\scriptsize{$2n-1$}}} (e);
	\draw (e)--($(e)+(1.5,0)$);
	\node at ($(e)+(.9,.2)$) {{\scriptsize{$d_2+1$}}};
	\draw[thick, rho] ($(b)+(0,-.2)$)--($(b)+(-1,-.2)$);
	\draw[thick, rho] (b)--(bc) -- (c);
	\draw[thick, rho] (bc) -- ($(bc)+(0,-.2)$);
	\draw[thick, rho] (c)--(cd) -- (d);
	\draw[thick, rho] (cd) -- ($(cd)+(0,-.2)$);
	\draw[thick, rho] (d)--(de) -- (e);
	\draw[thick, rho] (de) -- ($(de)+(0,-.2)$);
	\draw[thick, rho] (e)--(ef) .. controls ++(45:.2cm) and ++(180:.6cm) .. ($(ef)+(.8,.4)$);
	\draw[thick, rho] (ef) -- ($(ef)+(0,-.2)$);
	\draw[] ($(b)+(0,0)$)-- ($(b)+(0,-.8)$);
	\draw[] ($(c)+(0,0)$)-- ($(c)+(0,-.8)$);
	\draw[] ($(d)+(0,0)$)-- ($(d)+(0,-.8)$);
	\draw[] ($(e)+(0,0)$)-- ($(e)+(0,-.8)$);
	\Ucircle{U_1}{(b)}
	\Ucircle{U_2}{(c)}
	\Ucircle{U_3}{(d)}
	\Ucircle{U_4}{(e)}
	\filldraw[rho] (bc) circle (.05cm);
	\filldraw[rho] (cd) circle (.05cm);
	\filldraw[rho] (de) circle (.05cm);
	\filldraw[rho] (ef) circle (.05cm);
\end{tikzpicture}
\\
\hline
\end{array}
$$
\end{lem}

\subsection{Existence results for $\AT_{n,\omega_U}$}

We now use our jellyfish relations to prove our existence results for the $\AT_{n,\omega_U}$. We can prove these results thanks to the train bases afforded by Proposition \ref{prop:basis}.

\begin{defn}
Consider the set of diagrams of the form of Proposition \ref{prop:basis} with $8n$ external boundary points, together with the one car train below:
$$
B_r = 
\left\{ 
%
\right\}
\subset P_{(\rho\zeta)^3\zeta\rho}.
$$
Then $B_r$ is orthogonal and linearly independent. We call $B_r$ the \underline{right train basis}.

Similarly we define the \underline{left train basis} by
$$
B_\ell=
\left\{ 
%
\right\}
\subset P_{\zeta\rho(\rho\zeta)^3}.
$$

Note that conjugating by $U$ maps $B_\ell$ onto $B_r$, i.e., given an element in $B_\ell$, if we put the diagram
$$
%
$$
underneath, we get an element of $B_r$.
\end{defn}

\begin{thm}\label{thm:Unique}
For $n=1,2,$ or $3$, $\AT_{n,\omega_U}$ exists only if $\omega_U=1$.
\end{thm}
\begin{proof}
We compute $\rho\zeta (U)$ in two ways:
\begin{equation}\label{eqn:Consistency}
\rho\zeta (U)=
%
=
U(\zeta\rho(U))U^*.
\end{equation}
Note by Lemmas \ref{lem:ThetaJF} and \ref{lem:DoubleRhoJF}, applying the $\rho$ and $\theta$ strings in the order on the left always gives us a linear combination of elements from $B_r$, the right train basis.
Similarly, applying the $\rho$ and $\theta$ strings in the order on the right always gives us a linear combination of elements from $B_\ell$, the left train basis.
Then conjugating by $U$ as in Equation \eqref{eqn:Consistency}, we get back some linear combination of elements in $B_r$.

In Figures \ref{fig:n=1}, \ref{fig:n=2}, and \ref{fig:n=3} in Appendix \ref{sec:FiguresForAT}, we give, for $n=1$, $n=2$, and $n=3$ respectively the tables of coefficients in the linear combinations for $\rho\zeta (U)$ and $U(\zeta\rho(U))U^*$ after applying the jellyfish relations. 

These coefficients can agree only if $\sigma_U=\sigma_U^{-1}$, and hence $\omega_U=1$.
\end{proof}

\begin{thm}\label{thm:Nonexistence}
$\AT_{n,\omega_U}$ does not exist for $4\leq n\leq 10$.
\end{thm}
\begin{proof}
The technique is the same as the proof of Theorem \ref{thm:Unique}. We compute $\rho\zeta(U)$ in two different ways as in Equation \eqref{eqn:Consistency}, and we get different linear combinations. Hence by Proposition \ref{prop:basis}, $\AT_{n,\omega_U}=0$ for $4\leq n\leq 10$.

We give the coefficients when $n=4$ in the table in Figure \ref{fig:n=4} in Appendix \ref{sec:FiguresForAT}. Similar computations for $5 \leq n \leq 10$ using the same code also results in different linear combinations.
\end{proof}

\begin{conj}
The technique used for Theorems \ref{thm:Unique} and \ref{thm:Nonexistence} should show
\begin{enumerate}[(1)]
\item
$\AT_{n,\omega_U}$ exists only if $\omega_U=1$ for all $1\leq n<\infty$, and
\item
$\AT_{n,\omega_U}$ does not exist for all $4\leq n<\infty$.
\end{enumerate} 
\end{conj}

\section{Application to subfactors at index $3+\sqrt{5}$}\label{sec:Subfactors}

We now discuss the connection between $\AT_{n,\omega_U}$ and subfactors at index $3+\sqrt{5}$.
In 1994, Bisch-Haagerup found a sequence of possible principal graphs which converge to the Fuss-Catalan principal graph at index $3+\sqrt{5}$.
\begin{align*}
\fish_1&=
\bigraph{bwd1v1p1p1v1x1x0duals1v1x2x3}
\\
\fish_2&=
\bigraph{bwd1v1p1v1x0p0x1v1x0p1x0p1x0p0x1v1x0x0x1v1duals1v1x2v1x2x4x3v1}
\\
\fish_3&=
\bigraph{bwd1v1p1v1x0p0x1v1x0p1x0p0x1v1x0x0p0x0x1v1x0p1x0p1x0p0x1p0x1v0x1x0x1x0v1duals1v1x2v1x3x2v1x2x4x3x5v1}
\\
\fish_n&=
\begin{tikzpicture}[scale=.5, baseline=-.1cm]
\draw[fill] (0,0) circle (0.05);
\draw (0.,0.) -- (1.,0.);
\draw[fill] (1.,0.) circle (0.05);
\draw (1.,0.) -- (2.,-0.125);
\draw (1.,0.) -- (2.,0.125);
\draw[fill] (2.,-0.125) circle (0.05);
\draw[fill] (2.,0.125) circle (0.05);
\draw (2.,-0.125) -- (3.,-0.125);
\draw (2.,0.125) -- (3.,0.125);
\draw[fill] (3.,-0.125) circle (0.05);
\draw[fill] (3.,0.125) circle (0.05);
\draw (3.,-0.125) -- (4.,-0.25);
\draw (3.,-0.125) -- (4.,0.);
\draw (3.,0.125) -- (4.,0.25);
\draw[fill] (4.,-0.25) circle (0.05);
\draw[fill] (4.,0.) circle (0.05);
\draw[fill] (4.,0.25) circle (0.05);
\draw (4.,-0.25) -- (5.,-0.125);
\draw (4.,0.25) -- (5.,0.125);
\draw[fill] (5.,-0.125) circle (0.05);
\draw[fill] (5.,0.125) circle (0.05);
\draw (5.,-0.125) -- (6.,-0.5);
\draw (5.,-0.125) -- (6.,-0.166667);
\draw (5.,0.125) -- (6.,0.166667);
\draw (5.,0.125) -- (6.,0.5);
\draw[fill] (6.,-0.5) circle (0.05);
\draw[fill] (6.,-0.166667) circle (0.05);
\draw[fill] (6.,0.166667) circle (0.05);
\draw[fill] (6.,0.5) circle (0.05);
\draw (6.,-0.5) -- (7.,-0.125);
\draw (6.,0.166667) -- (7.,0.125);
\draw[fill] (7.,-0.125) circle (0.05);
\draw[fill] (7.,0.125) circle (0.05);
\draw[dashed] (5,-.6) -- (5,.6) -- (7,.6) -- (7,-.6) -- (5,-.6);
\node at (7.5,0) {{\scriptsize{$\cdots$}}};
\draw[fill] (8.,-0.125) circle (0.05);
\draw[fill] (8.,0.125) circle (0.05);
\draw (8.,-0.125) -- (9.,-0.5);
\draw (8.,-0.125) -- (9.,-0.25);
\draw (8.,-0.125) -- (9.,0.);
\draw (8.,0.125) -- (9.,0.25);
\draw (8.,0.125) -- (9.,0.5);
\draw[fill] (9.,-0.5) circle (0.05);
\draw[fill] (9.,-0.25) circle (0.05);
\draw[fill] (9.,0.) circle (0.05);
\draw[fill] (9.,0.25) circle (0.05);
\draw[fill] (9.,0.5) circle (0.05);
\draw (9.,-0.25) -- (10.,0.);
\draw (9.,0.25) -- (10.,0.);
\draw[fill] (10.,0.) circle (0.05);
\draw (10.,0.) -- (11.,0.);
\draw[fill] (11.,0.) circle (0.05);
\draw[red, thick] (0.,0.) -- +(0,0.0833333) ;
\draw[red, thick] (2.,-0.125) -- +(0,0.0833333) ;
\draw[red, thick] (2.,0.125) -- +(0,0.0833333) ;
\draw[red, thick] (4.,-0.25) -- +(0,0.0833333) ;
\draw[red, thick] (4.,0.) to[out=135,in=-135] (4.,0.25);
\draw[red, thick] (6.,-0.5) -- +(0,0.0833333) ;
\draw[red, thick] (6.,-0.166667) to[out=135,in=-135] (6.,0.166667);
\draw[red, thick] (6.,0.5) -- +(0,0.0833333) ;
\draw[red, thick] (9.,-0.5) -- +(0,0.0833333) ;
\draw[red, thick] (9.,-0.25) -- +(0,0.0833333) ;
\draw[red, thick] (9.,0.) to[out=135,in=-135] (9.,0.25);
\draw[red, thick] (9.,0.5) -- +(0,0.0833333) ;
\draw[red, thick] (11.,0.) -- +(0,0.0833333) ;
\end{tikzpicture}
\end{align*}
For $n\geq 4$, the dashed section appears a total of $n-3$ times in $\fish_n$.


The main result of this section is the following theorem:
\begin{thm}\label{thm:ATandFish}
A subfactor with principal graph $\fish_n$ exists if and only if $\AT_{n,\omega_U}$ exists for some $\omega_U^{2n}=1$.
\end{thm}
\begin{proof}
Existence of the subfactor implies existence of such a fusion category by Theorem \ref{thm:FishToAT}.
The converse follows from Theorem \ref{thm:ATToFish}.
\end{proof}

\begin{cor}
A unique subfactor exists with principal graphs $\fish_n$ for $n=1,2,3$. No subfactor exists with principal graph $\fish_n$ for $4\leq n\leq 10$.
\end{cor}
\begin{proof}
By Theorem \ref{thm:ATandFish}, uniqueness and nonexistence follow from Theorems \ref{thm:Unique} and \ref{thm:Nonexistence} respectively. Subsection \ref{sec:Existence} shows existence for $n=1,2,3$.
\end{proof}

Independently, and by a different method, Zhengwei Liu showed that no subfactor with principal graph $\fish_n$ exists for \underline{any} $n\geq 4$ \cite{LiuFish}.
Liu's result together with Theorem \ref{thm:ATandFish} shows that $\AT_{n,\omega_U}$ does not exist for any $n\geq 4$.

\subsection{From subfactors to quotients of $A_2*T_2$}\label{sec:SubfactorsToAT}

Let $N\subset M$ be a 1-supertransitive subfactor at index $3+\sqrt{5}$ with an intermediate subfactor $P$. 

By taking duals, we may assume that $[M\colon P]=2$ and $[P\colon N]=\tau^2=\frac{3+\sqrt{5}}{2}$.
Denote the planar algebra for $N \subset M$ by $\cP_\bullet$ and the principal even half of $M-M$ bimodules by $\frac{1}{2}\cP_+$.
In this subsection, we show that $\frac{1}{2}\cP_+$ must be a quotient of $A_2*T_2$.

By \cite{MR1437496}, $\cP_\bullet$ has a Fuss-Catalan planar subalgebra $\cF\cC_\bullet$. Recall that $\cF\cC_{j,+}$ consists of all $A_3*A_4$ diagrams with boundaries of the form $ab(baab)^{j-1}ba$, where $a$ and $b$ are the usual generators of $A_3$ and $A_4$. Similarly, $\cF\cC_{j,+}$ of those diagrams with boundary $ba(abba)^{j-1}ab$. In the diagrams below we represent $a$ by a green string and $b$ by an orange string. 

\begin{defn}\label{defn:ThetaRho}
Define projections $\rho$ and $\theta$ in $\cF\cC_+\subset \cP_+$ by
\begin{align*}
\rho &=
\frac{1}{\sqrt{2}}\,
\begin{tikzpicture}[baseline = -.1cm]
	\draw[thick, A4] (-.35,-.8) -- ( -.35, .8);	
	\draw[thick, A4] (.35,-.8) -- (.35, .8);
	\draw[thick, A3] (-.2,.8) arc (-180:0:.2cm);
	\draw[thick, A3] (-.2,-.8) arc (180:0:.2cm);	
	\draw[thick, A4, unshaded] (-.5,-.4)--(-.5,.4)--(.5,.4)--(.5,-.4)--(-.5,-.4);
	\node at (0,0) {\Afour{$\jw{2}$}};
\end{tikzpicture}
\quad \text{ and } \quad
\theta =
\frac{1}{2}\,
\begin{tikzpicture}[baseline = -.1cm]
	\draw[thick, A3] (-.2,-.8) -- ( -.2, .8);	
	\draw[thick, A3] (.2,-.8) -- (.2, .8);
	\draw[thick, A3] (-1.35,.8) arc (-180:0:.15cm);
	\draw[thick, A3] (-1.35,-.8) arc (180:0:.15cm);
	\draw[thick, A3] (1.05,.8) arc (-180:0:.15cm);
	\draw[thick, A3] (1.05,-.8) arc (180:0:.15cm);
	\draw[thick, A4] (-1.5,-.8) -- ( -1.5, .8);
	\draw[thick, A4] (-.9,-.8) -- (-.9, .8);	
	\draw[thick, A4] (-.75,-.8) -- (-.75, .8);	
	\draw[thick, A4] (.9,-.8) -- (.9, .8);	
	\draw[thick, A4] (.75,-.8) -- (.75, .8);		
	\draw[thick, A4] (1.5,-.8) -- (1.5, .8);	
	\draw[thick, A3, unshaded] (-.4,-.4)--(-.4,.4)--(.4,.4)--(.4,-.4)--(-.4,-.4);
	\node at (0,0) {\Athree{$\jw{2}$}};
	\draw[thick, A4, unshaded] (-1.65,-.4)--(-1.65,.4)--(-.6,.4)--(-.6,-.4)--(-1.65,-.4);
	\node at (-1.125,0) {\Afour{$\jw{3}$}};
	\draw[thick, A4, unshaded] (1.65,-.4)--(1.65,.4)--(.6,.4)--(.6,-.4)--(1.65,-.4);
	\node at (1.125,0) {\Afour{$\jw{3}$}};
\end{tikzpicture}\,,
\end{align*}
\end{defn}
These correspond to $N-N$ bimodules, and so to (possibly a collection of) vertices on $\Gamma_+$.
Clearly $\rho$ is a minimal projection of trace $\tau$, and $\rho^2\cong 1\oplus \rho$.

\begin{lem}\label{lem:Nontrivial}
The projection $\theta$ satisfies $\theta\otimes \theta\cong 1$ and $\theta\ncong 1$.
\end{lem}
\begin{proof}
First, $\theta^2\cong 1$, since \textcolor{A4}{$\jw{3}$} and \textcolor{A3}{$\jw{2}$} have dimension 1. This follows from the fact that if $p$ is a trace 1 symmetrically self-dual projection in a fantastic planar algebra, then
\begin{equation}
\begin{tikzpicture}[baseline = -.1cm]
	\draw (-.6,-.8)--(-.6,.8);
	\draw (.6,-.8)--(.6,.8);
	\nbox{thick,unshaded}{(-.6,0)}{0}{0}{p}
	\nbox{thick,unshaded}{(.6,0)}{0}{0}{p}
\end{tikzpicture}
=
\begin{tikzpicture}[baseline = -.1cm]
	\draw (-.4,.6) arc (270:180:.4cm);
	\draw (.4,.6) arc (-90:0:.4cm);
	\draw (-.4,-.6) arc (90:180:.4cm);
	\draw (.4,-.6) arc (90:0:.4cm);
	\nbox{thick,unshaded}{(0,-.6)}{0}{0}{p}
	\nbox{thick,unshaded}{(0,.6)}{0}{0}{p}
\end{tikzpicture}
\label{eqn:Automorphism}
\end{equation}
(take the norm squared of the difference). Now let $x$ be an intertwiner from $\theta$ to the empty diagram. Then by sphericality, and the fact that $\dim(\textcolor{A4}{\jw{3}})=1$, we have
$$
\begin{tikzpicture}[baseline = -.1cm,scale=.9]
	\draw[thick, A3] (-.2,-.8) -- ( -.2, .8);	
	\draw[thick, A3] (.2,-.8) -- (.2, .8);
	\draw[thick, A3] (-1.35,.8) arc (-180:0:.15cm);
	\draw[thick, A3] (-1.35,-.8) arc (180:0:.15cm);
	\draw[thick, A3] (1.05,.8) arc (-180:0:.15cm);
	\draw[thick, A3] (1.05,-.8) arc (180:0:.15cm);
	\draw[thick, A4] (-1.5,-.8) -- ( -1.5, .8);
	\draw[thick, A4] (-.9,-.8) -- (-.9, .8);	
	\draw[thick, A4] (-.75,-.8) -- (-.75, .8);	
	\draw[thick, A4] (.9,-.8) -- (.9, .8);	
	\draw[thick, A4] (.75,-.8) -- (.75, .8);		
	\draw[thick, A4] (1.5,-.8) -- (1.5, .8);	
	\draw[thick, A3, unshaded] (-.4,-.4)--(-.4,.4)--(.4,.4)--(.4,-.4)--(-.4,-.4);
	\node at (0,0) {\Athree{$\jw{2}$}};
	\draw[thick, A4, unshaded] (-1.65,-.4)--(-1.65,.4)--(-.6,.4)--(-.6,-.4)--(-1.65,-.4);
	\node at (-1.125,0) {\Afour{$\jw{3}$}};
	\draw[thick, A4, unshaded] (1.65,-.4)--(1.65,.4)--(.6,.4)--(.6,-.4)--(1.65,-.4);
	\node at (1.125,0) {\Afour{$\jw{3}$}};
	\nbox{thick,unshaded}{(0,1.2)}{1.25}{1.25}{x}
	\nbox{thick,unshaded}{(0,-1.2)}{1.25}{1.25}{x^*}
\end{tikzpicture}
=
\begin{tikzpicture}[baseline = -.1cm,scale=.9]
	\draw[thick, A3] (-.2,-.8) -- ( -.2, .8);	
	\draw[thick, A3] (.2,-.8) -- (.2, .8);
	\draw[thick, A3] (-1.35,.8) arc (-180:0:.15cm);
	\draw[thick, A3] (-1.35,-.8) arc (180:0:.15cm);
	\draw[thick, A3] (1.05,.8) arc (-180:0:.15cm);
	\draw[thick, A3] (1.05,-.8) arc (180:0:.15cm);
	\draw[thick, A4] (-1.5,-.8) -- ( -1.5, .8);
	\draw[thick, A4] (-.9,-.8) -- (-.9, .8);	
	\draw[thick, A4] (-.75,-.8) -- (-.75, .8);	
	\draw[thick, A4] (.75,-.8) arc (180:0:.7cm) -- (2.15,-1.4)  .. controls ++(270:1.35cm) and ++(270:1.25cm) ..  (-2.75,-1.2) -- (-2.75,-.4);	
	\draw[thick, A4] (.9,-.8) arc (180:0:.55cm) -- (2,-1.4)  .. controls ++(270:1.2cm) and ++(270:1.1cm) ..  (-2.6,-1.2) -- (-2.6,-.4);	
	\draw[thick, A4] (1.5,-.8) arc (180:0:.15cm) -- (1.8,-1.4)  .. controls ++(270:1cm) and ++(270:1cm) ..  (-2,-1.2) -- (-2,-.4);	
	\draw[thick, A4] (.75,.8) arc (-180:0:.7cm) -- (2.15,1.4)  .. controls ++(90:1.35cm) and ++(90:1.25cm) ..  (-2.75,1.2) -- (-2.75,.4);	
	\draw[thick, A4] (.9,.8) arc (-180:0:.55cm) -- (2,1.4)  .. controls ++(90:1.2cm) and ++(90:1.1cm) ..  (-2.6,1.2) -- (-2.6,.4);	
	\draw[thick, A4] (1.5,.8) arc (-180:0:.15cm) -- (1.8,1.4)  .. controls ++(90:1cm) and ++(90:1cm) ..  (-2,1.2) -- (-2,.4);
	\draw[thick, A3, unshaded] (-.4,-.4)--(-.4,.4)--(.4,.4)--(.4,-.4)--(-.4,-.4);
	\node at (0,0) {\Athree{$\jw{2}$}};
	\draw[thick, A4, unshaded] (-1.65,-.4)--(-1.65,.4)--(-.6,.4)--(-.6,-.4)--(-1.65,-.4);
	\node at (-1.125,0) {\Afour{$\jw{3}$}};
	\draw[thick, A4, unshaded] (-1.85,-.4)--(-1.85,.4)--(-2.9,.4)--(-2.9,-.4)--(-1.85,-.4);
	\node at (-2.325,0) {\Afour{$\jw{3}$}};
	\nbox{thick,unshaded}{(0,1.2)}{1.25}{1.25}{x}
	\nbox{thick,unshaded}{(0,-1.2)}{1.25}{1.25}{x^*}
\end{tikzpicture}
=
\begin{tikzpicture}[baseline = -.1cm,scale=.9]
	\draw[thick, A3] (-.2,-.8) -- ( -.2, .8);	
	\draw[thick, A3] (.2,-.8) -- (.2, .8);
	\draw[thick, A3] (-1.35,.8) arc (-180:0:.15cm);
	\draw[thick, A3] (-1.35,-.8) arc (180:0:.15cm);
	\draw[thick, A3] (1.05,.8) arc (-180:0:.15cm);
	\draw[thick, A3] (1.05,-.8) arc (180:0:.15cm);
	\draw[thick, A4] (.75,-.8) arc (180:0:.7cm) -- (2.15,-1.4)  .. controls ++(270:1.35cm) and ++(270:1.35cm) ..  (-2.15,-1.4) -- (-2.15,-.8) arc (180:0:.7cm);
	\draw[thick, A4] (.9,-.8) arc (180:0:.55cm) -- (2,-1.4)  .. controls ++(270:1.1cm) and ++(270:1.1cm) ..  (-2,-1.4) -- (-2,-.8) arc (180:0:.55cm);	
	\draw[thick, A4] (1.5,-.8) arc (180:0:.15cm) -- (1.8,-1.4)  .. controls ++(270:.9cm) and ++(270:.9cm)  ..  (-1.8,-1.4) -- (-1.8,-.8) arc (180:0:.15cm);
	\draw[thick, A4] (.75,.8) arc (-180:0:.7cm) -- (2.15,1.4)  .. controls ++(90:1.35cm) and ++(90:1.35cm) ..  (-2.15,1.4) -- (-2.15,.8) arc (-180:0:.7cm);	
	\draw[thick, A4] (.9,.8) arc (-180:0:.55cm) -- (2,1.4)  .. controls ++(90:1.1cm) and ++(90:1.1cm)  ..  (-2,1.4) -- (-2,.8) arc (-180:0:.55cm);		
	\draw[thick, A4] (1.5,.8) arc (-180:0:.15cm) -- (1.8,1.4)  .. controls ++(90:.9cm) and ++(90:.9cm)  ..  (-1.8,1.4) -- (-1.8,.8) arc (-180:0:.15cm);	
	\draw[thick, A3, unshaded] (-.4,-.4)--(-.4,.4)--(.4,.4)--(.4,-.4)--(-.4,-.4);
	\node at (0,0) {\Athree{$\jw{2}$}};
	\nbox{thick,A4,unshaded}{(0,2.2)}{0}{0}{\rotatebox{90}{$\Afour{\jw{3}}$}}
	\nbox{thick,A4,unshaded}{(0,-2.2)}{0}{0}{\rotatebox{90}{$\Afour{\jw{3}}$}}
	\nbox{thick,unshaded}{(0,1.2)}{1.25}{1.25}{x}
	\nbox{thick,unshaded}{(0,-1.2)}{1.25}{1.25}{x^*}
\end{tikzpicture}
=
0,
$$
since any intertwiner from \textcolor{A3}{$\jw{2}$} to the empty diagram must be zero.
\end{proof}

\begin{prop}\label{prop:DepthTwo}
In $FC_\bullet$, $\jw{2}\ominus \rho \cong \rho\theta\rho$.
\end{prop}
\begin{proof}
We have
$$
\jw{2}-\rho
= 
\begin{tikzpicture}[baseline = -.1cm, scale=1.5]
	\draw[thick, A4] (-.225,-.4) -- (-.225,.4);
	\draw[thick, A3] (-.075,-.4) -- (-.075,.4);
	\draw[thick, A3] (.075,-.4) -- (.075,.4);
	\draw[thick, A4] (.225,-.4) -- (.225,.4);
\end{tikzpicture}
-
\frac{1}{\tau\sqrt{2}}\,
\begin{tikzpicture}[baseline = -.1cm, scale=1.5]
	\draw[thick, A4] (-.33,.4) arc (-180:0:.33cm);	
	\draw[thick, A4] (-.33,-.4) arc (180:0:.33cm);	
	\draw[thick, A3] (-.2,.4) arc (-180:0:.2cm);
	\draw[thick, A3] (-.2,-.4) arc (180:0:.2cm);	
\end{tikzpicture}
-
\frac{1}{\sqrt{2}}\,
\begin{tikzpicture}[baseline = -.1cm]
	\draw[thick, A4] (-.35,-.8) -- ( -.35, .8);	
	\draw[thick, A4] (.35,-.8) -- (.35, .8);
	\draw[thick, A3] (-.2,.8) arc (-180:0:.2cm);
	\draw[thick, A3] (-.2,-.8) arc (180:0:.2cm);	
	\draw[thick, A4, unshaded] (-.5,-.4)--(-.5,.4)--(.5,.4)--(.5,-.4)--(-.5,-.4);
	\node at (0,0) {\Afour{$\jw{2}$}};
\end{tikzpicture}
=
\begin{tikzpicture}[baseline = -.1cm]
	\draw[thick, A4] (-.6,-.8) -- ( -.6, .8);	
	\draw[thick, A4] (.6,-.8) -- (.6, .8);
	\draw[thick, A3] (-.2,-.8) -- ( -.2, .8);	
	\draw[thick, A3] (.2,-.8) -- (.2, .8);
	\draw[thick, A3, unshaded] (-.4,-.4)--(-.4,.4)--(.4,.4)--(.4,-.4)--(-.4,-.4);
	\node at (0,0) {\Athree{$\jw{2}$}};
\end{tikzpicture}.
$$
It is then easy to see that $\rho\theta\rho\cong \jw{2}\ominus\rho$, since $\Afour{\jw{1}}\cong \Afour{\jw{2}}\Afour{\jw{3}}\cong \Afour{\jw{3}}\Afour{\jw{2}}$.
\end{proof}

\begin{cor}\label{cor:EvenHalfOfFish}
The even half $\frac{1}{2}\cP_+$ of $P_\bullet$ is generated by $\rho$ and $\theta$. Hence $\frac{1}{2}\cP_+$ is either $A_2*T_2$ or $\AT_{n,\omega_U}$ for some $1\leq n< \infty$ and some $2n$-th root of unity $\omega_U$.
\end{cor}
\begin{proof}
Note that all of the $N-N$ bimodules are summands of a tensor power of $\jw{2}\cong \rho\oplus\rho\theta\rho$, and thus every $N-N$ bimodule is a summand of some alternating word in $\rho,\theta$. Hence, by sending $\theta \in A_2$ to $\theta\in \frac{1}{2}\cP_{+}$ (defined above) and $\rho\in T_2$ to $\rho\in \frac{1}{2}\cP_+$, we get a  dominant functor $F\colon A_2*T_2\to \frac{1}{2}\cP_+.$  This functor is faithful because $A_2 * T_2 \iso \frac{1}{2}\cF\cC_+$ and $\cF\cC_\bullet$ is a planar subalgebra of $\cP_\bullet$.
Thus $\frac{1}{2}\cP_+$ is a quotient of $A_2*T_2$. 
\end{proof}

We now show that any subfactor with principal graph $\fish_n$ must have an intermediate subfactor, and thus its even half must be $\AT_{n,\omega_U}$ for some $2n$-th root of unity $\omega_U$.

We provide a planar algebraic proof of the following lemma for the convenience of the reader.
In fact, there are many stronger versions well known to experts, but we do not need them at this time.

\begin{lem}\label{lem:Normalizer}
Suppose the subfactor $N\subset M$ has planar algebra $\cP_\bullet$ and principal graph $\Gamma_+$. Suppose $\Gamma_+$ is 1 supertransitive, has depth greater than 2, and has exactly one univalent (self-dual) vertex $\beta$ at depth 2. Then $e_1+\beta$ is a biprojection \cite{MR1262294,MR1950890}, so there is an intermediate subfactor $N\subset P\subset M$ where $[P:N]=2$.
\end{lem}
\begin{proof}
First, since $\beta e_1=0$ in $\cP_{2,+}$, it is clear $e_1+\beta$ is a projection. Denoting the coproduct of $x,y$ by $x\circ y$, it is easy to see that $e_1\circ x=\delta^{-1}x$ for $x=e_1,\beta$. We compute using Equation \eqref{eqn:Automorphism} that since $\dim(\beta)=1$,
$$
\beta\circ\beta = 
\begin{tikzpicture}[baseline = -.1cm]
	\draw (-.8,-.8)--(-.8,.8);
	\draw (.8,-.8)--(.8,.8);
	\draw (-.4,.4) arc (180:0:.4cm);
	\draw (-.4,-.4) arc (-180:0:.4cm);
	\nbox{thick,unshaded}{(-.6,0)}{0}{0}{\beta}
	\nbox{thick,unshaded}{(.6,0)}{0}{0}{\beta}
\end{tikzpicture}
=
\begin{tikzpicture}[baseline = -.1cm]
	\draw (-.4,.4) arc (270:180:.6cm);
	\draw (.4,.4) arc (-90:0:.6cm);
	\draw (-.4,-.4) arc (90:180:.6cm);
	\draw (.4,-.4) arc (90:0:.6cm);
	\draw (-.4,.8) arc (270:90:.2cm) -- (.4,1.2) arc (90:-90:.2cm);
	\draw (-.4,-.8) arc (90:270:.2cm) -- (.4,-1.2) arc (-90:90:.2cm);
	\node at (-.55,.3) {$\star$};
	\node at (-.55,-.3) {$\star$};
	\nbox{thick,unshaded}{(0,-.6)}{0}{0}{\beta}
	\nbox{thick,unshaded}{(0,.6)}{0}{0}{\beta}
\end{tikzpicture}
=
\frac{1}{\delta^2}
\begin{tikzpicture}[baseline = -.1cm]
	\draw (-.25,.4) arc (-180:0:.25cm);
	\draw (-.25,-.4) arc (180:0:.25cm);
\end{tikzpicture}
=
\delta^{-1}e_1,
$$
and thus $(e_1+\beta)\circ (e_1+\beta)=\delta^{-1}(e_1+\beta)$, and $e_1+\beta$ is a biprojection.
\end{proof}

\begin{thm}\label{thm:FishToAT}
If the principal graph of $N\subset M$ is $\fish_n$, then there is an intermediate subfactor $N\subset P\subset M$ such that $[M:P]=2$ and $[P:N]=\frac{3+\sqrt{5}}{2}$. Hence the even half of $N\subset M$ is necessarily $\AT_{n,\omega_U}$ for some $\omega_U^{2n} = 1$.
\end{thm}
\begin{proof}
If $n=1$, then the dual graph $\Gamma_-$ must be one of
$$
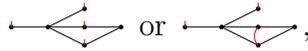
\bigraph{bwd1v1p1p1v1x1x0duals1v1x2x3}
\text{ or }
\bigraph{bwd1v1p1p1v1x1x0duals1v2x1x3},
$$
and thus there is an intermediate subfactor with the desired indices by applying Lemma \ref{lem:Normalizer} to the dual subfactor. (In fact, the dual graph cannot be the second graph above, since the dual even half must also be $\AT_{2,\omega_U}$, which only exists if $\omega_U=1$ by Theorem \ref{thm:Unique}.)

If $n\geq 2$, since $\fish_n$ starts with a triple point, the dual graph $\Gamma_-$ also starts with a triple point, and by Ocneanu's triple point obstruction \cite{MR1317352}, $\Gamma_-$ has a univalent vertex at depth 2. By applying Lemma \ref{lem:Normalizer} to the dual subfactor, we see that $N\subset M$ has an intermediate subfactor $P$ with the desired indices. 

Now that we know there is an intermediate subfactor, Corollary \ref{cor:EvenHalfOfFish} implies that the even half of $N\subset M$ must be $\AT_{k,\omega_U}$ for some $2k$-th root of unity $\omega_U$. By Proposition \ref{prop:DistinctIrreducibleBimodules}, it suffices to count the even vertices of $\fish_n$ to see that $k=n$.
\end{proof}

\subsection{From quotients of $A_2*T_2$ to subfactors}\label{sec:ATToSubfactors}

\begin{prop}\label{prop:AlgebraObject}
In $A_2*T_2$, $A=1\oplus \rho\oplus \rho\theta\rho$ is a Frobenius algebra object with Frobenius subalgebra object $B=1\oplus \rho$. 
\end{prop}
\begin{proof}
First, it is well known that $B$ is an algebra object, but we provide a proof as a warmup to showing $A$ is an algebra. We need to specify the map $B\otimes B\to B$, which can be thought of as 8 maps between the summands. Since the map must be unital and rotationally invariant, we only have one unknown parameter:
\begin{align*}
\text{maps from $X\otimes Y \to 1$: }&
\begin{array}{|c|c|c|}
\hline
\otimes & 1 & \rho
\\\hline
1 &
\begin{tikzpicture}[baseline = -.1cm]
	\clip (-.45,-.45) rectangle (.45,.45);
	\nbox{thick}{(0,0)}{0}{0}{}
\end{tikzpicture}
& 0
\\\hline
\rho & 0 & 
\begin{tikzpicture}[baseline = -.1cm]
	\clip (-.45,-.45) rectangle (.45,.45);
	\nbox{thick}{(0,0)}{0}{0}{}
	\draw[thick, rho] (-.25,-.4) arc (180:0:.25cm);
\end{tikzpicture}
\\\hline
\end{array}
\displaybreak[1]\\
\text{maps from $X\otimes Y \to \rho$: }&
\begin{array}{|c|c|c|}
\hline
\otimes & 1 & \rho
\\\hline
1 & 0 & 
\begin{tikzpicture}[baseline = -.1cm]
	\clip (-.45,-.45) rectangle (.45,.45);
	\nbox{thick}{(0,0)}{0}{0}{}
	\draw[thick, rho] (.2,-.4) -- (0,.4);
\end{tikzpicture}
\\\hline
\rho &  
\begin{tikzpicture}[baseline = -.1cm]
	\clip (-.45,-.45) rectangle (.45,.45);
	\nbox{thick}{(0,0)}{0}{0}{}
	\draw[thick, rho] (-.2,-.4) -- (0,.4);
\end{tikzpicture}
& 
\lambda\,
\begin{tikzpicture}[baseline = -.1cm]
	\clip (-.45,-.45) rectangle (.45,.45);
	\filldraw[rho] (0,0) circle (.05cm);
	\draw[thick, rho] (0,0) -- (-.2,-.4);
	\draw[thick, rho] (0,0) -- (0,.4);
	\draw[thick, rho] (0,0) -- (.2,-.4);
	\nbox{}{(0,0)}{0}{0}{}
\end{tikzpicture}
\\\hline
\end{array}
\end{align*}
for some constant $\lambda\in\Complex$. 
Checking associativity amounts to checking associativity of $\rho\otimes \rho\otimes \rho\to \rho\otimes \rho\to 1$, which yields the following equation: 
$$
\begin{tikzpicture}[baseline = -.1cm]
	\draw[thick, rho] (-.3,-.4) -- (0,.4);
	\draw[thick, rho] (-.1,-.4) arc (180:0:.2cm);	
	\nbox{}{(0,0)}{.1}{.1}{}
\end{tikzpicture}
+
\lambda^2\,
\begin{tikzpicture}[baseline = -.1cm]
	\coordinate (a) at (-.1,.1);
	\coordinate (b) at (.1,-.15);
	\filldraw[rho] (a) circle (.05cm);
	\draw[thick, rho] (a) -- (-.3,-.4);
	\draw[thick, rho] (a) -- (-.1,.4);
	\filldraw[rho] (b) circle (.05cm);
	\draw[thick, rho] (a) -- (b);
	\draw[thick, rho] (b) -- (.3,-.4);
	\draw[thick, rho] (b) -- (0,-.4);
	\nbox{}{(0,0)}{.1}{.1}{}
\end{tikzpicture}
=
\begin{tikzpicture}[baseline = -.1cm]
	\draw[thick, rho] (.3,-.4) -- (0,.4);
	\draw[thick, rho] (-.3,-.4) arc (180:0:.2cm);	
	\nbox{}{(0,0)}{.1}{.1}{}
\end{tikzpicture}
+
\lambda^2\,
\begin{tikzpicture}[baseline = -.1cm]
	\coordinate (a) at (.1,.1);
	\coordinate (b) at (-.1,-.15);
	\filldraw[rho] (a) circle (.05cm);
	\draw[thick, rho] (a) -- (.3,-.4);
	\draw[thick, rho] (a) -- (.1,.4);
	\filldraw[rho] (b) circle (.05cm);
	\draw[thick, rho] (a) -- (b);
	\draw[thick, rho] (b) -- (-.3,-.4);
	\draw[thick, rho] (b) -- (0,-.4);
	\nbox{}{(0,0)}{.1}{.1}{}
\end{tikzpicture}\,.
$$
This equation is satisfied whenever $\lambda=\pm\tau^{-1/2}$ by Relation \eqref{rel:AT1}.

We now specify the map $A\otimes A\to A$ for $A=1\oplus\rho \oplus \rho\theta\rho$ by specifying maps between the summands as before. We already know one constraint if $B$ is a subalgebra.
\begin{align*}
\text{maps from $X\otimes Y \to 1$: }&
\begin{array}{|c|c|c|c|}
\hline
\otimes & 1 & \rho & \rho\theta\rho
\\\hline
1 & 
\begin{tikzpicture}[baseline = -.1cm]
	\clip (-.45,-.45) rectangle (.45,.45);
	\nbox{thick}{(0,0)}{0}{0}{}
\end{tikzpicture}
& 0 & 0
\\\hline
\rho & 0 & 
\begin{tikzpicture}[baseline = -.1cm]
	\clip (-.45,-.45) rectangle (.45,.45);
	\nbox{thick}{(0,0)}{0}{0}{}
	\draw[thick, rho] (-.25,-.4) arc (180:0:.25cm);
\end{tikzpicture}
& 0
\\\hline
\rho\theta\rho &  
0 & 0 &
\begin{tikzpicture}[baseline = -.1cm]
	\clip (-.45,-.45) rectangle (.45,.45);
	\nbox{thick}{(0,0)}{0}{0}{}
	\draw[thick, rho] (-.3,-.4) arc (180:0:.3cm);
	\draw[thick, theta] (-.2,-.4) arc (180:0:.2cm);
	\draw[thick, rho] (-.1,-.4) arc (180:0:.1cm);
\end{tikzpicture}
\\\hline
\end{array}
\displaybreak[1]\\
\text{maps from $X\otimes Y \to \rho$: }&
\begin{array}{|c|c|c|c|}
\hline
\otimes & 1 & \rho & \rho\theta\rho
\\\hline
1 & 0 & 
\begin{tikzpicture}[baseline = -.1cm]
	\clip (-.45,-.45) rectangle (.45,.45);
	\nbox{thick}{(0,0)}{0}{0}{}
	\draw[thick, rho] (.2,-.4) -- (0,.4);
\end{tikzpicture}
& 0
\\\hline
\rho &  
\begin{tikzpicture}[baseline = -.1cm]
	\clip (-.45,-.45) rectangle (.45,.45);
	\nbox{thick}{(0,0)}{0}{0}{}
	\draw[thick, rho] (-.2,-.4) -- (0,.4);
\end{tikzpicture}
& 
\frac{\pm1}{\sqrt{\tau}}\,
\begin{tikzpicture}[baseline = -.1cm]
	\clip (-.45,-.45) rectangle (.45,.45);
	\filldraw[rho] (0,0) circle (.05cm);
	\draw[thick, rho] (0,0) -- (-.2,-.4);
	\draw[thick, rho] (0,0) -- (0,.4);
	\draw[thick, rho] (0,0) -- (.2,-.4);
	\nbox{}{(0,0)}{0}{0}{}
\end{tikzpicture}
& 0
\\\hline
\rho\theta\rho &  
0 & 0 & 
\lambda\,
\begin{tikzpicture}[baseline = -.1cm]
	\clip (-.45,-.45) rectangle (.45,.45);
	\filldraw[rho] (0,0) circle (.05cm);
	\draw[thick, rho] (0,0) -- (-.3,-.4);
	\draw[thick, rho] (0,0) -- (0,.4);
	\draw[thick, rho] (0,0) -- (.3,-.4);
	\draw[thick, theta] (-.175,-.4) arc (180:0:.175cm);
	\draw[thick, rho] (-.1,-.4) arc (180:0:.1cm);
	\nbox{}{(0,0)}{0}{0}{}
\end{tikzpicture}
\\\hline
\end{array}
\displaybreak[1]\\
\text{maps from $X\otimes Y \to \rho\theta\rho$: }&
\begin{array}{|c|c|c|c|}
\hline
\otimes & 1 & \rho & \rho\theta\rho
\\\hline
1 & 0 & 0 &
\begin{tikzpicture}[baseline = -.1cm]
	\clip (-.45,-.45) rectangle (.45,.45);
	\nbox{thick}{(0,0)}{0}{0}{}
	\draw[thick, rho] (.1,-.4) -- (-.1,.4);
	\draw[thick, theta] (.2,-.4) -- (0,.4);
	\draw[thick, rho] (.3,-.4) -- (.1,.4);
\end{tikzpicture}
\\\hline
\rho &  
0 & 0 &
\lambda\,
\begin{tikzpicture}[baseline = -.1cm]
	\clip (-.45,-.45) rectangle (.45,.45);
	\coordinate (a) at (-.1,0);
	\filldraw[rho] (a) circle (.05cm);
	\draw[thick, rho] (a) -- (-.3,-.4);
	\draw[thick, rho] (a) -- (-.1,.4);
	\draw[thick, rho] (a) -- (.1,-.4);
	\nbox{}{(0,0)}{0}{0}{}
	\draw[thick, theta] (.2,-.4) -- (0,.4);
	\draw[thick, rho] (.3,-.4) -- (.1,.4);
\end{tikzpicture}
\\\hline
\rho\theta\rho &  
\begin{tikzpicture}[baseline = -.1cm]
	\clip (-.45,-.45) rectangle (.45,.45);
	\nbox{thick}{(0,0)}{0}{0}{}
	\draw[thick, rho] (-.3,-.4) -- (-.1,.4);
	\draw[thick, theta] (-.2,-.4) -- (0,.4);
	\draw[thick, rho] (-.1,-.4) -- (.1,.4);
\end{tikzpicture}
& 
\lambda\,
\begin{tikzpicture}[baseline = -.1cm]
	\clip (-.45,-.45) rectangle (.45,.45);
	\coordinate (a) at (.1,0);
	\filldraw[rho] (a) circle (.05cm);
	\draw[thick, rho] (a) -- (.3,-.4);
	\draw[thick, rho] (a) -- (.1,.4);
	\draw[thick, rho] (a) -- (-.1,-.4);
	\nbox{}{(0,0)}{0}{0}{}
	\draw[thick, theta] (-.2,-.4) -- (0,.4);
	\draw[thick, rho] (-.3,-.4) -- (-.1,.4);
\end{tikzpicture}
& 0
\\\hline
\end{array}
\end{align*}
We get the following constraint
$$
\begin{tikzpicture}[baseline = -.1cm]
	\clip (-.45,-.45) rectangle (.45,.45);
	\nbox{thick}{(0,0)}{0}{0}{}
	\draw[thick, rho] (.1,-.4) -- (-.1,.4);
	\draw[thick, theta] (.2,-.4) -- (0,.4);
	\draw[thick, rho] (.3,-.4) -- (.1,.4);
	\draw[thick, rho] (-.3,-.4) arc (180:0:.15cm);
\end{tikzpicture}
\pm
\frac{\lambda}{\sqrt{\tau}}\,
\begin{tikzpicture}[baseline = -.1cm]
	\clip (-.55,-.45) rectangle (.75,.45);
	\coordinate (a) at (-.15,-.15);
	\coordinate (b) at (.1,.1);
	\filldraw[rho] (a) circle (.05cm);
	\draw[thick, rho] (a) -- (-.3,-.4);
	\draw[thick, rho] (a) -- (0,-.4);
	\filldraw[rho] (b) circle (.05cm);
	\draw[thick, rho] (a) -- (b);
	\draw[thick, rho] (b) -- (.25,-.4);
	\draw[thick, rho] (b) -- (0,.4);
	\draw[thick, theta] (.4,-.4) -- (.2,.4);
	\draw[thick, rho] (.5,-.4) -- (.3,.4);
	\nbox{}{(0,0)}{.1}{.3}{}
\end{tikzpicture}
=
\lambda^2\,
\begin{tikzpicture}[baseline = -.1cm]
	\clip (-.55,-.45) rectangle (.75,.45);
	\coordinate (a) at (0,.1);
	\coordinate (b) at (.15,-.15);
	\filldraw[rho] (a) circle (.05cm);
	\draw[thick, rho] (a) -- (-.3,-.4);
	\draw[thick, rho] (a) -- (-.1,.4);
	\filldraw[rho] (b) circle (.05cm);
	\draw[thick, rho] (a) -- (b);
	\draw[thick, rho] (b) -- (.3,-.4);
	\draw[thick, rho] (b) -- (0,-.4);
	\draw[thick, theta] (.4,-.4) -- (.2,.4);
	\draw[thick, rho] (.5,-.4) -- (.3,.4);
	\nbox{}{(0,0)}{.1}{.3}{}
\end{tikzpicture}\,,
$$
which is satisfied if $\lambda=\mp\sqrt{\tau}$. We leave it to the reader that this restriction is sufficient for the map $A\otimes A\to A$ to be associative.
\end{proof}

\begin{thm}\label{thm:ATToFish}
Suppose $\AT_{n,\omega_U}$ exists for some $\omega_U^{2n}=1$. Then there are {\rm II}$_1$-factors $N\subset P\subset M$ where $[M\colon P]=2$ and $[P\colon N]=\frac{3+\sqrt{5}}{2}$ such that the even half of $N\subset M$ is $\AT_{n,\omega_U}$, and the principal graph of $N\subset M$ is $\fish_n$.
\end{thm}
\begin{proof}
By Proposition \ref{prop:AlgebraObject}, $1\oplus \rho\oplus \rho\theta\rho$ is a Frobenius algebra object with subalgebra $1\oplus \rho$ in $A_2*T_2$, and thus they are also algebra objects in $\AT_{n,\omega_U}$. 
Now the usual construction (see Remark \ref{rem:AlgebrasAndSF}) provides a subfactor $N \subset M$ with $\jw{2}=\rho\oplus\rho\theta\rho$. 

A straightforward calculation shows that the fusion graph of $\rho\oplus\rho\theta\rho$ in $\AT_{n,\omega_U}$ is the same as the even part of $\fish_n$.
We give the fusion graph for $n=1,2,3$ below, where we use the convention that vertices with no red lines attached are self-dual.
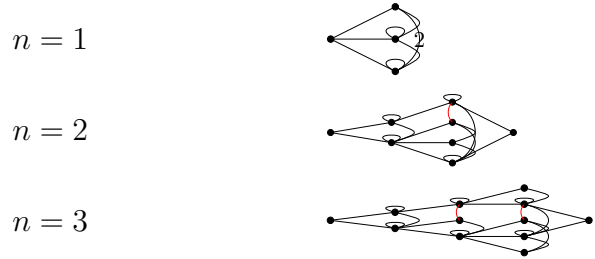
\begin{align*}
n&=1 && \graph{fg1v0v1p1p1v1x2x3v1p2x1p1x1x0}\\
n&=2 && \graph{fg1v0v1p1v1x2v1p1x1v1x0p1x0p1x0p0x1v1x2x4x3v1p1x0p1x1x0p1x0x0x1v1x0x0x1v1v0}\\
n&=3 && \graph{fg1v0v1p1v1x2v1p1x1v1x0p1x0p0x1v1x3x2v1p1x0p0x0x1v1x0x0p1x0x0p1x0x0p0x0x1p0x0x1v1x2x4x3x5v0p1x1p1x1x0p0x1x0x1p0x0x0x1x0v0x1x0x1x0v1v0}
\end{align*}

Finally, we can show that $\fish_n$ is the unique principal graph with this even part.
First, we note that $N\subset M$ is irreducible. Since $[M: P]=2$ and $[P:N]=\frac{3+\sqrt{5}}{2}$, $\dim(\jw{1})=\sqrt{2}\tau$, which cannot be written as the sum of two numbers from the set $\set{2\cos(\pi/k)}{k\geq 3}$.
Next, since $\jw{1}$ is simple and $\jw{1}\otimes \jw{1}\cong 1\oplus \jw{2}$, the number of self-loops on a vertex in the even principal graph is exactly one more than the valence of that vertex in the principal graph. This condition uniquely determines the number of vertices at each odd depth, and their connectivity to the vertices at even depths.
\end{proof}

\begin{remark}\label{rem:AlgebrasAndSF}
Frobenius algebra objects in unitary fusion categories correspond to finite depth subfactor planar algebras by \cite{MR1257245,MR1444286,MR1966524,MR1976459,MR1960417,MR2091457}. See \cite[Section 2]{MR2909758} for a good background and a dictionary between the two viewpoints.

The articles \cite{MR1257245,MR1444286,MR1966524} work with type {\rm III}$_1$ factors. See \cite{MR1269266,MR1339767} to translate between the type {\rm II}$_1$ and {\rm III}$_1$ cases for finite index, finite depth subfactors.
\end{remark}

\subsection{Existence of $\AT_{n,1}$ for $n=1,2$, and $3$}\label{sec:Existence}
In this section, we show that for $n=1,2$, and $3$, $\AT_{n,1}$ exists. 
Hence by Theorems \ref{thm:Unique} and \ref{thm:ATToFish}, there is a unique hyperfinite subfactor whose principal graph is $\fish_n$ for $n=1,2$, and $3$.
First, $\AT_{1,1}=A_2\boxtimes T_2$, which exists.
For $n=2,3$, we first construct $\fish_n$, and its even half is necessarily $\AT_{n,1}$ by Theorems \ref{thm:Unique} and \ref{thm:FishToAT}.

Bisch and Haagerup first constructed a subfactor with principal graph $\fish_2$. 
We work out the details below for completeness and for the convenience of the reader. 

Suppose $N_0\subset N_1$ is the hyperfinite $A_4$ subfactor, and let $\beta\in \Out(N_1\otimes N_1)$ be the flip automorphism.
We write
$M_0=N_0\otimes N_1$ and
$M_1=(N_1\otimes N_1)\rtimes \Integer/2$, where $\Integer/2=\langle\beta\rangle$.
We work with the $M_0-M_0$ bimodules to stay consistent with the previous two subsections.
(One can also show that the category of $P-P$ bimodules where $P=N_1\otimes N_1$ is equivalent to $\AT_{2,1}$.)

\begin{prop}
The basic construction of $M_0\subset M_1$ is given by
$$
M_2= N_2\otimes N_1 \otimes B(\ell^2(\Integer/2))\cong \Mat_{2}(N_2\otimes N_1)
$$
where the inclusion $M_1\to M_2$ is given by
\begin{align*}
x\otimes y &\mapsto 
\begin{pmatrix}
x\otimes y & 0\\
0 & y\otimes x 
\end{pmatrix}
\text{ and }
u_\beta \mapsto 
\begin{pmatrix}
0 & 1\\
1 & 0
\end{pmatrix}
\end{align*}
and the Jones projection is given by
$$
e_1=
\begin{pmatrix}
f_1\otimes 1 & 0\\
0 & 0
\end{pmatrix}
$$
where $f_1$ is the Jones projection for $N_0\subset N_1$.
\end{prop}
\begin{proof}
By \cite{MR965748}, it suffices to show that the {\rm II}$_1$-factor $M_2$ given above is generated by $e_1$ and the image of $M_1$, and $e_1$ implements the conditional expectation $E_1\colon M_1\to M_0$, i.e., for every $x\in M_1$, $e_1xe_1=E_1(x)e_1$. 
To see the latter, for every $w,x,y,z\in N_1$,
$$
e_1
\begin{pmatrix}
w\otimes x & y\otimes z\\
z\otimes y & w\otimes x 
\end{pmatrix}
e_1
=
\begin{pmatrix}
E_{N_0}(w)\otimes x & 0\\
0 & 0 
\end{pmatrix}
=
\begin{pmatrix}
E_{N_0}(w)\otimes x & 0\\
0 & x\otimes E_{N_0}(w)
\end{pmatrix}
e_1.
$$
For the former, it is easy to get any element of the form 
$$
\begin{pmatrix}
x\otimes y & 0\\
0 & 0
\end{pmatrix}
$$
where $x\in N_2$ and $y\in N_1$ just by using the image of $N_1\otimes N_1$ in $M_2$ along with $e_1$. Now use the image of $u_\beta$ to move such elements around. The rest is straightforward.
\end{proof}

Since $M_0\subset M_1$ has an intermediate subfactor, $\jw{2}\cong \rho\oplus \rho\theta\rho$ by Proposition \ref{prop:DepthTwo}. We now identify these projections in the relative commutant.

\begin{cor}\label{cor:MinimalProjections}
$M_0'\cap M_2\cong \Complex^3$, where the minimal projections are given by
$$
e_1 =
\begin{pmatrix}
f_1\otimes 1 & 0\\
0 & 0
\end{pmatrix}
,\,
\rho=
\begin{pmatrix}
(1-f_1)\otimes 1 & 0\\
0 & 0
\end{pmatrix}
,\text{ and }
\rho\theta\rho=
\begin{pmatrix}
0 & 0\\
0 & 1\otimes 1
\end{pmatrix}.
$$
Note these have (non-normalized) traces $1,\tau,\tau^2$.
Hence $\rho\theta\ncong \theta\rho$.
\end{cor}
\begin{proof}
It is a straightforward calculation to show that
\begin{align*}
M_0'\cap M_2 
&\cong 
\left[(N_0\otimes N_1)'\cap N_2\otimes N_1\right] \oplus \left[(N_1\otimes N_0)'\cap N_2\otimes N_1\right]
\\& \cong 
\left[(N_0'\cap N_2)\otimes (N_1'\cap N_1)\right]\oplus\left[ (N_1'\cap N_2)\otimes (N_1'\cap N_1)\right]
\\& \cong
\spann\{f_1\otimes 1, (1-f_1)\otimes 1\} \oplus \spann\{1\otimes 1\}.
\end{align*}
The rest follows immediately.
\end{proof}

\begin{prop}
$(\rho\theta\rho)^2$ breaks up into 4 non-isomorphic bimodules, and $\rho\theta\rho\theta\rho$ is reducible.
\end{prop}
\begin{proof}
Recall that $M_0'\cap M_2=\End_{M_0-M_0}(L^2(M_1))$. 
In particular, $\rho\theta\rho \cong L^2(N_1\otimes N_1)_{\beta}$ as an $N_0\otimes N_1$ bimodule, i.e., 
$$
(x_0\otimes x_1)\cdot (\widehat{m}\otimes \widehat{n})\cdot (y_0\otimes y_1) = \widehat{x_0 m y_1}\otimes \widehat{x_1 n y_0}.
$$
By Corollary \ref{cor:MinimalProjections}, $\rho\theta\rho$ must be self-dual, and we compute that 
\begin{align*}
(\rho\theta\rho)^2
&\cong
L^2(N_1\otimes N_1)_{\beta} \bigotimes_{N_0\otimes N_1} {\sb{\beta}L}^2(N_1\otimes N_1) 
\\&\cong 
L^2(N_1\otimes N_1) \bigotimes_{N_1\otimes N_0} L^2(N_1\otimes N_1)
\\& \cong
L^2(N_1\otimes N_2)
\end{align*}
as an $N_0\otimes N_1$ bimodule. Now the sub-bimodules of $(\rho\theta\rho)^2$ correspond to the minimal projections in the relative commutant.
The relative commutant is given by
$$
(N_0\otimes N_1)'\cap (N_2\otimes N_3)\cong (N_0'\cap N_2)\otimes (N_1'\cap N_3)\cong \Complex^2\otimes \Complex^2,
$$
which has minimal projections
$$
f_1\otimes f_2,\, f_1\otimes (1-f_2),\, (1-f_1)\otimes f_2,\,\text{ and }(1-f_1)\otimes (1-f_2).
$$
Thus $(\rho\theta\rho)^2$ has 4 non-isomorphic summands. But we also see that
$$
(\rho\theta\rho)^2
=
\rho\theta\rho^2\theta\rho
= 
\rho\theta(1\oplus \rho)\theta\rho
=
\rho^2 \oplus \rho\theta\rho\theta\rho
=
1\oplus \rho\oplus \rho\theta\rho\theta\rho,
$$
so $\rho\theta\rho\theta\rho$ must be reducible.
\end{proof}

\begin{cor}
$(\rho\theta)^2\cong (\theta\rho)^2$.
\end{cor}
\begin{proof}
We have 
\begin{align*}
2&\geq \langle \rho\theta\rho\theta\rho, \rho\theta\rho\theta\rho\rangle \\
&= \langle \rho^2\theta\rho\theta,\theta\rho\theta\rho^2\rangle\\
&= \langle (1\oplus\rho)\theta\rho\theta,\theta\rho\theta(1\oplus\rho)\rangle\\
&= \langle \theta\rho\theta,\theta\rho\theta\rangle
+
\underbrace{\langle \theta\rho\theta,\theta\rho\theta\rho\rangle}_{0}
+
\underbrace{\langle \rho\theta\rho\theta,\theta\rho\theta\rangle}_{0}
+
\langle \rho\theta\rho\theta,\theta\rho\theta\rho\rangle 
\end{align*}
where we conclude the middle two are zero by Proposition \ref{prop:DistinctIrreducibleBimodules}.
\end{proof}

\begin{thm}
The principal graphs for $M_0\subset M_1$ are given by
$$
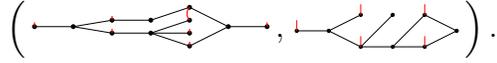
\left(
\bigraph{bwd1v1p1v1x0p0x1v1x0p1x0p1x0p0x1v1x0x0x1v1duals1v1x2v1x2x4x3v1},
\bigraph{bwd1v1p1v1x0p1x0v1x0p1x0v1x1duals1v1x2v1x2}
\right).
$$
Hence a subfactor with principal graphs $\fish_2$ exists, and $\AT_{2,1}$ exists.
\end{thm}
\begin{proof}
The principal graph is correct by Theorem \ref{thm:ATToFish}.
To calculate the dual graph, we use the {\texttt{FusionAtlas}} program {\texttt{FindGraphPartners}} to see that the only possibilities are
$$
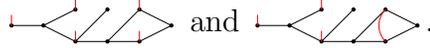
\bigraph{bwd1v1p1v1x0p1x0v1x0p1x0v1x1duals1v1x2v1x2}
\text{ and }
\bigraph{bwd1v1p1v1x0p1x0v1x0p1x0v1x1duals1v1x2v2x1}.
$$
We now identify the dual data. The dimensions of the even bimodules are 1, $2\tau$, 1, $\tau^2$, $\tau^2$ reading lexicographically left to right and bottom to top.
If $P,Q$ are the bimodules at the penultimate depth, then
$$
\dim(P\otimes \overline{P})=\tau^2\tau^2=(1+\tau)^2 = 1+2\tau+\tau^2.
$$
By a simple knapsacking argument, $P\otimes \overline{P}=1\oplus A\oplus B$ where $\dim(A)=2\tau$ and $\dim(B)=\tau^2$, and thus $B$ is either $P$ or $Q$. Since $P\otimes\overline{P}$ is self-dual, $\overline{B}=B$, and thus $P$ and $Q$ are self-dual. 
\end{proof}

We now construct $\fish_3$ as an intermediate subfactor of a reduced subfactor of the $3^{\Integer/4}$ subfactor with principal graphs 
$$
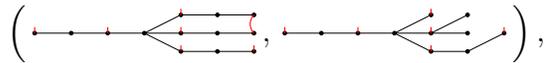
\left(
\bigraph{bwd1v1v1v1p1p1v1x0x0p0x1x0p0x0x1v1x0x0p0x1x0p0x0x1duals1v1v1x2x3v1x3x2}
,
\bigraph{bwd1v1v1v1p1p1v1x0x0p0x1x0p0x1x0v1x0x0duals1v1v1x2x3v1}
\right),
$$
which was constructed in unpublished work of Izumi and also in \cite{1308.5197}.

We will denote the even bimodules on the dual principal graph lexicographically left to right and bottom to top by
$1,\kappa,\beta\kappa,\chi,\sigma,\beta$.
Using the {\texttt{FusionAtlas}} program {\texttt{FindFusionRules}}, we see that the dual principal even half $\sb{M}{\sf{Mod}}_M$ of $3^{\Integer/4}$ has the following fusion rules

{\scriptsize{
$$
\begin{array}{|c|c|c|c|c|c|}
\hline 
\otimes & \kappa & \beta\kappa & \chi & \sigma & \beta \\
\hline \kappa & 1 {+} \kappa {+} \chi {+} \sigma {+} \beta\kappa & \kappa {+} \chi {+} \sigma {+} \beta {+} \beta\kappa & \kappa {+} 2 \chi {+} \sigma {+} \beta\kappa & \kappa {+} \chi {+} \beta\kappa & \beta\kappa \\\hline \beta\kappa & \kappa {+} \chi {+} \sigma {+} \beta {+} \beta\kappa & 1 {+} \kappa {+} \chi {+} \sigma {+} \beta\kappa & \kappa {+} 2 \chi {+} \sigma {+} \beta\kappa & \kappa {+} \chi {+} \beta\kappa & \kappa \\
\hline \chi & \kappa {+} 2 \chi {+} \sigma {+} \beta\kappa & \kappa {+} 2 \chi {+} \sigma {+} \beta\kappa & 1 {+} 2 \kappa {+} \chi {+} \sigma {+} \beta {+} 2 \beta\kappa & \kappa {+} \chi {+} \sigma {+} \beta\kappa & \chi \\
\hline \sigma & \kappa {+} \chi {+} \beta\kappa & \kappa {+} \chi {+} \beta\kappa & \kappa {+} \chi {+} \sigma {+} \beta\kappa & 1 {+} \chi {+} \sigma {+} \beta & \sigma \\
\hline \beta & \beta\kappa & \kappa & \chi & \sigma & 1 \\
\hline
\end{array}
$$
}}

The Frobenius-Perron dimensions of the $M-M$ bimodules are as follows:
\begin{align*}
\dim(1) &= \dim(\beta) = 1\\
\dim(\kappa) &= \dim(\beta\kappa) = 2 + \sqrt{5}\\
\dim(\chi) &= 2\tau^2=3+\sqrt{5}\\
\dim(\sigma) &= 2\tau = 1+\sqrt{5}.
\end{align*}
Using the {\texttt{FusionAtlas}} program {\texttt{ExtractPairOfBigraphsWithDuals}}, we compute the principal graphs of the reduced subfactor at $\sigma$  to be
$$
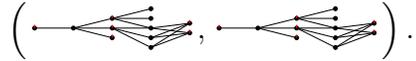
\left(
\bigraph{bwd1v1p1p1v0x1x0p0x1x0p0x1x1p0x0x1p0x0x1v1x1x1x0x0p1x1x1x0x0duals1v1x2x3v1x2},
\bigraph{bwd1v1p1p1v0x1x0p0x1x0p0x1x1p0x0x1p0x0x1v1x1x1x0x0p1x1x1x0x0duals1v1x2x3v1x2}
\right).
$$
Here, the reduced subfactor is $M\subset Q$, where $Q$ is the commutant of the right $M$-action on $\sigma$. 
By Lemma \ref{lem:Normalizer} applied to the dual graph, there is an intermediate subfactor $M\subset P\subset Q$ such that $[Q:P]=2$. 
By Goldman's Theorem \cite{MR0107827}, 
we have $Q\cong P\rtimes \Integer/2$, and $\sb{P}Q_P \cong 1_P \oplus \alpha$ for some $\alpha$ with dimension 1.

\begin{thm}
The principal graphs of $M\subset P$ are
$$
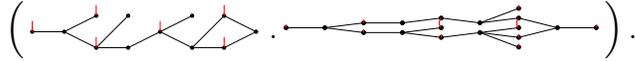
\left(
\bigraph{bwd1v1p1v1x0p1x0v1x0v1p1v1x0p1x0v1x1duals1v1x2v1v1x2}\,.
\bigraph{bwd1v1p1v1x0p0x1v1x0p1x0p0x1v1x0x0p0x0x1v1x0p1x0p1x0p0x1p0x1v0x1x0x1x0v1duals1v1x2v1x3x2v1x2x4x3x5v1}
\right).
$$
Hence a subfactor with principal graph $\fish_3$ exists, and $\AT_{3,1}$ exists.
\end{thm}
\begin{proof}
We factor $\sb{M}Q_Q \cong {\sb{M} P}\otimes_P Q_Q$, and for notational convenience we write $\xi = {\sb{M}P}_P$, so ${\sb{M}P}_M = \xi\overline{\xi}$.
Since $\sigma\overline{\sigma}\cong {\sb{M}Q}_M$, we have
\begin{align*}
1\oplus \beta \oplus \chi\oplus \sigma 
&= \sigma \overline{\sigma}
= {\sb{M}Q}_M 
= {\sb{M}Q}\otimes_Q Q_M \\
&= \xi( Q \otimes_Q Q)\overline{\xi}
= \xi (1_P\oplus \alpha)\overline{\xi}
= \xi\overline{\xi}\oplus  \xi \alpha \overline{\xi}.
\end{align*}
We also know $\xi\overline{\xi}$ has dimension $2\tau^2=3+\sqrt{5}$ and is not irreducible, since it contains a copy of the trivial. Hence by the Frobenius-Perron dimensions listed above, we must have $\xi\overline{\xi}=1\oplus \beta\oplus \sigma$. We immediately see that the even half of the subfactor $M\subset P$ is the even half of $3^{\Integer/4}$.

We continue computing the principal graph. We have
$$
\langle \sigma \xi, \sigma\xi\rangle = \langle \sigma^2, \xi \overline{\xi}\rangle = \langle 1\oplus \sigma\oplus \beta\oplus \chi, 1\oplus\sigma\oplus \beta\rangle = 3,
$$
so $\sigma\xi$ breaks up into 3 distinct irreducibles $\sigma\xi=\xi\oplus\nu\oplus\mu$. Moreover,
$$
\sigma\xi\overline{\xi} = \sigma(1\oplus \sigma\oplus \beta) = 1\oplus 3\sigma \oplus \beta \oplus \chi,
$$
so without loss of generality, $\nu$ is a univalent vertex, and $\mu$ connects to only $\sigma$ and $\chi$. Similarly as before, 
\begin{align*}
\langle \chi\xi,\chi\xi\rangle 
&= 3
\text{ and }\\
\chi(1\oplus \sigma\oplus\beta)
&=
\sigma \oplus 3\chi\oplus \kappa\oplus\beta\kappa,
\end{align*}
and since $\kappa,\beta\kappa$ are self-dual, the principal graph is 
$$
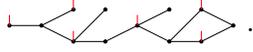
\bigraph{bwd1v1p1v1x0p1x0v1x0v1p1v1x0p1x0v1x1duals1v1x2v1v1x2}\,.
$$
Using the {\texttt{FusionAtlas}} program {\texttt{FindGraphPartners}}, the only possible dual graphs are
$$
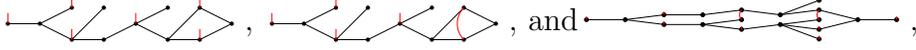
\bigraph{bwd1v1p1v1x0p1x0v1x0v1p1v1x0p1x0v1x1duals1v1x2v1v1x2}\,,\,
\bigraph{bwd1v1p1v1x0p1x0v1x0v1p1v1x0p1x0v1x1duals1v1x2v1v2x1}\,,\,\text{and}
\bigraph{bwd1v1p1v1x0p0x1v1x0p1x0p0x1v1x0x0p0x0x1v1x0p1x0p1x0p0x1p0x1v0x1x0x1x0v1duals1v1x2v1x3x2v1x2x4x3x5v1}\,,
$$
and Ocneanu's triple point obstruction \cite{MR1317352,MR2902285} implies that the third graph must be the dual graph.  
\end{proof}


\section{$T_2$ with $T_2$}

Our method also applies to composites of two copies of $T_2$ with little alteration.
Suppose we have two copies of $T_2$ generated by objects $\rho,\mu$, which we represent with red, orange strands respectively.
$$
\rho=
\begin{tikzpicture}[baseline = -.1cm]	
	\draw[thick, rho] (0,-.4) -- (0, .4);
	\draw[thick] (-.4,-.4)--(-.4,.4)--(.4,.4)--(.4,-.4)--(-.4,-.4);
\end{tikzpicture}
\text{ and }
\mu=
\begin{tikzpicture}[baseline = -.1cm]	
	\draw[thick, mu] (0,-.4) -- (0, .4);
	\draw[thick] (-.4,-.4)--(-.4,.4)--(.4,.4)--(.4,-.4)--(-.4,-.4);
\end{tikzpicture}\,.
$$
We also have intertwiners $\rho\otimes\rho\to \rho$ and $\mu\otimes\mu\to \mu$ which are represented by red and orange trivalent vertices respectively, both satisfying the relations in Proposition \ref{prop:RhoSkeinRelations}.

Again, we see nontrivial unitary quotients of $T_2*T_2$ are parametrized by an $n$ such that the alternating words in $\rho,\mu$ and $\mu,\rho$ of length $n$ are isomorphic, i.e.,
$$
\underbrace{(\rho \mu \cdots)}_{\text{length $n$}} \cong \underbrace{(\mu\rho \cdots)}_{\text{length $n$}},
$$ 
and an $n$-th root of unity $\omega_U$. In this case, there is a unitary isomorphism  $U\colon (\rho\mu)^n\to (\mu\rho)^n$ satisfying
\begin{align}
UU^* &= 
\begin{tikzpicture}[baseline = .5cm]
	\draw[thick, rho] (-.2,2)--(-.2,1.2);
	\draw[thick, mu] (0,2)--(0,1.2);
	\draw[thick, rho] (.2,2)--(.2,1.2);
	\draw[thick, mu] (-.2,1.2)--(-.2,0);
	\draw[thick, rho] (0,1.2)--(0,0);
	\draw[thick, mu] (.2,1.2)--(.2,0);
	\draw[thick, rho] (-.2,0)--(-.2,-.8);
	\draw[thick, mu] (0,0)--(0,-.8);
	\draw[thick, rho] (.2,0)--(.2,-.8);
	\draw[unshaded, thick] (0,0) circle (.4cm);
	\draw[unshaded, thick] (0,1.2) circle (.4cm);
	\draw[unshaded, thick] (0,0) circle (.4cm);
	\draw[unshaded, thick] (0,1.2) circle (.4cm);
	\node at (0,0) {$U$};
	\node at (0,1.2) {$U^*$};
	\node at (-.55,0) {$\star$};
	\node at (-.55,1.2) {$\star$};
\end{tikzpicture}
=
\begin{tikzpicture}[baseline = -.1cm]
	\draw[thick, rho] (-.2,-.8)--(-.2,.8);
	\draw[thick, mu] (0,-.8)--(0,.8);
	\draw[thick, rho] (.2,-.8)--(.2,.8);
\end{tikzpicture}
=
\id_{\rho\mu\rho}
\text{ and }
U^*U= 
\begin{tikzpicture}[baseline = .5cm]
	\draw[thick, mu] (-.2,2)--(-.2,1.2);
	\draw[thick, rho] (0,2)--(0,1.2);
	\draw[thick, mu] (.2,2)--(.2,1.2);
	\draw[thick, rho] (-.2,1.2)--(-.2,0);
	\draw[thick, mu] (0,1.2)--(0,0);
	\draw[thick, rho] (.2,1.2)--(.2,0);
	\draw[thick, mu] (-.2,0)--(-.2,-.8);
	\draw[thick, rho] (0,0)--(0,-.8);
	\draw[thick, mu] (.2,0)--(.2,-.8);
	\draw[unshaded, thick] (0,0) circle (.4cm);
	\draw[unshaded, thick] (0,1.2) circle (.4cm);
	\node at (0,0) {$U^*$};
	\node at (0,1.2) {$U$};
	\node at (-.55,0) {$\star$};
	\node at (-.55,1.2) {$\star$};
\end{tikzpicture}
=
\begin{tikzpicture}[baseline = -.1cm]
	\draw[thick, mu] (-.2,-.8)--(-.2,.8);
	\draw[thick, rho] (0,-.8)--(0,.8);
	\draw[thick, mu] (.2,-.8)--(.2,.8);
\end{tikzpicture}
=
\id_{\mu\rho\mu}
\text{ and}
\tag{TT1}\label{rel:TT1}
\\
\cF^{-1}(U)&=
\begin{tikzpicture}[baseline = -.1cm, yscale=-1]
	\draw[thick, mu] (-.2,-.3) arc (0:-180:.25cm) --(-.7,.8);
	\draw[thick, rho] (-.2,.8)--(-.2,0);
	\draw[thick, mu] (0,.8)--(0,0);
	\draw[thick, rho] (0,0)--(0,-.8);
	\draw[thick, mu] (.2,0)--(.2,-.8);
	\draw[thick, rho] (.2,.3) arc (180:0:.25cm) --(.7,-.8);
	\draw[unshaded, thick] (0,0) circle (.4cm);
	\node at (0,0) {$U$};
	\node at (-.55,0) {$\star$};
\end{tikzpicture}
=
\begin{tikzpicture}[baseline = -.1cm]
	\draw[thick, rho] (-.2,.8)--(-.2,0);
	\draw[thick, mu] (0,.8)--(0,0);
	\draw[thick, rho] (.2,.8)--(.2,0);
	\draw[thick, mu] (-.2,0)--(-.2,-.8);
	\draw[thick, rho] (0,0)--(0,-.8);
	\draw[thick, mu] (.2,0)--(.2,-.8);
	\draw[unshaded, thick] (0,0) circle (.4cm);
	\node at (0,0) {$U^*$};
	\node at (-.55,0) {$\star$};
\end{tikzpicture}
=
\omega_U^{-1}
\begin{tikzpicture}[baseline = -.1cm]
	\draw[thick, rho] (-.2,-.3) arc (0:-180:.25cm) --(-.7,.8);
	\draw[thick, mu] (-.2,.8)--(-.2,0);
	\draw[thick, rho] (0,.8)--(0,0);
	\draw[thick, mu] (0,0)--(0,-.8);
	\draw[thick, rho] (.2,0)--(.2,-.8);
	\draw[thick, mu] (.2,.3) arc (180:0:.25cm) --(.7,-.8);
	\draw[unshaded, thick] (0,0) circle (.4cm);
	\node at (0,0) {$U$};
	\node at (-.55,0) {$\star$};
\end{tikzpicture}
=
\omega_U^{-1}\cF(U)
\tag{TT2}\label{rel:TT2}
\end{align}
(the diagrams above are for the case $n=3$). Again, if we weren't looking at just unitary quotients, we would find a second option for the $*$-structure as in Remarks \ref{rem:A2A2StarStructure} and \ref{rem:A2T2StarStructure}.

When $n$ is even, Lemma \ref{lem:RhoSpaghettiJellyfish} follows verbatim with $\zeta=(\mu\rho)^{n/2-1}\mu$, and we also get
$$
\begin{tikzpicture}[baseline = -.1cm]
	\draw[thick, A3] (-.4,.2) -- (.4,.2);
	\filldraw[mu] (0,0) circle (.05cm);	
	\draw[thick, mu] (0,0) -- (-.4,-.4);
	\draw[thick, mu] (0,0) -- (0,-.4);
	\draw[thick, mu] (0,0) -- (.4,-.4);
	\draw[thick] (-.4,-.4)--(-.4,.4)--(.4,.4)--(.4,-.4)--(-.4,-.4);
\end{tikzpicture}
=
\begin{tikzpicture}[baseline = -.1cm]
	\draw[thick, A3] (-1.8,0) -- (1.8,0);	
	\coordinate (a) at (0,.9);
	\filldraw[mu] (a) circle (.05cm);
	\draw[thick, mu] (a) -- (-1.8,-.6);
	\draw[thick, mu] (a) -- (0,-.6);
	\draw[thick, mu] (a) -- (1.8,-.6);
	\draw[thick, unshaded] (-1,0) circle (.4cm);
	\node at (-1,0) {$U$};
	\node at (-1,.55) {$\star$};
	\draw[thick, unshaded]  (0,0) circle (.4cm);
	\node at (0,0) {$U$};
	\node at (115:.55cm) {$\star$};
	\draw[thick, unshaded]  (1,0) circle (.4cm);
	\node at (1,0) {$U^*$};
	\node at (1,.55) {$\star$};
\end{tikzpicture}\,,
$$
where we recycle the use of the green strand to denote $(\rho\mu)^{(n/2)-1}\rho$.
This yields identical jellyfish relations to (2) and (3) in Theorem \ref{thm:JellyfishRelationsU} and Corollary \ref{cor:JellyfishRelationsUstar}, except we now replace (1) by the following two relations and their adjoints:
\begin{align*}
\begin{tikzpicture}[baseline = -.3cm]
	\draw[thick, mu] (-.7,-.8)-- (-.7,0) arc(180:0:.7cm) -- (.7,-.8);
	\coordinate (a) at (0,.7);
	\filldraw[mu] (a) circle (.05cm);
	\draw[thick, mu] (a)--(0,0);
	\draw[thick, A3] (-.2,0) -- (-.2,-.8);
	\draw[thick, mu] (0,0) -- (0,-.8);
	\draw[thick, A3] (.2,0) -- (.2,-.8);
	\draw[thick, unshaded]  (0,0) circle (.4cm);
	\node at (0,0) {$U^*$};
	\node at (45:.55cm) {$\star$};
\end{tikzpicture}
&=
\begin{tikzpicture}[baseline = -.3cm]
	\draw[thick, mu] (-.7,-.8)-- (-.7,0) arc(180:0:.7cm) -- (.7,-.8);
	\coordinate (a) at (0,.7);
	\filldraw[mu] (a) circle (.05cm);
	\draw[thick, mu] (a)--(0,0);
	\draw[thick, A3] (-.2,0) -- (-.2,-.8);
	\draw[thick, mu] (0,0) -- (0,-.8);
	\draw[thick, A3] (.2,0) -- (.2,-.8);
	\draw[thick, unshaded]  (0,0) circle (.4cm);
	\node at (0,0) {$U$};
	\node at (135:.55cm) {$\star$};
\end{tikzpicture}
=
\begin{tikzpicture}[baseline = -.3cm]
	\draw[thick, A3] (-.6,0) -- (.6,0);
	\coordinate (a) at (0,-.5);	
	\filldraw[mu] (a) circle (.05cm);
	\draw[thick, mu] (-.6,0)--(a)--(.6,0);
	\draw[thick, mu] (a)--(0,-.8);
	\draw[thick, mu] (-.7,0) -- (-.7,-.8);
	\draw[thick, A3] (-.5,0) -- (-.5,-.8);
	\draw[thick, A3] (.5,0) -- (.5,-.8);
	\draw[thick, mu] (.7,0) -- (.7,-.8);
	\draw[thick, unshaded]  (-.6,0) circle (.4cm);
	\node at (-.6,0) {$U^*$};
	\node at (-.6,.55) {$\star$};
	\draw[thick, unshaded]  (.6,0) circle (.4cm);
	\node at (.6,0) {$U$};
	\node at (.6,.55) {$\star$};
\end{tikzpicture}
\\
\displaystyle
\begin{tikzpicture}[baseline = -.3cm]
	\draw[thick, mu] (-.7,-.8)-- (-.7,0) arc(180:0:.7cm) -- (.7,-.8);
	\draw[thick, mu] (-.3,0) -- (-.3,-.8);
	\draw[thick, A3] (-.1,0) -- (-.1,-.8);
	\draw[thick, mu] (.1,0) -- (.1,-.8);
	\draw[thick, A3] (.3,0) -- (.3,-.8);
	\draw[thick, unshaded]  (0,0) circle (.4cm);
	\node at (0,0) {$U^*$};
	\node at (-.55,0) {$\star$};
\end{tikzpicture}
&=
\frac{1}{\tau}\,
\begin{tikzpicture}[baseline = -.3cm]
	\draw[thick, A3] (-.3,0) -- (-.3,-.8);
	\draw[thick, mu] (-.1,0) -- (-.1,-.8);
	\draw[thick, A3] (.1,0) -- (.1,-.8);
	\draw[thick, mu] (.3,0) -- (.3,-.8);
	\draw[thick, mu] (-.9,-.8) arc (180:0:.2cm);
	\draw[thick, unshaded]  (0,0) circle (.4cm);
	\node at (0,0) {$U$};
	\node at (-.55,0) {$\star$};
\end{tikzpicture}
+
\begin{tikzpicture}[baseline = -.3cm]
	\draw[thick, A3] (-.6,0) -- (.6,0);
	\coordinate (a) at (0,-.5);	
	\coordinate (b) at (-1,-.5);	
	\filldraw[mu] (a) circle (.05cm);
	\filldraw[mu] (b) circle (.05cm);
	\draw[thick, mu] (-.6,0)--(a)--(.6,0);
	\draw[thick, mu] (a)--(0,-.8);
	\draw[thick, mu] (-1.2,-.8)--(b)--(-.8,-.8);
	\draw[thick, mu] (b)--(-.6,0);
	\draw[thick, A3] (-.6,0) -- (-.6,-.8);
	\draw[thick, A3] (.5,0) -- (.5,-.8);
	\draw[thick, mu] (.7,0) -- (.7,-.8);
	\draw[thick, unshaded]  (-.6,0) circle (.4cm);
	\node at (-.6,0) {$U^*$};
	\node at (-.6,.55) {$\star$};
	\draw[thick, unshaded]  (.6,0) circle (.4cm);
	\node at (.6,0) {$U$};
	\node at (.6,.55) {$\star$};
\end{tikzpicture}\,.
\end{align*}

When $n$ is odd, we need a slight alteration. In this case, we write an oriented strand
$$
\begin{tikzpicture}[baseline = -.1cm]
	\draw[thick, zeta] (0,-.4) -- (0,.4);
	\draw[thick, zeta] (-.1,-.05) -- (0,.05) -- (.1,-.05); 
	\draw[thick] (-.4,-.4)--(-.4,.4)--(.4,.4)--(.4,-.4)--(-.4,-.4);
\end{tikzpicture} 
= \zeta = (\rho\mu)^{\frac{n-1}{2}}.
$$
Using an argument similar to Lemma \ref{lem:RhoSpaghettiJellyfish}, we get
\begin{align*}
\begin{tikzpicture}[baseline = -.1cm]
	\draw[thick, zeta] (-.4,.2) -- (.4,.2);
	\draw[thick, zeta] (-.05,.3)-- (.05,.2) -- (-.05,.1);	
	\filldraw[rho] (0,0) circle (.05cm);	
	\draw[thick, rho] (0,0) -- (-.4,-.4);
	\draw[thick, rho] (0,0) -- (0,-.4);
	\draw[thick, rho] (0,0) -- (.4,-.4);
	\draw[thick] (-.4,-.4)--(-.4,.4)--(.4,.4)--(.4,-.4)--(-.4,-.4);
\end{tikzpicture}
&=
\begin{tikzpicture}[baseline = -.1cm]
	\coordinate (a) at (0,.9);
	\coordinate (b) at (-1,0);
	\coordinate (c) at (0,0);
	\coordinate (d) at (1,0);
	\draw[thick, zeta] (-1.8,0) -- (1.8,0);
	\draw[thick, zeta] (-1.6,.1)-- (-1.5,0) -- (-1.6,-.1);	
	\draw[thick, zeta] (1.5,.1)-- (1.6,0) -- (1.5,-.1);	
	\filldraw[mu] (a) circle (.05cm);
	\draw[thick, mu] (a) -- (b);
	\draw[thick, mu] (a) -- (c);
	\draw[thick, mu] (a) -- (d);	 
	\draw[thick, rho] (b) -- (-1.8,-.6);
	\draw[thick, rho] (c) -- (0,-.6);
	\draw[thick, rho] (d) -- (1.8,-.6);
	\draw[thick, unshaded] (b) circle (.4cm);
	\node at (-1,0) {$U$};
	\node at (-1,.55) {$\star$};
	\draw[thick, unshaded]  (c) circle (.4cm);
	\node at (0,0) {$U$};
	\node at (115:.55cm) {$\star$};
	\draw[thick, unshaded]  (d) circle (.4cm);
	\node at (1,0) {$U^*$};
	\node at (1,.55) {$\star$};
\end{tikzpicture}
\text{ and}
\\
\begin{tikzpicture}[baseline = -.1cm]
	\draw[thick, zeta] (-.4,.2) -- (.4,.2);
	\draw[thick, zeta] (.05,.3)-- (-.05,.2) -- (.05,.1);	
	\filldraw[mu] (0,0) circle (.05cm);	
	\draw[thick, mu] (0,0) -- (-.4,-.4);
	\draw[thick, mu] (0,0) -- (0,-.4);
	\draw[thick, mu] (0,0) -- (.4,-.4);
	\draw[thick] (-.4,-.4)--(-.4,.4)--(.4,.4)--(.4,-.4)--(-.4,-.4);
\end{tikzpicture}
&=
\sigma_U
\,
\begin{tikzpicture}[baseline = -.1cm]
	\coordinate (a) at (0,.9);
	\coordinate (b) at (-1,0);
	\coordinate (c) at (0,0);
	\coordinate (d) at (1,0);
	\draw[thick, zeta] (-1.8,0) -- (1.8,0);
	\draw[thick, zeta] (1.6,.1)-- (1.5,0) -- (1.6,-.1);	
	\draw[thick, zeta] (-1.5,.1)-- (-1.6,0) -- (-1.5,-.1);	
	\filldraw[rho] (a) circle (.05cm);
	\draw[thick, rho] (a) -- (b);
	\draw[thick, rho] (a) -- (c);
	\draw[thick, rho] (a) -- (d);	 
	\draw[thick, mu] (b) -- (-1.8,-.6);
	\draw[thick, mu] (c) -- (0,-.6);
	\draw[thick, mu] (d) -- (1.8,-.6);
	\draw[thick, unshaded] (b) circle (.4cm);
	\node at (-1,0) {$U^*$};
	\node at (-1,.55) {$\star$};
	\draw[thick, unshaded]  (c) circle (.4cm);
	\node at (0,0) {$U^*$};
	\node at (115:.55cm) {$\star$};
	\draw[thick, unshaded]  (d) circle (.4cm);
	\node at (1,0) {$U$};
	\node at (1,.55) {$\star$};
\end{tikzpicture}\,.
\end{align*}
We now use the skein theory for $\rho,\mu$ strands to get the following jellyfish relations for $U,U^*$:
\begin{align*}
\begin{tikzpicture}[baseline = -.3cm]
	\draw[thick, rho] (-.7,-.8)-- (-.7,0) arc(180:0:.7cm) -- (.7,-.8);
	\coordinate (a) at (0,.7);
	\filldraw[rho] (a) circle (.05cm);
	\draw[thick, rho] (a)--(0,0);
	\draw[thick, zeta] (-.2,0) -- (-.2,-.8);
	\zetaDownArrow{(-.2,-.6)}	
	\draw[thick, mu] (0,0) -- (0,-.8);
	\draw[thick, zeta] (.2,0) -- (.2,-.8);
	\zetaUpArrow{(.2,-.6)}	
	\draw[thick, unshaded]  (0,0) circle (.4cm);
	\node at (0,0) {$U$};
	\node at (45:.55cm) {$\star$};
\end{tikzpicture}
&=
\omega_U
\begin{tikzpicture}[baseline = -.3cm]
	\draw[thick, rho] (-.7,-.8)-- (-.7,0) arc(180:0:.7cm) -- (.7,-.8);
	\coordinate (a) at (0,.7);
	\filldraw[rho] (a) circle (.05cm);
	\draw[thick, rho] (a)--(0,0);
	\draw[thick, zeta] (-.2,0) -- (-.2,-.8);
	\zetaDownArrow{(-.2,-.6)}			
	\draw[thick, mu] (0,0) -- (0,-.8);
	\draw[thick, zeta] (.2,0) -- (.2,-.8);
	\zetaUpArrow{(.2,-.6)}	
	\draw[thick, unshaded]  (0,0) circle (.4cm);
	\node at (0,0) {$U^*$};
	\node at (135:.55cm) {$\star$};
\end{tikzpicture}
=
\sigma_U^{-1}
\begin{tikzpicture}[baseline = -.3cm]
	\draw[thick, zeta] (-.6,0) -- (.6,0);
	\draw[thick, zeta] (.05,.1)-- (-.05,0) -- (.05,-.1);
	\coordinate (a) at (0,-.5);	
	\filldraw[mu] (a) circle (.05cm);
	\draw[thick, mu] (-.6,0)--(a)--(.6,0);
	\draw[thick, mu] (a)--(0,-.8);
	\draw[thick, rho] (-.7,0) -- (-.7,-.8);
	\draw[thick, zeta] (-.5,0) -- (-.5,-.8);
	\zetaDownArrow{(-.5,-.6)}		
	\draw[thick, zeta] (.5,0) -- (.5,-.8);
	\zetaUpArrow{(.5,-.6)}		
	\draw[thick, rho] (.7,0) -- (.7,-.8);
	\draw[thick, unshaded]  (-.6,0) circle (.4cm);
	\node at (-.6,0) {$U^*$};
	\node at (-.6,.55) {$\star$};
	\draw[thick, unshaded]  (.6,0) circle (.4cm);
	\node at (.6,0) {$U$};
	\node at (.6,.55) {$\star$};
\end{tikzpicture}
\\
\displaystyle
\begin{tikzpicture}[baseline = -.3cm]
	\draw[thick, rho] (-.7,-.8)-- (-.7,0) arc(180:0:.7cm) -- (.7,-.8);
	\draw[thick, rho] (-.3,0) -- (-.3,-.8);
	\draw[thick, zeta] (-.1,0) -- (-.1,-.8);
	\zetaDownArrow{(-.1,-.6)}
	\draw[thick, rho] (.1,0) -- (.1,-.8);
	\draw[thick, zeta] (.3,0) -- (.3,-.8);
	\zetaDownArrow{(.3,-.6)}
	\draw[thick, unshaded]  (0,0) circle (.4cm);
	\node at (0,0) {$U$};
	\node at (-.55,0) {$\star$};
\end{tikzpicture}
&=
\frac{\omega_U}{\tau}\,
\begin{tikzpicture}[baseline = -.3cm]
	\draw[thick, zeta] (-.3,0) -- (-.3,-.8);
	\zetaDownArrow{(-.3,-.6)}
	\draw[thick, rho] (-.1,0) -- (-.1,-.8);
	\draw[thick, zeta] (.1,0) -- (.1,-.8);
	\zetaDownArrow{(.1,-.6)}
	\draw[thick, rho] (.3,0) -- (.3,-.8);
	\draw[thick, rho] (-.9,-.8) arc (180:0:.2cm);
	\draw[thick, unshaded]  (0,0) circle (.4cm);
	\node at (0,0) {$U^*$};
	\node at (-.55,0) {$\star$};
\end{tikzpicture}
+
\sigma_U^{-1}
\begin{tikzpicture}[baseline = -.3cm]
	\draw[thick, zeta] (-.6,0) -- (.6,0);
	\draw[thick, zeta] (.05,.1)-- (-.05,0) -- (.05,-.1);
	\coordinate (a) at (0,-.5);	
	\coordinate (b) at (-1,-.5);	
	\filldraw[mu] (a) circle (.05cm);
	\filldraw[rho] (b) circle (.05cm);
	\draw[thick, mu] (-.6,0)--(a)--(.6,0);
	\draw[thick, mu] (a)--(0,-.8);
	\draw[thick, rho] (-1.2,-.8)--(b)--(-.8,-.8);
	\draw[thick, rho] (b)--(-.6,0);
	\draw[thick, zeta] (-.6,0) -- (-.6,-.8);
	\zetaDownArrow{(-.6,-.6)}
	\draw[thick, zeta] (.5,0) -- (.5,-.8);
	\zetaUpArrow{(.5,-.6)}
	\draw[thick, rho] (.7,0) -- (.7,-.8);
	\draw[thick, unshaded]  (-.6,0) circle (.4cm);
	\node at (-.6,0) {$U$};
	\node at (-.6,.55) {$\star$};
	\draw[thick, unshaded]  (.6,0) circle (.4cm);
	\node at (.6,0) {$U^*$};
	\node at (.6,.55) {$\star$};
\end{tikzpicture}
\\
\begin{tikzpicture}[baseline = -.3cm]
	\draw[thick, mu] (-.7,-.8)-- (-.7,0) arc(180:0:.7cm) -- (.7,-.8);
	\coordinate (a) at (0,.7);
	\filldraw[mu] (a) circle (.05cm);
	\draw[thick, mu] (a)--(0,0);
	\draw[thick, zeta] (-.2,0) -- (-.2,-.8);
	\zetaUpArrow{(-.2,-.6)}	
	\draw[thick, rho] (0,0) -- (0,-.8);
	\draw[thick, zeta] (.2,0) -- (.2,-.8);
	\zetaDownArrow{(.2,-.6)}
	\draw[thick, unshaded]  (0,0) circle (.4cm);
	\node at (0,0) {$U^*$};
	\node at (45:.55cm) {$\star$};
\end{tikzpicture}
&=
\begin{tikzpicture}[baseline = -.3cm]
	\draw[thick, mu] (-.7,-.8)-- (-.7,0) arc(180:0:.7cm) -- (.7,-.8);
	\coordinate (a) at (0,.7);
	\filldraw[mu] (a) circle (.05cm);
	\draw[thick, mu] (a)--(0,0);
	\draw[thick, zeta] (-.2,0) -- (-.2,-.8);
	\zetaUpArrow{(-.2,-.6)}	
	\draw[thick, rho] (0,0) -- (0,-.8);
	\draw[thick, zeta] (.2,0) -- (.2,-.8);
	\zetaDownArrow{(.2,-.6)}
	\draw[thick, unshaded]  (0,0) circle (.4cm);
	\node at (0,0) {$U$};
	\node at (135:.55cm) {$\star$};
\end{tikzpicture}
=
\begin{tikzpicture}[baseline = -.3cm]
	\draw[thick, zeta] (-.6,0) -- (.6,0);
	\draw[thick, zeta] (-.05,.1)-- (.05,0) -- (-.05,-.1);
	\coordinate (a) at (0,-.5);	
	\filldraw[rho] (a) circle (.05cm);
	\draw[thick, rho] (-.6,0)--(a)--(.6,0);
	\draw[thick, rho] (a)--(0,-.8);
	\draw[thick, mu] (-.7,0) -- (-.7,-.8);
	\draw[thick, zeta] (-.5,0) -- (-.5,-.8);
	\zetaUpArrow{(-.5,-.6)}
	\draw[thick, zeta] (.5,0) -- (.5,-.8);
	\zetaDownArrow{(.5,-.6)}
	\draw[thick, mu] (.7,0) -- (.7,-.8);
	\draw[thick, unshaded]  (-.6,0) circle (.4cm);
	\node at (-.6,0) {$U^*$};
	\node at (-.6,.55) {$\star$};
	\draw[thick, unshaded]  (.6,0) circle (.4cm);
	\node at (.6,0) {$U$};
	\node at (.6,.55) {$\star$};
\end{tikzpicture}
\\
\displaystyle
\begin{tikzpicture}[baseline = -.3cm]
	\draw[thick, mu] (-.7,-.8)-- (-.7,0) arc(180:0:.7cm) -- (.7,-.8);
	\draw[thick, mu] (-.3,0) -- (-.3,-.8);
	\draw[thick, zeta] (-.1,0) -- (-.1,-.8);
	\zetaUpArrow{(-.1,-.6)}
	\draw[thick, mu] (.1,0) -- (.1,-.8);
	\draw[thick, zeta] (.3,0) -- (.3,-.8);
	\zetaUpArrow{(.3,-.6)}
	\draw[thick, unshaded]  (0,0) circle (.4cm);
	\node at (0,0) {$U^*$};
	\node at (-.55,0) {$\star$};
\end{tikzpicture}
&=
\frac{1}{\tau}\,
\begin{tikzpicture}[baseline = -.3cm]
	\draw[thick, zeta] (-.3,0) -- (-.3,-.8);
	\zetaUpArrow{(-.3,-.6)}
	\draw[thick, mu] (-.1,0) -- (-.1,-.8);
	\draw[thick, zeta] (.1,0) -- (.1,-.8);
	\zetaUpArrow{(.1,-.6)}
	\draw[thick, mu] (.3,0) -- (.3,-.8);
	\draw[thick, mu] (-.9,-.8) arc (180:0:.2cm);
	\draw[thick, unshaded]  (0,0) circle (.4cm);
	\node at (0,0) {$U$};
	\node at (-.55,0) {$\star$};
\end{tikzpicture}
+
\begin{tikzpicture}[baseline = -.3cm]
	\draw[thick, zeta] (-.6,0) -- (.6,0);
	\draw[thick, zeta] (-.05,.1)-- (.05,0) -- (-.05,-.1);
	\coordinate (a) at (0,-.5);	
	\coordinate (b) at (-1,-.5);	
	\filldraw[rho] (a) circle (.05cm);
	\filldraw[mu] (b) circle (.05cm);
	\draw[thick, rho] (-.6,0)--(a)--(.6,0);
	\draw[thick, rho] (a)--(0,-.8);
	\draw[thick, mu] (-1.2,-.8)--(b)--(-.8,-.8);
	\draw[thick, mu] (b)--(-.6,0);
	\draw[thick, zeta] (-.6,0) -- (-.6,-.8);
	\zetaUpArrow{(-.6,-.6)}
	\draw[thick, zeta] (.5,0) -- (.5,-.8);
	\zetaDownArrow{(.5,-.6)}
	\draw[thick, mu] (.7,0) -- (.7,-.8);
	\draw[thick, unshaded]  (-.6,0) circle (.4cm);
	\node at (-.6,0) {$U^*$};
	\node at (-.6,.55) {$\star$};
	\draw[thick, unshaded]  (.6,0) circle (.4cm);
	\node at (.6,0) {$U$};
	\node at (.6,.55) {$\star$};
\end{tikzpicture}\,.
\end{align*}

Arguing as in Theorem \ref{thm:JellyfishEvaluation}, the relations in Proposition \ref{prop:RhoSkeinRelations} for $\rho$ and $\mu$ strings, Relations \eqref{rel:TT1}-\eqref{rel:TT2}, and the jellyfish relations are sufficient to evaluate all closed diagrams, and thus we make the following definition.

\begin{defn}
For $2\leq n<\infty$, let $\TT_{n,\omega_U}$ be the unitary quotient of $T_2*T_2$ generated by $U$ satisfying the relations of Proposition \ref{prop:RhoSkeinRelations} for $\rho$ and $\mu$ strings and Relations \eqref{rel:TT1}-\eqref{rel:TT2},  provided that it exists. Note that $\TT_{2,1}$ is $T_2\boxtimes T_2$.
\end{defn}

We get a similar basis as before for our jellyfish calculations, and we can use a similar twocar formalism. 
Again, for our basis elements, we only specify the number of strings connecting the $U,U^*$'s along the top, and the type of strings and vertices underneath is determined.
For $\kappa=\rho,\mu$, regardless of parity, we get the same action on $\twocar$'s (provided that the topmost string in the middle bundle is also a $\kappa$ string):
\begin{align*}
\kappa(\twocar[a&,b,c])=
\frac{1}{\tau}\twocar[a+1,b-1,c+1] \\
&+ 
\begin{cases}
\twocar[a + 1, n - 1, b]\twocar[b, n - 1, c + 1] & \text{if $a,c<n-1$}\\
\twocar[a + 1, n - 1, b]\twocar[b, n - 1, n] & \text{if $a<n-1\leq c$}\\
\twocar[n, n - 1, b]\twocar[b, n - 1, c + 1] & \text{if $c<n-1\leq a$}\\
\twocar[n, n - 1, b]\twocar[b, n - 1, n] & \text{if $n-1\leq a,c$,}
\end{cases}
\end{align*}
and we have the following concatenation relations:
\begin{align*}
\twocar[x_1,x_2&,n]\twocar[n,x_3,x_4]\\
&=
\begin{cases}
0 & \text{ if $x_2,x_3\neq 2n-1$}\\
\twocar[x_1,x_3,x_4] &\text{ if $x_2=n-1$ and $x_3< n-1$}\\
\twocar[x_1,x_2,x_4] &\text{ if $x_2< n-1$ and $x_3=n-1$}\\
-\tau^{-1}\twocar[x_1,n-1,x_4] &\text{ if $x_2,x_3=n-1$.}
\end{cases}
\end{align*}

Our uniqueness and non-existence proofs are also similar to before. As in Equation \eqref{eqn:Consistency}, we evaluate $(\rho\mu)^{\lfloor\frac{n}{2}\rfloor}\rho^{\epsilon}(U)$ in two different ways, where $\epsilon=0$ if $n$ is even (so $\rho^0=1$) and $\epsilon=1$ if $n$ is odd.
In Figures \ref{fig:TTn=2} and \ref{fig:TTn=3} in Appendix \ref{sec:FiguresForTT}, we give for $n=2$ and $n=3$ respectively the tables of coefficients in the linear combinations for $(\rho\mu)^{\lfloor\frac{n}{2}\rfloor}\rho^{\epsilon}(U)$ and $U((\mu\rho)^{\lfloor\frac{n}{2}\rfloor}\mu^{\epsilon}(U))U^*$ after applying the jellyfish relations. 

We also see that $n=4$ is not possible, since the linear combinations of trains given in Figure \ref{fig:TTn=4} can never be equal in the train basis.
We also prove the non-existence of $\TT_{n,\omega_U}$ for $5\leq n \leq 10$ using the same code. 

The above discussion proves the following theorems.

\begin{thm}\label{thm:A4A4Unique}
For $n=2,3$, $\TT_{n,\omega_U}$ exists only if $\omega_U=1$.
\end{thm}

\begin{thm}\label{thm:A4A4Nonexistence}
For $4\leq n\leq 10$, $\TT_{n,\omega_U}$ does not exist.
\end{thm}

\begin{conj}
The technique used for Theorems \ref{thm:A4A4Unique} and \ref{thm:A4A4Nonexistence} should show
\begin{enumerate}[(1)]
\item
$\TT_{n,\omega_U}$ exists only if $\omega_U=1$ for all $2\leq n<\infty$, and
\item
$\TT_{n,\omega_U}$ does not exist for all $4\leq n<\infty$.
\end{enumerate} 
\end{conj}

\subsection{Application to subfactors}\label{sec:A4A4Subfactors}

The techniques of Section \ref{sec:Subfactors} can be used to prove the following theorem.

\begin{thm}\label{thm:SubfactorIffTT}
Any $A_4-A_4$ composite subfactor has (dual) even half $T_2*T_2$ or $\TT_{n,\omega_U}$ for some $n$-th root of unity $\omega_U$.

Conversely, in $T_2*T_2$, $1\oplus \rho\oplus \rho\mu\rho$ is a Frobenius algebra object with subalgebra $1\oplus\rho$. Thus if $\TT_{n,\omega_U}$ exists, then there is an $A_4-A_4$ composite subfactor with (dual) even half $\TT_{n,\omega_U}$.
\end{thm}
\begin{proof}
If we have a composite subfactor $N\subset P\subset M$, where $N\subset P$ and $P\subset M$ are $A_4$ subfactors, we can define $\rho,\mu$ analogously to Definition \ref{defn:ThetaRho}.
Thus $\rho,\mu\ncong 1$ are irreducible, satisfying $\rho^2\cong 1\oplus \rho$ and $\mu^2\cong 1\oplus \mu$ by the same proof as in Lemma \ref{lem:Nontrivial}. 
Again, we have $\jw{2}\cong \rho\oplus \rho\mu\rho$, so the $M-M$ bimodules are generated by $\rho,\mu$, and the even half is either $T_2*T_2$ or $\TT_{n,\omega_U}$ for some $\omega_U$.

For the converse, the algebra map is given by
\begin{align*}
\text{maps from $X\otimes Y \to 1$: }&
\begin{array}{|c|c|c|c|}
\hline
\otimes & 1 & \rho & \rho\mu\rho
\\\hline
1 & 
\begin{tikzpicture}[baseline = -.1cm]
	\clip (-.45,-.45) rectangle (.45,.45);
	\nbox{thick}{(0,0)}{0}{0}{}
\end{tikzpicture}
& 0 & 0
\\\hline
\rho & 0 & 
\begin{tikzpicture}[baseline = -.1cm]
	\clip (-.45,-.45) rectangle (.45,.45);
	\nbox{thick}{(0,0)}{0}{0}{}
	\draw[thick, rho] (-.25,-.4) arc (180:0:.25cm);
\end{tikzpicture}
& 0
\\\hline
\rho\mu\rho &  
0 & 0 &
\begin{tikzpicture}[baseline = -.1cm]
	\clip (-.45,-.45) rectangle (.45,.45);
	\nbox{thick}{(0,0)}{0}{0}{}
	\draw[thick, rho] (-.3,-.4) arc (180:0:.3cm);
	\draw[thick, mu] (-.2,-.4) arc (180:0:.2cm);
	\draw[thick, rho] (-.1,-.4) arc (180:0:.1cm);
\end{tikzpicture}
\\\hline
\end{array}
\displaybreak[1]\\
\text{maps from $X\otimes Y \to \rho$: }&
\begin{array}{|c|c|c|c|}
\hline
\otimes & 1 & \rho & \rho\mu\rho
\\\hline
1 & 0 & 
\begin{tikzpicture}[baseline = -.1cm]
	\clip (-.45,-.45) rectangle (.45,.45);
	\nbox{thick}{(0,0)}{0}{0}{}
	\draw[thick, rho] (.2,-.4) -- (0,.4);
\end{tikzpicture}
& 0
\\\hline
\rho &  
\begin{tikzpicture}[baseline = -.1cm]
	\clip (-.45,-.45) rectangle (.45,.45);
	\nbox{thick}{(0,0)}{0}{0}{}
	\draw[thick, rho] (-.2,-.4) -- (0,.4);
\end{tikzpicture}
& 
\frac{\pm1}{\sqrt{\tau}}\,
\begin{tikzpicture}[baseline = -.1cm]
	\clip (-.45,-.45) rectangle (.45,.45);
	\filldraw[rho] (0,0) circle (.05cm);
	\draw[thick, rho] (0,0) -- (-.2,-.4);
	\draw[thick, rho] (0,0) -- (0,.4);
	\draw[thick, rho] (0,0) -- (.2,-.4);
	\nbox{}{(0,0)}{0}{0}{}
\end{tikzpicture}
& 0
\\\hline
\rho\mu\rho &  
0 & 0 & 
\mp\sqrt{\tau}\,
\begin{tikzpicture}[baseline = -.1cm]
	\clip (-.45,-.45) rectangle (.45,.45);
	\filldraw[rho] (0,0) circle (.05cm);
	\draw[thick, rho] (0,0) -- (-.3,-.4);
	\draw[thick, rho] (0,0) -- (0,.4);
	\draw[thick, rho] (0,0) -- (.3,-.4);
	\draw[thick, mu] (-.175,-.4) arc (180:0:.175cm);
	\draw[thick, rho] (-.1,-.4) arc (180:0:.1cm);
	\nbox{}{(0,0)}{0}{0}{}
\end{tikzpicture}
\\\hline
\end{array}
\displaybreak[1]\\
\text{maps from $X\otimes Y \to \rho\mu\rho$: }&
\begin{array}{|c|c|c|c|}
\hline
\otimes & 1 & \rho & \rho\mu\rho
\\\hline
1 & 0 & 0 &
\begin{tikzpicture}[baseline = -.1cm]
	\clip (-.45,-.45) rectangle (.45,.45);
	\nbox{thick}{(0,0)}{0}{0}{}
	\draw[thick, rho] (.1,-.4) -- (-.1,.4);
	\draw[thick, mu] (.2,-.4) -- (0,.4);
	\draw[thick, rho] (.3,-.4) -- (.1,.4);
\end{tikzpicture}
\\\hline
\rho &  
0 & 0 &
\mp\sqrt{\tau}\,
\begin{tikzpicture}[baseline = -.1cm]
	\clip (-.45,-.45) rectangle (.45,.45);
	\coordinate (a) at (-.1,0);
	\filldraw[rho] (a) circle (.05cm);
	\draw[thick, rho] (a) -- (-.3,-.4);
	\draw[thick, rho] (a) -- (-.1,.4);
	\draw[thick, rho] (a) -- (.1,-.4);
	\nbox{}{(0,0)}{0}{0}{}
	\draw[thick, mu] (.2,-.4) -- (0,.4);
	\draw[thick, rho] (.3,-.4) -- (.1,.4);
\end{tikzpicture}
\\\hline
\rho\mu\rho &  
\begin{tikzpicture}[baseline = -.1cm]
	\clip (-.45,-.45) rectangle (.45,.45);
	\nbox{thick}{(0,0)}{0}{0}{}
	\draw[thick, rho] (-.3,-.4) -- (-.1,.4);
	\draw[thick, mu] (-.2,-.4) -- (0,.4);
	\draw[thick, rho] (-.1,-.4) -- (.1,.4);
\end{tikzpicture}
& 
\mp\sqrt{\tau}\,
\begin{tikzpicture}[baseline = -.1cm]
	\clip (-.45,-.45) rectangle (.45,.45);
	\coordinate (a) at (.1,0);
	\filldraw[rho] (a) circle (.05cm);
	\draw[thick, rho] (a) -- (.3,-.4);
	\draw[thick, rho] (a) -- (.1,.4);
	\draw[thick, rho] (a) -- (-.1,-.4);
	\nbox{}{(0,0)}{0}{0}{}
	\draw[thick, mu] (-.2,-.4) -- (0,.4);
	\draw[thick, rho] (-.3,-.4) -- (-.1,.4);
\end{tikzpicture}
& 
\lambda\,
\begin{tikzpicture}[baseline = -.1cm]
	\clip (-.45,-.45) rectangle (.45,.45);
	\coordinate (a) at (0,0);
	\filldraw[mu] (a) circle (.05cm);
	\draw[thick, mu] (a) -- (-.2,-.4);
	\draw[thick, mu] (a) -- (0,.4);
	\draw[thick, mu] (a) -- (.2,-.4);
	\draw[thick, rho] (-.3,-.4) -- (-.1,.4);
	\draw[thick, rho] (-.1,-.4) arc (180:0:.1cm);
	\draw[thick, rho] (.3,-.4) -- (.1,.4);
	\nbox{}{(0,0)}{0}{0}{}
\end{tikzpicture}
\\\hline
\end{array}
\end{align*}
We analyze maps $\rho\mu\rho\otimes \rho\mu\rho\otimes\rho\mu\rho\to\rho\mu\rho$ to get the restriction
$$
\lambda^2\,
\begin{tikzpicture}[baseline = -.1cm]
	\clip (-.45,-.45) rectangle (.85,.45);
	\nbox{thick}{(0,0)}{0}{.4}{}
	\draw[thick, rho] (-.3,-.4) -- (-.1,.4);
	\draw[thick, mu] (-.2,-.4) -- (0,.4);
	\draw[thick, rho] (.7,-.4) -- (.5,.4);
	\draw[thick, rho] (-.1,-.4) arc (180:0:.1cm);
	\draw[thick, mu] (.2,-.4) arc (180:0:.2cm);
	\draw[thick, rho] (.3,-.4) arc (180:0:.1cm);
\end{tikzpicture}
+\left(\frac{-\lambda^2}{\tau}+\tau\right)
\begin{tikzpicture}[baseline = -.1cm,xscale=-1]
	\clip (-.45,-.45) rectangle (.85,.45);
	\nbox{thick}{(0,0)}{0}{.4}{}
	\draw[thick, rho] (-.3,-.4) -- (-.1,.4);
	\draw[thick, mu] (-.2,-.4) -- (0,.4);
	\draw[thick, rho] (.7,-.4) -- (.5,.4);
	\draw[thick, rho] (-.1,-.4) arc (180:0:.1cm);
	\draw[thick, mu] (.2,-.4) arc (180:0:.2cm);
	\draw[thick, rho] (.3,-.4) arc (180:0:.1cm);
\end{tikzpicture}
=
\lambda^2\,
\begin{tikzpicture}[baseline = -.1cm,xscale=-1]
	\clip (-.45,-.45) rectangle (.85,.45);
	\nbox{thick}{(0,0)}{0}{.4}{}
	\draw[thick, rho] (-.3,-.4) -- (-.1,.4);
	\draw[thick, mu] (-.2,-.4) -- (0,.4);
	\draw[thick, rho] (.7,-.4) -- (.5,.4);
	\draw[thick, rho] (-.1,-.4) arc (180:0:.1cm);
	\draw[thick, mu] (.2,-.4) arc (180:0:.2cm);
	\draw[thick, rho] (.3,-.4) arc (180:0:.1cm);
\end{tikzpicture}
+\left(\frac{-\lambda^2}{\tau}+\tau\right)
\begin{tikzpicture}[baseline = -.1cm]
	\clip (-.45,-.45) rectangle (.85,.45);
	\nbox{thick}{(0,0)}{0}{.4}{}
	\draw[thick, rho] (-.3,-.4) -- (-.1,.4);
	\draw[thick, mu] (-.2,-.4) -- (0,.4);
	\draw[thick, rho] (.7,-.4) -- (.5,.4);
	\draw[thick, rho] (-.1,-.4) arc (180:0:.1cm);
	\draw[thick, mu] (.2,-.4) arc (180:0:.2cm);
	\draw[thick, rho] (.3,-.4) arc (180:0:.1cm);
\end{tikzpicture}\,,
$$
i.e., $\lambda^2=1$.
\end{proof}

\begin{cor}
There is a unique $A_4-A_4$ composite subfactor for $n=2,3$. For $n=4,\dots, 10$ there is no such composite subfactor.
\end{cor}
\begin{proof}
By Theorem \ref{thm:SubfactorIffTT}, uniqueness and nonexistence follow from Theorems \ref{thm:A4A4Unique} and \ref{thm:A4A4Nonexistence} respectively. Existence for $n=2,3$ is proved below.
\end{proof}

Liu's method also applies to the $A_4-A_4$ composite subfactors, and he shows that no subfactor with even half $\TT_{n,\omega_U}$ exists for any $n\geq 4$ \cite{LiuFish}. His result together with Theorem \ref{thm:SubfactorIffTT} shows that $\TT_{n,\omega_U}$ does not exist for any $n\geq 4$.

\paragraph{Existence of $\TT_{2,1}$ and $\TT_{3,1}$.\\}
Clearly $\TT_{2,1}=T_2\boxtimes T_2$ exists. We can construct $\TT_{3,1}$ from the 2D2 subfactor with principal graphs
$$
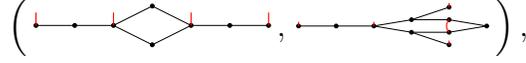
\left(
\bigraph{bwd1v1v1p1v1x1v1v1duals1v1v1v1}
,
\bigraph{bwd1v1v1p1v1x0p1x0p0x1p0x1v0x1x1x0duals1v1v1x3x2x4}
\right),
$$
which is constructed in unpublished work of Izumi, and also in \cite{4442equi}. First, naming the even bimodules on the dual graph $1,\jw{2},\rho,\sigma,\overline{\sigma},\mu$ lexicographically from left to right, bottom to top, the fusion rules are

{\scriptsize{
$$
\begin{array}{|c|c|c|c|c|c|}
\hline 
\otimes & \jw{2} & \rho & \sigma & \overline{\sigma} & \mu \\
\hline \jw{2} & 1 {+} 2 \jw{2} {+} \sigma {+} \overline{\sigma} {+} \rho {+} \mu & \jw{2} {+} \overline{\sigma} & \jw{2} {+} \sigma {+} \overline{\sigma} {+} \mu & \jw{2} {+} \sigma {+} \overline{\sigma} {+} \rho & \jw{2} {+} \sigma \\
\hline \rho & \jw{2} {+} \sigma & 1 {+} \rho & \jw{2} & \overline{\sigma} {+} \mu & \overline{\sigma} \\
\hline \sigma & \jw{2} {+} \sigma {+} \overline{\sigma} {+} \rho & \sigma {+} \mu & \jw{2} {+} \overline{\sigma} & 1 {+} \jw{2} {+} \mu & \jw{2} \\
\hline \overline{\sigma} & \jw{2} {+} \sigma {+} \overline{\sigma} {+} \mu & \jw{2} & 1 {+} \jw{2} {+} \rho & \jw{2} {+} \sigma & \overline{\sigma} {+} \rho \\
\hline \mu & \jw{2} {+} \overline{\sigma} & \sigma & \sigma {+} \rho & \jw{2} & 1 {+} \mu \\
\hline
\end{array}
$$
}}

We see that $\rho$ and $\mu$ give two copies of $A_4$, and they satisfy the relation $\rho\mu\rho\cong \jw{2}\cong \mu\rho\mu$, but $\rho\mu\ncong \mu\rho$.  Hence $\TT_{3,1}$ exists.

\paragraph{$A_4-A_4$ composite principal graphs.\\}
It is possible to determine the principal graph as in Theorem \ref{thm:ATToFish}.
For $n=2,3$, the fusion graphs for the $N-N$ bimodules with respect to $\rho\oplus \rho\mu\rho$ are given by
$$
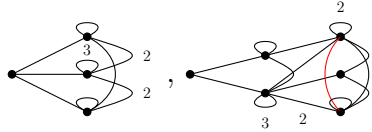
\begin{tikzpicture}[scale =1,baseline=-.1cm]
\draw[fill] (0,0) circle (0.05);
\draw (0.,0.) -- (1.,-0.5);
\draw (0.,0.) -- (1.,0.);
\draw (0.,0.) -- (1.,0.5);
\draw[fill] (1.,-0.5) circle (0.05);
\draw[fill] (1.,0.) circle (0.05);
\draw[fill] (1.,0.5) circle (0.05);
\draw (1.,-0.5) .. controls (0.5,-0.25) and (1.5,-0.25) .. (1.,-0.5);
\draw (1.,0.) .. controls (1.8,-0.25) and (1.8,-0.25) .. node[right] {\scalebox{.5}{$2$}} (1.,-0.5);
\draw (1.,0.5) .. controls (1.5,0.25) and (1.5,-0.25) .. (1.,-0.5);
\draw (1.,0.) .. controls (0.5,0.25) and (1.5,0.25) .. node[above=-2pt] {\scalebox{.5}{$3$}} (1.,0.);
\draw (1.,0.5) .. controls (1.8,0.25) and (1.8,0.25) .. node[right] {\scalebox{.5}{$2$}} (1.,0.);
\draw (1.,0.5) .. controls (0.5,0.75) and (1.5,0.75) .. (1.,0.5);
\end{tikzpicture}
\,,\,
\begin{tikzpicture}[scale =1,baseline=-.1cm]
\draw[fill] (0,0) circle (0.05);
\draw (0.,0.) -- (1.,-0.25);
\draw (0.,0.) -- (1.,0.25);
\draw[fill] (1.,-0.25) circle (0.05);
\draw[fill] (1.,0.25) circle (0.05);
\draw (1.,-0.25) -- node[below] {\scalebox{.5}{$2$}} (2.,-0.5);
\draw (1.,-0.25) -- (2.,0.);
\draw (1.,-0.25) -- (2.,0.5);
\draw (1.,0.25) -- (2.,0.5);
\draw[fill] (2.,-0.5) circle (0.05);
\draw[fill] (2.,0.) circle (0.05);
\draw[fill] (2.,0.5) circle (0.05);
\draw[red] (2.,-0.5) to[out=135,in=-135] (2.,0.5);
\draw (1.,-0.25) .. controls (0.5,-0.5) and (1.5,-0.5) .. node[below] {\scalebox{.5}{$3$}} (1.,-0.25);
\draw (1.,0.25) .. controls (1.5,0.) and (1.5,0.) .. (1.,-0.25);
\draw (1.,0.25) .. controls (0.5,0.5) and (1.5,0.5) .. (1.,0.25);
\draw (2.,-0.5) .. controls (1.5,-0.25) and (2.5,-0.25) .. (2.,-0.5);
\draw (2.,0.) .. controls (2.5,-0.25) and (2.5,-0.25) .. (2.,-0.5);
\draw (2.,0.5) .. controls (2.5,0.25) and (2.5,-0.25) .. (2.,-0.5);
\draw (2.,0.5) .. controls (2.5,0.25) and (2.5,0.25) .. (2.,0.);
\draw (2.,0.5) .. controls (1.5,0.75) and (2.5,0.75) .. node[above] {\scalebox{.5}{$2$}} (2.,0.5);
\end{tikzpicture}
$$
resulting in the following principal graphs for $n=2,3$ respectively:
$$
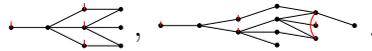
\bigraph{bwd1v1p1p1v1x1x0p0x1x0p0x1x1duals1v1x2x3}\,,\,\bigraph{bwd1v1p1v1x0p1x0p1x0p0x1v0x1x1x0p0x0x1x0p0x0x1x1v0x0x1duals1v1x2v3x2x1}\,.
$$
The first is the tensor product $A_4\otimes A_4$.

Liu pointed out to us that these two $A_4-A_4$ composite subfactors are also the reduced subfactors of $\fish_2$ and $\fish_3$ at $\rho\theta\rho$. 

\newpage
\appendix

\section{Tables for $A_2$ with $T_2$}\label{sec:FiguresForAT}

\begin{figure}[!htb]
$$
\begin{array}{|c|c|c|}
\hline
\text{Diagram} & \rho\zeta(U) & U(\zeta\rho(U))U^*
\\\hline
\begin{tikzpicture}[baseline = -.3cm]
	\draw[thick, theta] (.3,0)--(.3,-.8);
	\draw[thick, rho] (.1,0)--(.1,-.8);
	\draw[thick, theta] (-.1,0)--(-.1,-.8);
	\draw[thick, rho] (-.3,0)--(-.3,-.8);
	\Ucircle{U}{(0,0)}
	\draw[thick, theta] (.6,-.8) arc (180:0:.2cm);
	\draw[thick, rho] (.45,-.8) arc (180:0:.35cm);
\end{tikzpicture}
&
\displaystyle\frac{\omega_U^{-1}}{\tau}
&
\displaystyle\frac{\omega_U^{1}}{\tau}
\\[4pt]\hline
\begin{tikzpicture}[baseline = -.3cm]
	\coordinate (a) at (0,-.6);
	\coordinate (b) at (1.05,-.4);
	\coordinate (c) at (-.6,0);
	\coordinate (d) at (.6,0);
	\draw[thick, rho] (a)--(c);
	\draw[thick, rho] (a)--(d);
	\draw[thick, rho] (a)--(0,-.8);
	\draw[thick, rho] (b)--(d);
	\draw[thick, rho] (b)--(.7,-.8);
	\draw[thick, rho] (b)--(1.4,-.8);
	\draw[thick, theta] (c) -- (d);
	\draw[thick, theta] (.9,-.8) arc (180:0:.15cm);	
	\draw[thick, rho] ($(c)+(-.3,0)$)--($(c)+(-.3,-.8)$);
	\draw[thick, theta] ($(c)+(-.1,0)$)--($(c)+(-.1,-.8)$);
	\draw[thick, theta] ($(d)+(-.1,0)$)--($(d)+(-.1,-.8)$);
	\filldraw[rho] (a) circle (.05cm);
	\filldraw[rho] (b) circle (.05cm);
	\Ucircle{U}{(c)}
	\Ucircle{U^*}{(d)}
\end{tikzpicture}
&
\sigma_U
&
\sigma_U^{-1}
\\[4pt]\hline
\end{array}
$$
\caption{Coefficients for $n=1$}
\label{fig:n=1}
\end{figure}

\begin{figure}[!htb]
$$
\begin{array}{|c|c|c|}
\hline
\text{Diagram} & \rho\zeta(U) & U(\zeta\rho(U))U^*
\\\hline
\begin{tikzpicture}[baseline = -.3cm]
	\draw[thick, zeta] (.3,0)--(.3,-.8);
	\draw[thick, rho] (.1,0)--(.1,-.8);
	\draw[thick, zeta] (-.1,0)--(-.1,-.8);
	\draw[thick, rho] (-.3,0)--(-.3,-.8);
	\Ucircle{U}{(0,0)}
	\draw[thick, zeta] (.6,-.8) arc (180:0:.2cm);
	\draw[thick, rho] (.45,-.8) arc (180:0:.35cm);
\end{tikzpicture}
&
\displaystyle\frac{\omega_U^{-2}}{\tau^2}
&
\displaystyle\frac{\omega_U^{2}}{\tau^2}
\\[4pt]\hline
\begin{tikzpicture}[baseline = -.3cm]
	\coordinate (a) at (0,-.6);
	\coordinate (b) at (1.05,-.4);
	\coordinate (c) at (-.6,0);
	\coordinate (d) at (.6,0);
	\draw[thick, rho] (a)--(c);
	\draw[thick, rho] (a)--(d);
	\draw[thick, rho] (a)--(0,-.8);
	\draw[thick, rho] (b)--(d);
	\draw[thick, rho] (b)--(.7,-.8);
	\draw[thick, rho] (b)--(1.4,-.8);
	\draw[thick, zeta] (c) -- (d);
	\draw[thick, zeta] (.9,-.8) arc (180:0:.15cm);	
	\draw[thick, rho] ($(c)+(-.3,0)$)--($(c)+(-.3,-.8)$);
	\draw[thick, zeta] ($(c)+(-.1,0)$)--($(c)+(-.1,-.8)$);
	\draw[thick, zeta] ($(d)+(-.1,0)$)--($(d)+(-.1,-.8)$);
	\filldraw[rho] (a) circle (.05cm);
	\filldraw[rho] (b) circle (.05cm);
	\Ucircle{U}{(c)}
	\Ucircle{U^*}{(d)}
\end{tikzpicture}
&
\displaystyle\frac{\sigma_U^{-1}}{\tau}
&
\displaystyle\frac{\sigma_U^{-1}}{\tau}
\\
\hline
\end{array}
$$
\caption{Coefficients for $n=2$}
\label{fig:n=2}
\end{figure}

\begin{figure}[!htb]
$$
\begin{array}{|c|c|c|}
\hline
\text{Diagram} & \rho\zeta(U) & U(\zeta\rho(U))U^*
\\\hline
\begin{tikzpicture}[baseline = -.3cm]
	\draw[thick, zeta] (.3,0)--(.3,-.8);
	\draw[thick, rho] (.1,0)--(.1,-.8);
	\draw[thick, zeta] (-.1,0)--(-.1,-.8);
	\draw[thick, rho] (-.3,0)--(-.3,-.8);
	\Ucircle{U}{(0,0)}
	\draw[thick, zeta] (.6,-.8) arc (180:0:.2cm);
	\draw[thick, rho] (.45,-.8) arc (180:0:.35cm);
\end{tikzpicture}
&
\displaystyle\frac{\omega_U^{-3}}{\tau^3}
&
\displaystyle\frac{\omega_U^{3}}{\tau^3}
\\[4pt]\hline
\begin{tikzpicture}[baseline = -.3cm]
	\coordinate (a) at (0,-.6);
	\coordinate (b) at (1.05,-.4);
	\coordinate (c) at (-.6,0);
	\coordinate (d) at (.6,0);
	\draw[thick, rho] (a)--(c);
	\draw[thick, rho] (a)--(d);
	\draw[thick, rho] (a)--(0,-.8);
	\draw[thick, rho] (b)--(d);
	\draw[thick, rho] (b)--(.7,-.8);
	\draw[thick, rho] (b)--(1.4,-.8);
	\draw[thick, zeta] (c) -- (d);
	\draw[thick, zeta] (.9,-.8) arc (180:0:.15cm);	
	\draw[thick, rho] ($(c)+(-.3,0)$)--($(c)+(-.3,-.8)$);
	\draw[thick, zeta] ($(c)+(-.1,0)$)--($(c)+(-.1,-.8)$);
	\draw[thick, zeta] ($(d)+(-.1,0)$)--($(d)+(-.1,-.8)$);
	\filldraw[rho] (a) circle (.05cm);
	\filldraw[rho] (b) circle (.05cm);
	\Ucircle{U}{(c)}
	\Ucircle{U^*}{(d)}
\end{tikzpicture}
&
\displaystyle\frac{\sigma_U^{-3}}{\tau^2}
&
\displaystyle\frac{\sigma_U^{-1}}{\tau^2}
\\
\hline
\end{array}
$$
\caption{Coefficients for $n=3$}
\label{fig:n=3}
\end{figure}

\begin{figure}[!htb]
$$
\begin{array}{|c|c|c|}
\hline\text{Train} & \rho\zeta(U) & U(\zeta\rho(U))U^*
\\\hline
\train[\{\},\wheel[8]] &\displaystyle \frac{1}{\tau ^4 \sigma _U^8} &\displaystyle \frac{\sigma _U^8}{\tau ^4}
\\\hline
\train[8,7,0,\wheel[7]] &\displaystyle \frac{1}{\tau ^3 \sigma _U^5} &\displaystyle \frac{1}{\tau ^3 \sigma _U}
\\\hline
\train[10,5,2,\wheel[5]] &\displaystyle \frac{1}{\tau ^3 \sigma _U^3} &\displaystyle \frac{\sigma _U}{\tau ^3}
\\\hline
\train[12,3,4,\wheel[3]] &\displaystyle \frac{1}{\tau ^3 \sigma _U} &\displaystyle \frac{\sigma _U^3}{\tau ^3}
\\\hline
\train[14,1,6,\wheel[1]] &\displaystyle \frac{\sigma _U}{\tau ^3} &\displaystyle \frac{\sigma _U^5}{\tau ^3}
\\\hline
\train[8,7,2,7,6,\wheel[1]] &\displaystyle \frac{\sigma _U}{\tau ^2} &\displaystyle \frac{1}{\tau ^2 \sigma _U}
\\\hline
\train[8,7,4,5,6,\wheel[1]] &\displaystyle \frac{\sigma _U}{\tau ^2} &\displaystyle \frac{1}{\tau ^2 \sigma _U}
\\
\hline
\end{array}
$$
\caption{Coefficients of $\rho\zeta(U)$ and $U(\zeta\rho(U))U^*$ in the train basis for $n=4$, computed in the {\texttt{Mathematica}} notebook {\texttt{JellyfishCalculations}}, bundled with the {\texttt{arXiv}} source of this article.}
\label{fig:n=4}
\end{figure}

\clearpage

\section{Tables for $T_2$ with $T_2$}\label{sec:FiguresForTT}

\begin{figure}[!htb]
$$
\begin{array}{|c|c|c|}
\hline
\text{Diagram} & \rho\mu(U) & U(\mu\rho(U))U^*
\\\hline
\begin{tikzpicture}[baseline = -.3cm]
	\draw[thick, mu] (.3,0)--(.3,-.8);
	\draw[thick, rho] (.1,0)--(.1,-.8);
	\draw[thick, mu] (-.1,0)--(-.1,-.8);
	\draw[thick, rho] (-.3,0)--(-.3,-.8);
	\Ucircle{U}{(0,0)}
	\draw[thick, mu] (.6,-.8) arc (180:0:.2cm);
	\draw[thick, rho] (.45,-.8) arc (180:0:.35cm);
\end{tikzpicture}
&
\displaystyle\frac{\omega_U^{-1}}{\tau^2}
&
\displaystyle\frac{\omega_U}{\tau^2}
\\[4pt]\hline
\begin{tikzpicture}[baseline = -.3cm]
	\coordinate (a) at (0,-.6);
	\coordinate (b) at (1.05,-.4);
	\coordinate (c) at (-.6,0);
	\coordinate (d) at (.6,0);
	\draw[thick, rho] (a)--(c);
	\draw[thick, rho] (a)--(d);
	\draw[thick, rho] (a)--(0,-.8);
	\draw[thick, rho] (b)--(d);
	\draw[thick, rho] (b)--(.7,-.8);
	\draw[thick, rho] (b)--(1.4,-.8);
	\draw[thick, mu] (c) -- (d);
	\draw[thick, mu] (.9,-.8) arc (180:0:.15cm);	
	\draw[thick, rho] ($(c)+(-.3,0)$)--($(c)+(-.3,-.8)$);
	\draw[thick, mu] ($(c)+(-.1,0)$)--($(c)+(-.1,-.8)$);
	\draw[thick, mu] ($(d)+(-.1,0)$)--($(d)+(-.1,-.8)$);
	\filldraw[rho] (a) circle (.05cm);
	\filldraw[rho] (b) circle (.05cm);
	\Ucircle{U}{(c)}
	\Ucircle{U^*}{(d)}
\end{tikzpicture}
&
\displaystyle\frac{\sigma_U}{\tau}
&
\displaystyle\frac{\sigma_U^{-1}}{\tau}
\\
\hline
\end{array}
$$
\caption{Coefficients for $n=2$}
\label{fig:TTn=2}
\end{figure}

\begin{figure}[!htb]
$$
\begin{array}{|c|c|c|}
\hline
\text{Diagram} & \rho\mu\rho(U) & U(\mu\rho\mu(U))U^*
\\\hline
\begin{tikzpicture}[baseline = -.3cm]
	\draw[thick, rho] (.25,0)--(.25,-.8);
	\draw[thick, mu] (.15,0)--(.15,-.8);
	\draw[thick, rho] (.05,0)--(.05,-.8);
	\draw[thick, mu] (-.05,0)--(-.05,-.8);
	\draw[thick, rho] (-.15,0)--(-.15,-.8);
	\draw[thick, mu] (-.25,0)--(-.25,-.8);
	\Ucircle{U^*}{(0,0)}
	\draw[thick, mu] (.65,-.8) arc (180:0:.15cm);
	\draw[thick, rho] (.55,-.8) arc (180:0:.25cm);
	\draw[thick, mu] (.45,-.8) arc (180:0:.35cm);
\end{tikzpicture}
&
\displaystyle\frac{\omega_U^{-1}}{\tau^3}
&
\displaystyle\frac{\omega_U^{2}}{\tau^3}
\\[4pt]\hline
\begin{tikzpicture}[baseline = -.3cm]
	\coordinate (a) at (0,-.6);
	\coordinate (b) at (1.05,-.4);
	\coordinate (c) at (-.6,0);
	\coordinate (d) at (.6,0);
	\draw[thick, rho] (a)--(c);
	\draw[thick, rho] (a)--(d);
	\draw[thick, rho] (a)--(0,-.8);
	\draw[thick, mu] (b)--(d);
	\draw[thick, mu] (b)--(.7,-.8);
	\draw[thick, mu] (b)--(1.4,-.8);
	\draw[thick, rho] ($(c)+(0,0)$)--($(d)+(0,0)$);
	\draw[thick, mu] ($(c)+(0,-.1)$)--($(d)+(0,-.1)$);
	\draw[thick, mu] (.95,-.8) arc (180:0:.1cm);	
	\draw[thick, rho] (.85,-.8) arc (180:0:.2cm);
	\draw[thick, mu] ($(c)+(-.25,0)$)--($(c)+(-.25,-.8)$);
	\draw[thick, rho] ($(c)+(-.1,0)$)--($(c)+(-.1,-.8)$);
	\draw[thick, mu] ($(c)+(.05,0)$)--($(c)+(.05,-.8)$);
	\draw[thick, mu] ($(d)+(-.15,0)$)--($(d)+(-.15,-.8)$);
	\draw[thick, rho] ($(d)+(0,0)$)--($(d)+(0,-.8)$);
	\filldraw[rho] (a) circle (.05cm);
	\filldraw[mu] (b) circle (.05cm);
	\Ucircle{U^*}{(c)}
	\Ucircle{U}{(d)}
\end{tikzpicture}
&
\displaystyle\frac{\omega_U^{-1}}{\tau^2}
&
\displaystyle\frac{\sigma_U^{-1}}{\tau^2}
\\
\hline
\end{array}
$$
\caption{Coefficients for $n=3$}
\label{fig:TTn=3}
\end{figure}

\begin{figure}[!htb]
$$
\begin{array}{|c|c|c|}
\hline \text{Train} & \rho\mu\rho\mu(U) & U(\mu\rho\mu\rho(U))U^*
\\\hline
\train[\{\},\wheel[4]] & \displaystyle \frac{1}{\tau ^4 \sigma _U^4} & \displaystyle \frac{\sigma _U^4}{\tau ^4}
\\\hline
\train[4,3,0,\wheel[3]] & \displaystyle \frac{1}{\tau ^3 \sigma _U} & \displaystyle \frac{1}{\tau ^3 \sigma _U}
\\\hline
\train[5,2,1,\wheel[2]] & \displaystyle \frac{1}{\tau ^3 \sigma _U^2} & \displaystyle \frac{\sigma _U^2}{\tau ^3}
\\\hline
\train[6,1,2,\wheel[1]] & \displaystyle \frac{\sigma _U}{\tau ^3} & \displaystyle \frac{\sigma _U}{\tau ^3}
\\\hline
\train[7,0,3,\wheel[0]] & \displaystyle \frac{1}{\tau ^3} & \displaystyle \frac{\sigma _U^4}{\tau ^3}
\\\hline
\train[4,3,1,3,3,\wheel[0]] & \displaystyle \frac{1}{\tau ^2} & \displaystyle \frac{1}{\tau ^2 \sigma _U}
\\\hline
\train[4,3,2,2,3,\wheel[0]] & \displaystyle \frac{1}{\tau ^2} & \displaystyle \frac{1}{\tau ^2 \sigma _U}
\\
\hline
\end{array}
$$
\caption{Coefficients of $\rho\mu\rho\mu(U)$ and $U(\mu\rho\mu\rho(U))U^*$ in the train basis for $n=4$, computed in the {\texttt{Mathematica}} notebook {\texttt{JellyfishCalculations}}, bundled with the {\texttt{arXiv}} source of this article.}
\label{fig:TTn=4}
\end{figure}

\clearpage

\bibliographystyle{alpha}
\bibliography{../../bibliography/bibliography}

\newcommand{\noopsort}[1]{}\def\cprime{$'$} \def\cprime{$'$}
\begin{thebibliography}{BPMS12}

\bibitem[BH96]{MR1386923}
Dietmar Bisch and Uffe Haagerup.
\newblock Composition of subfactors: new examples of infinite depth subfactors.
\newblock {\em Ann. Sci. \'Ecole Norm. Sup. (4)}, 29(3):329--383, 1996.
\newblock \mathscinet{MR1386923} \numdam{ASENS\_1996\_4\_29\_3\_329\_0}.

\bibitem[Bis94]{MR1262294}
Dietmar Bisch.
\newblock A note on intermediate subfactors.
\newblock {\em Pacific J. Math.}, 163(2):201--216, 1994.
\newblock \mathscinet{MR1262294} \euclid{euclid.pjm/1102622455}.

\bibitem[BJ97]{MR1437496}
Dietmar Bisch and Vaughan F.~R. Jones.
\newblock Algebras associated to intermediate subfactors.
\newblock {\em Invent. Math.}, 128(1):89--157, 1997.
\newblock \mathscinet{MR1437496} \doi{10.1007/s002220050137}.

\bibitem[BPMS12]{MR2979509}
Stephen Bigelow, Emily Peters, Scott Morrison, and Noah Snyder.
\newblock Constructing the extended {H}aagerup planar algebra.
\newblock {\em Acta Math.}, 209(1):29--82, 2012.
\newblock \mathscinet{MR2979509} \arxiv{0909.4099}
  \doi{10.1007/s11511-012-0081-7}.

\bibitem[Gol59]{MR0107827}
Malcolm Goldman.
\newblock On subfactors of factors of type {${\rm II}_{1}$}.
\newblock {\em Michigan Math. J.}, 6:167--172, 1959.
\newblock \mathscinet{MR0107827}.

\bibitem[GS12]{MR2909758}
Pinhas Grossman and Noah Snyder.
\newblock {Quantum subgroups of the Haagerup fusion categories}.
\newblock {\em Comm. Math. Phys.}, 311(3):617--643, 2012.
\newblock \arxiv{1102.2631} \mathscinet{MR2909758}
  \doi{10.1007/s00220-012-1427-x}.

\bibitem[Haa94]{MR1317352}
Uffe Haagerup.
\newblock Principal graphs of subfactors in the index range
  {$4<[M:N]<3+\sqrt2$}.
\newblock In {\em Subfactors ({K}yuzeso, 1993)}, pages 1--38. World Sci. Publ.,
  River Edge, NJ, 1994.
\newblock \mathscinet{MR1317352}.

\bibitem[IK93]{MR1213139}
Masaki Izumi and Yasuyuki Kawahigashi.
\newblock Classification of subfactors with the principal graph {$D\sp {(1)}\sb
  n$}.
\newblock {\em J. Funct. Anal.}, 112(2):257--286, 1993.
\newblock \mathscinet{MR1213139} \doi{10.1006/jfan.1993.1033}.

\bibitem[Izu93]{MR1269266}
Masaki Izumi.
\newblock On type {${\rm II}$} and type {${\rm III}$} principal graphs of
  subfactors.
\newblock {\em Math. Scand.}, 73(2):307--319, 1993.
\newblock \mathscinet{MR1269266}.

\bibitem[Jon83]{MR0696688}
Vaughan F.~R. Jones.
\newblock Index for subfactors.
\newblock {\em Invent. Math.}, 72(1):1--25, 1983.
\newblock \mathscinet{MR696688} \doi{10.1007/BF01389127}.

\bibitem[Lan02]{MR1950890}
Zeph~A. Landau.
\newblock Exchange relation planar algebras.
\newblock In {\em Proceedings of the {C}onference on {G}eometric and
  {C}ombinatorial {G}roup {T}heory, {P}art {II} ({H}aifa, 2000)}, volume~95,
  pages 183--214, 2002.

\bibitem[Lon94]{MR1257245}
Roberto Longo.
\newblock A duality for {H}opf algebras and for subfactors. {I}.
\newblock {\em Comm. Math. Phys.}, 159(1):133--150, 1994.
\newblock \mathscinet{MR1257245}.

\bibitem[LR97]{MR1444286}
R.~Longo and J.~E. Roberts.
\newblock A theory of dimension.
\newblock {\em $K$-Theory}, 11(2):103--159, 1997.
\newblock \mathscinet{MR1444286}.

\bibitem[MP13]{MPAffineAandD}
Scott Morrison and David Penneys.
\newblock The affine {A} and {D} planar algebras, 2013.
\newblock In preparation.

\bibitem[MPP13]{4442equi}
Scott Morrison, David Penneys, and Emily Peters.
\newblock Equivariantizations and 3333 spoke subfactors at index $3+\sqrt{5}$,
  2013.
\newblock In preparation.

\bibitem[MPPS12]{MR2902285}
Scott Morrison, David Penneys, Emily Peters, and Noah Snyder.
\newblock Subfactors of index less than 5, {P}art 2: {T}riple points.
\newblock {\em Internat. J. Math.}, 23(3):1250016, 33, 2012.
\newblock \mathscinet{MR2902285} \arxiv{1007.2240}
  \doi{10.1142/S0129167X11007586}.

\bibitem[M{\"u}g03]{MR1966524}
Michael M{\"u}ger.
\newblock From subfactors to categories and topology. {I}. {F}robenius algebras
  in and {M}orita equivalence of tensor categories.
\newblock {\em J. Pure Appl. Algebra}, 180(1-2):81--157, 2003.
\newblock \mathscinet{MR1966524} \doi{10.1016/S0022-4049(02)00247-5}
  \arxiv{math.CT/0111204}.

\bibitem[Ost03]{MR1976459}
Victor Ostrik.
\newblock Module categories, weak {H}opf algebras and modular invariants.
\newblock {\em Transform. Groups}, 8(2):177--206, 2003.
\newblock \mathscinet{MR1976459} \arxiv{0111139}.

\bibitem[Pop90]{MR1065437}
Sorin Popa.
\newblock Sur la classification des sous-facteurs d'indice fini du facteur
  hyperfini.
\newblock {\em C. R. Acad. Sci. Paris S\'er. I Math.}, 311(2):95--100, 1990.
\newblock \mathscinet{MR1065437}.

\bibitem[Pop94]{MR1278111}
Sorin Popa.
\newblock Classification of amenable subfactors of type {II}.
\newblock {\em Acta Math.}, 172(2):163--255, 1994.
\newblock \mathscinet{MR1278111} \doi{10.1007/BF02392646}.

\bibitem[Pop95a]{MR1339767}
Sorin Popa.
\newblock {\em Classification of subfactors and their endomorphisms}, volume~86
  of {\em CBMS Regional Conference Series in Mathematics}.
\newblock Published for the Conference Board of the Mathematical Sciences,
  Washington, DC, 1995.
\newblock \mathscinet{MR1339767}.

\bibitem[Pop95b]{MR1372533}
Sorin Popa.
\newblock Free-independent sequences in type {${\rm II}_1$} factors and related
  problems.
\newblock {\em Ast\'erisque}, (232):187--202, 1995.
\newblock \mathscinet{MR1372533}, Recent advances in operator algebras
  (Orl{\'e}ans, 1992).

\bibitem[PP88]{MR965748}
Mihai Pimsner and Sorin Popa.
\newblock Iterating the basic construction.
\newblock {\em Trans. Amer. Math. Soc.}, 310(1):127--133, 1988.
\newblock \mathscinet{MR965748} \doi{10.2307/2001113}.

\bibitem[Vae09]{MR2471930}
Stefaan Vaes.
\newblock Factors of type {II{$_1$}} without non-trivial finite index
  subfactors.
\newblock {\em Trans. Amer. Math. Soc.}, 361(5):2587--2606, 2009.
\newblock \arxiv{math/0610231} \mathscinet{MR2471930}
  \doi{0.1090/S0002-9947-08-04585-6}.

\bibitem[Yam03]{MR1960417}
Shigeru Yamagami.
\newblock {$C^\ast$}-tensor categories and free product bimodules.
\newblock {\em J. Funct. Anal.}, 197(2):323--346, 2003.
\newblock \mathscinet{MR1960417} \doi{10.1016/S0022-1236(02)00036-8}.

\bibitem[Yam04]{MR2091457}
Shigeru Yamagami.
\newblock Frobenius duality in {$C^*$}-tensor categories.
\newblock {\em J. Operator Theory}, 52(1):3--20, 2004.
\newblock \mathscinet{MR2091457}.

\end{thebibliography}

\end{document}